\documentclass{amsart}
\usepackage[english]{babel}
\usepackage{csquotes}
\usepackage[letterpaper,top=2cm,bottom=2cm,left=3cm,right=3cm,marginparwidth=1.75cm]{geometry}
\usepackage{tikz}
\usepackage{quiver}
\usepackage{xparse}
\usepackage{enumitem}
\usepackage{anindex}
\anindexsetup{
   heading,
   lines,
     title = {Index of Notation}
}
\usepackage[maxbibnames=99,backend=biber,url=false, style=alphabetic, doi=true,eprint=false]{biblatex}
\renewbibmacro*{url+urldate}{}
\AtEveryBibitem{
  \clearfield{isbn}
  \clearlist{location}
  \clearfield{month}
  \clearfield{day}
  \clearfield{series}
  \clearfield{pages}
  \clearlist{language}
  \clearlist{translator}
  \clearfield{urldate}
    \clearfield{pubstate}
    \clearfield{url}
\renewbibmacro{in:}{}
}
\usepackage{subcaption}
\usetikzlibrary{shapes}
\usetikzlibrary{backgrounds}
\usetikzlibrary{decorations,decorations.pathreplacing,decorations.markings}
\usetikzlibrary{fit,calc,through}
\usetikzlibrary{external}
\usetikzlibrary{positioning}
\tikzstyle{mid>}=[decoration={markings, mark=at position 0.5 with {\arrow{>}}}, postaction={decorate}]
\tikzstyle{mid<}=[decoration={markings, mark=at position 0.5 with {\arrow{<}}}, postaction={decorate}]
\tikzstyle{upper>}=[decoration={markings, mark=at position 0.8 with {\arrow{>}}}, postaction={decorate}]
\tikzstyle{upper<}=[decoration={markings, mark=at position 0.8 with {\arrow{<}}}, postaction={decorate}]
\tikzstyle{lower<}=[decoration={markings, mark=at position 0.25 with {\arrow{<}}}, postaction={decorate}]
\tikzstyle{lower>}=[decoration={markings, mark=at position 0.25 with {\arrow{>}}}, postaction={decorate}]
\tikzstyle{rung}=[blue,dash pattern=on 5pt off 1pt on 1pt off 1pt]
\tikzstyle{rung>}=[blue,dash pattern=on 5pt off 1pt on 1pt off 1pt, decoration={markings, mark=at position 0.5 with {\arrow{>}}}, postaction={decorate}]
\tikzstyle{rung<}=[blue,dash pattern=on 5pt off 1pt on 1pt off 1pt, decoration={markings, mark=at position 0.5 with {\arrow{<}}}, postaction={decorate}]

\usepackage{environ}
\usepackage{xargs}
\newcommandx{\NewEnvironx}[5][2,3]{\expandafter\newcommandx\csname start#1\endcsname[#2][#3]{#4}\NewEnviron{#1}{\csname start#1\expandafter\endcsname\BODY #5}}
\newcommand{\ladderX}{1.5}
\newcommand{\ladderY}{1.5}
\newcommand{\ladderR}{0.6}
\newcommand{\laddercoordinates}[2]{
\foreach \x in {0,...,#1} {
	\foreach \y in {0,...,#2} {
		\coordinate (l\x\y) at (\x * \ladderX, \y * \ladderY);
		\coordinate (u\x\y) at ($(l\x\y)+\ladderR*(0,\ladderY)$);
		\coordinate (d\x\y) at ($(l\x\y)+(0,\ladderY)-\ladderR*(0,\ladderY)$);
	}
}
}
\newcommand{\ladderEn}[5]{
\draw[mid>] (l#1#2) -- (d#1#2);
\draw[mid>] (d#1#2) -- ($(l#1#2)+(0,\ladderY)$) node[left] {#3};
\draw[mid>] ($(l#1#2)+(\ladderX,0)$) -- ($(u#1#2)+(\ladderX,0)$);
\draw[mid>] ($(u#1#2)+(\ladderX,0)$) -- ($(l#1#2)+(\ladderX,\ladderY)$) node[right] {#4};
\draw[mid>] (d#1#2) --node[above]{#5} ($(u#1#2)+(\ladderX,0)$);
}

\newcommand{\ladderEd}[7]{
\draw[#6] (l#1#2) -- (d#1#2);
\draw[#6] (d#1#2) -- ($(l#1#2)+(0,\ladderY)$) node[left] {#3};
\draw[#7] ($(l#1#2)+(\ladderX,0)$) -- ($(u#1#2)+(\ladderX,0)$);
\draw[#7] ($(u#1#2)+(\ladderX,0)$) -- ($(l#1#2)+(\ladderX,\ladderY)$) node[right] {#4};
\draw[mid>] (d#1#2) --node[above]{#5} ($(u#1#2)+(\ladderX,0)$);
}
\newcommand{\TladderEd}[7]{
\draw[#6,very thick] (l#1#2) -- (d#1#2);
\draw[#6,very thick] (d#1#2) -- ($(l#1#2)+(0,\ladderY)$) node[left] {#3};
\draw[#7,very thick] ($(l#1#2)+(\ladderX,0)$) -- ($(u#1#2)+(\ladderX,0)$);
\draw[#7,very thick] ($(u#1#2)+(\ladderX,0)$) -- ($(l#1#2)+(\ladderX,\ladderY)$) node[right] {#4};
\draw[mid>] (d#1#2) --node[above]{#5} ($(u#1#2)+(\ladderX,0)$);
}

\newcommand{\ladderFn}[5]{
\draw[mid>] (l#1#2) -- (u#1#2);
\draw[mid>] (u#1#2) -- ($(l#1#2)+(0,\ladderY)$) node[left] {#3};
\draw[mid>] ($(l#1#2)+(\ladderX,0)$) -- ($(d#1#2)+(\ladderX,0)$);
\draw[mid>] ($(d#1#2)+(\ladderX,0)$) -- ($(l#1#2)+(\ladderX,\ladderY)$) node[right] {#4};
\draw[mid>] ($(d#1#2)+(\ladderX,0)$) --node[above]{#5} (u#1#2);
}

\newcommand{\ladderFd}[7]{
\draw[#6] (l#1#2) -- (u#1#2);
\draw[#6] (u#1#2) -- ($(l#1#2)+(0,\ladderY)$) node[left] {#3};
\draw[#7] ($(l#1#2)+(\ladderX,0)$) -- ($(d#1#2)+(\ladderX,0)$);
\draw[#7] ($(d#1#2)+(\ladderX,0)$) -- ($(l#1#2)+(\ladderX,\ladderY)$) node[right] {#4};
\draw[mid>] ($(d#1#2)+(\ladderX,0)$) --node[above]{#5} (u#1#2);
}
\newcommand{\TladderFd}[7]{
\draw[#6,very thick] (l#1#2) -- (u#1#2);
\draw[#6,very thick] (u#1#2) -- ($(l#1#2)+(0,\ladderY)$) node[left] {#3};
\draw[#7,very thick] ($(l#1#2)+(\ladderX,0)$) -- ($(d#1#2)+(\ladderX,0)$);
\draw[#7,very thick] ($(d#1#2)+(\ladderX,0)$) -- ($(l#1#2)+(\ladderX,\ladderY)$) node[right] {#4};
\draw[mid>] ($(d#1#2)+(\ladderX,0)$) --node[above]{#5} (u#1#2);
}

\newcommand{\ladderIn}[3]{\draw[mid>] (l#1#2) -- +($#3*(0,\ladderY)$);}
\newcommand{\ladderI}[2]{\ladderIn{#1}{#2}{1}}
\NewEnvironx{ladder}[2]{\begin{tikzpicture}[baseline=13*\ladderY*#2]\laddercoordinates{#1}{#2}
  \end{tikzpicture}}

\def\semicolon{;}
\def\applytolist#1{
    \expandafter\def\csname multi#1\endcsname##1{
        \def\omultiack{##1}\ifx\omultiack\semicolon
            \def\next{\relax}
        \else
            \csname #1\endcsname{##1}
            \def\next{\csname multi#1\endcsname}
        \fi
        \next}
    \csname multi#1\endcsname}
\def\calc#1{\expandafter\def\csname c#1\endcsname{{\mathcal #1}}}
\applytolist{calc}QWERTYUIOPLKJHGFDSAZXCVBNM;

\newcommand{\qBinomial}[3][q]{\genfrac{[}{]}{0pt}{}{#2}{#3}_{#1}}

\newcommand{\braidmodule}[2]{\begin{scope}[shift={(0,#1)}]
    \def\spacing{2} \def\totalstrands{4} \def\braidindex{#2} \foreach \j in {1,...,\totalstrands} {\ifnum\j=\braidindex
\else
        \ifnum\j=\numexpr\braidindex+1\relax
\else
          \draw[very thick] (\j*\spacing,0) -- (\j*\spacing,2);
        \fi
      \fi
    }\draw[very thick] ({\braidindex*\spacing},2) ..
      controls ({\braidindex*\spacing},1.5) and ({(\braidindex+1)*\spacing},0.5) ..
      ({(\braidindex+1)*\spacing},0); 
\pgfmathsetmacro{\xcross}{(\braidindex+0.5)*\spacing}
    \pgfmathsetmacro{\ycross}{1}
    \fill[white] (\xcross,\ycross) circle (0.1);

      \draw[very thick] ({(\braidindex+1)*\spacing},2) ..
      controls ({(\braidindex+1)*\spacing},1.5) and ({\braidindex*\spacing},0.5) ..
      ({\braidindex*\spacing},0);
  \end{scope}}

\newcommand{\braidmoduletwo}[2]{\begin{scope}[shift={(0,#1)}]
    \def\spacing{2} \def\totalstrands{2} \def\braidindex{#2} \foreach \j in {1,...,\totalstrands} {\ifnum\j=\braidindex
\else
        \ifnum\j=\numexpr\braidindex+1\relax
\else
          \draw[very thick] (\j*\spacing,0) -- (\j*\spacing,2);
        \fi
      \fi
    }\draw[very thick] ({\braidindex*\spacing},2) ..
      controls ({\braidindex*\spacing},1.5) and ({(\braidindex+1)*\spacing},0.5) ..
      ({(\braidindex+1)*\spacing},0); 
\pgfmathsetmacro{\xcross}{(\braidindex+0.5)*\spacing}
    \pgfmathsetmacro{\ycross}{1}
    \fill[white] (\xcross,\ycross) circle (0.1);

      \draw[very thick] ({(\braidindex+1)*\spacing},2) ..
      controls ({(\braidindex+1)*\spacing},1.5) and ({\braidindex*\spacing},0.5) ..
      ({\braidindex*\spacing},0);
  \end{scope}}

\newcommand{\braidmoduletwoI}[2]{\begin{scope}[shift={(0,#1)}]
    \def\spacing{2} \def\totalstrands{2} \def\braidindex{#2} \foreach \j in {1,...,\totalstrands} {\ifnum\j=\braidindex
\else
        \ifnum\j=\numexpr\braidindex+1\relax
\else
          \draw[very thick] (\j*\spacing,0) -- (\j*\spacing,2);
        \fi
      \fi
    }\draw[very thick] ({(\braidindex+1)*\spacing},2) ..
      controls ({(\braidindex+1)*\spacing},1.5) and ({\braidindex*\spacing},0.5) ..
      ({\braidindex*\spacing},0); 
\pgfmathsetmacro{\xcross}{(\braidindex+0.5)*\spacing}
    \pgfmathsetmacro{\ycross}{1}
    \fill[white] (\xcross,\ycross) circle (0.1);
\draw[very thick] ({\braidindex*\spacing},2) ..
      controls ({\braidindex*\spacing},1.5) and ({(\braidindex+1)*\spacing},0.5) ..
      ({(\braidindex+1)*\spacing},0);
  \end{scope}}

\newcommand{\braidmoduleI}[2]{\begin{scope}[shift={(0,#1)}]
    \def\spacing{2} \def\totalstrands{4} \def\braidindex{#2} \foreach \j in {1,...,\totalstrands} {\ifnum\j=\braidindex
\else
        \ifnum\j=\numexpr\braidindex+1\relax
\else
          \draw[very thick] (\j*\spacing,0) -- (\j*\spacing,2);
        \fi
      \fi
    }\draw[very thick] ({(\braidindex+1)*\spacing},2) ..
      controls ({(\braidindex+1)*\spacing},1.5) and ({\braidindex*\spacing},0.5) ..
      ({\braidindex*\spacing},0); 
\pgfmathsetmacro{\xcross}{(\braidindex+0.5)*\spacing}
    \pgfmathsetmacro{\ycross}{1}
    \fill[white] (\xcross,\ycross) circle (0.1);
\draw[very thick] ({\braidindex*\spacing},2) ..
      controls ({\braidindex*\spacing},1.5) and ({(\braidindex+1)*\spacing},0.5) ..
      ({(\braidindex+1)*\spacing},0);
  \end{scope}}

\newcommand{\abraidmodule}[2]{\begin{scope}[shift={(0,#1)}]
    \def\spacing{2} \def\totalstrands{4} \def\braidindex{#2} \foreach \j in {1,...,\totalstrands} {\ifnum\j=\braidindex
\else
        \ifnum\j=\numexpr\braidindex+1\relax
\else
          \draw[very thick] (\j*\spacing,0) -- (\j*\spacing,2);
        \fi
      \fi
    }\draw[very thick] ({(\braidindex+1)*\spacing},2) ..
      controls ({(\braidindex+1)*\spacing},1.5) and ({\braidindex*\spacing},0.5) ..
      ({\braidindex*\spacing},0);
\pgfmathsetmacro{\xcross}{(\braidindex+0.5)*\spacing}
    \pgfmathsetmacro{\ycross}{1}
    \fill[white] (\xcross,\ycross) circle (0.1);
\draw[very thick] ({\braidindex*\spacing},2) ..
      controls ({\braidindex*\spacing},1.5) and ({(\braidindex+1)*\spacing},0.5) ..
      ({(\braidindex+1)*\spacing},0); 
  \end{scope}}

 \newcommand{\C}{\mathbb{C}}
\newcommand{\R}{\mathbb{R}}
\newcommand{\Z}{\mathbb{Z}}

\newcommand{\diag}{\operatorname{diag}}
\newcommand{\freeA}{\mathcal{F}}

\usepackage{amsmath,amsfonts,amssymb,xcolor}
\usepackage{graphicx}
\usepackage[colorlinks=true, allcolors=blue]{hyperref}
\usepackage{zref-clever}
\usepackage{aliascnt}
\zcsetup{cap}
\providecommand{\cref}[1]{\zcref{#1}}
\providecommand{\Cref}[1]{\zcref{#1}}
\providecommand{\crefrange}[2]{\zcref{#1,#2}}

\zcRefTypeSetup{equation}{
  name-sg=,
  Name-sg=,
  name-pl=,
  Name-pl=
}

\NewDocumentCommand{\NewTheoremWithZref}{ m o m m }{\IfNoValueTF{#2}{\newtheorem{#1}{#3}\zcRefTypeSetup{#1}{name-sg=#3, Name-sg=#3, name-pl=#4, Name-pl=#4}}{\newaliascnt{#1}{#2}\newtheorem{#1}[#1]{#3}\aliascntresetthe{#1}\zcRefTypeSetup{#1}{name-sg=#3, Name-sg=#3, name-pl=#4, Name-pl=#4}}}

\numberwithin{equation}{section}

\newtheorem{itheorem}{Theorem}
\zcRefTypeSetup{itheorem}{
  name-sg=theorem,
  Name-sg=Theorem,
  name-pl=theorems,
  Name-pl=Theorems
}
\newtheorem{theorem}{Theorem}[section]

\NewTheoremWithZref{corollary}[theorem]{Corollary}{Corollaries}
\NewTheoremWithZref{lemma}[theorem]{Lemma}{Lemmas}
\NewTheoremWithZref{assumption}[theorem]{Assumption}{Assumptions}
\NewTheoremWithZref{definition}[theorem]{Definition}{Definitions}
\NewTheoremWithZref{ques}[theorem]{Question}{Questions}
\NewTheoremWithZref{conjecture}[theorem]{Conjecture}{Conjectures}
\NewTheoremWithZref{algorithm}[theorem]{Algorithm}{Algorithms}

\theoremstyle{remark}

\NewTheoremWithZref{remark}[theorem]{Remark}{Remarks}
\NewTheoremWithZref{examplex}{Example}{Examples}
\newcommand{\Spid}[1]{\mathsf{Sp}(#1)}

\newcommand{\SpidM}[1]{\mathsf{Sp}_M(#1)}

\newcommand{\tSpidM}[1]{\mathsf{\widetilde{Sp}}_M(#1)}
\newcommand{\Spidm}[2]{\mathsf{Sp}_{#1}(#2)}

\newcommand{\Spl}[2]{\mathsf{Spl}_{#1}(#2)}

\newcommand{\Spidp}[2]{\mathbf{Sp}_{#1}(#2)}
\newcommand{\Spidq}[2]{\mathsf{Sp}'_{#1}(#2)}
\newcommand{\cij}{c_{ij}}
\newcommand{\Hom}{\operatorname{Hom}}
\newcommand{\shift}{\Lambda}

\title[Knot contact homology as a planar limit of Chern--Simons theory]{Knot contact homology as \\
a planar limit of Chern--Simons theory}
\author{Ben Webster}
\author{Meri Zaimi}

\newcommand{\Ben}[1]{}
\newcommand{\Meri}[1]{}
\newcommand{\ourlambda}{\Lambda}
\newcommand{\nglambda}{\lambda}
\newcommand{\ourmu}{\nu}
\newcommand{\pverma}{\mathsf{M}}
\newcommand{\ngmu}{\mu}
\newcommand{\ngab}[1]{\operatorname{KCH}(#1)}
\newcommand{\Iq}[1]{I_q(#1)}
\newcommand{\Iqp}[1]{I_q'(#1)}

\newcommand{\ourabq}[1]{\mathsf{H}_q(#1)}
\newcommand{\ourab}[1]{\mathsf{H}(#1)}
\newcommand{\wri}{\operatorname{wr}}
\newcommand{\aij}{\gamma_{ij}}
\newcommand{\Br}{\mathbf{r}}

\newcommand{\Bs}{\mathbf{s}}
\newcommand{\Bw}{\mathbf{w}}
\newcommand{\funcs}[1]{\mathcal{F}_{#1}}
\newcommand{\qt}[1]{\mathcal{T}_{#1}}
\newcommand{\nus}[1]{\mathcal{N}_{#1}}
\newcommand{\ev}{\mathsf{ev}}
\newcommand{\Laur}[1]{\mathbb{T}_{#1}}
\newcommand{\Tor}{\operatorname{Tor}}

\begin{document}

\begin{abstract}

We prove a conjecture relating augmentation varieties to the large $N$ limit of Chern--Simons theory. Although our result does not directly establish that the augmentation polynomial of a knot is the classical limit of a deformed $\hat{A}$-polynomial---as suggested by Aganagi\'c and Vafa---it reduces that problem to characterizing certain algebraic properties of a module over the quantum torus introduced by Gaiotto, Kannagi, and Sanjurjo. We call this module the \emph{HOMFLYPT difference module}; it encodes relations among colored HOMFLYPT polynomials with different antisymmetric colorings. We show that, for a knot, the classical limit of this difference module is precisely degree 0 abelianized knot contact homology, and we give a natural extension of this result to links.

\end{abstract}
\maketitle
\section{Introduction}

In \cite{gaiottoDbranesPlanar2026}, Gaiotto, Kannagi, and Sanjurjo discuss applications of holographic duality between Chern--Simons theory and the A-model of topological string theory. The former has well-known applications in knot theory as the source of quantum invariants, such as the Jones polynomial; on the dual side, this should manifest as the study of open Gromov--Witten theory for the conormal to the knot in $T^*S^3$. This, in turn, has been related to knot contact homology by the work of Aganagi\'c, Ekholm, Ng, and Vafa \cite{vafaTopologicalStrings2014}. From a mathematical perspective, this leads to the conjecture that the three-variable augmentation variety of a link can be described in terms of a large $N$ limit of Chern--Simons theory for $\mathfrak{sl}_N$.

More precisely, by a special case of the work of Garoufalidis, Lauda, and L\^e \cite{garoufalidisColoredHOMFLYPT2018}, the HOMFLYPT polynomials of a link corresponding to the wedge powers $\bigwedge{}^{\! m}\C^{N}$ form a $q$-holonomic sequence in $m$. That is, the sequence satisfies a recursion relation in the $q$-commuting transformations $\ourlambda$ and $\ourmu$, where $\ourlambda$ sends $m \mapsto m-1$ and $\ourmu$ acts by multiplication by $q^{m}$. This relation is sometimes called the deformed $\hat{A}$-polynomial. 

Based on \cite{gaiottoDbranesPlanar2026}, we reinterpret this as a statement about a {\bf HOMFLYPT difference module}: a module over the quantum torus that provides a purely skein-theoretic description of these recursions by treating $m$ as a formal variable. While interpreting the polynomial as a function of $q$ for formal $m$ presents difficulties, one can work skein-theoretically to relate values where $m$ is shifted by an integer.

We illustrate this with an example based on \cite[(2.15) \& \S 2.3]{gaiottoDbranesPlanar2026}. Consider the unknot labeled with the representation $\bigwedge{}^{\! m}\C^{N}$. We can compare this with the value of the quantum knot invariant for $\bigwedge{}^{\! m-1}\C^{N}$ using relations between webs in MOY (Murakami--Ohtsuki--Yamada) calculus, a diagrammatic calculus for quantum $\mathfrak{gl}_N$ invariants that we review in detail in \cref{sec:spider}:
\[  \tikz[baseline]{\draw[->,very thick] (0,0.5) node[above] {$m$} arc (90:-315:0.5cm);} =\frac{q-q^{-1}}{q^{m}-q^{-m}} \,  \tikz[baseline]{\draw[->,very thick] (0,0.5)  arc (90:-315:0.5cm); \node at (.75,-.1){$m$}; \node at (-1,.1){$m-1$}; \draw[mid<,very thick] (0,0.5) -- (0,-.5);}=\frac{q^{N-m+1}-q^{-N+m-1}}{q^{m}-q^{-m}}\tikz[baseline]{\draw[->,very thick] (0,0.5) node[above] {$m-1$} arc (90:-315:0.5cm);}. \]
We can ``HOMFLYPTize'' this formula in the final step by substituting $g=q^N$ and $\ourmu=q^{m}$. This substitution is purely formal: while evaluating the HOMFLYPT polynomial at a non-integer value of $m$ is problematic, the skein relations themselves remain valid when we treat $\ourmu=q^m$ as an independent variable. The key insight is that the MOY relations relate diagrams with labels differing by integers, so we can work with formal shifts of $m$ without ever needing to evaluate at a specific value. This removes the explicit dependence on $N$, and now the equation only depends on $m$ through $\ourmu$.  This is effectively taking the planar limit of Chern--Simons theory, sending $N\to \infty$ and $q\to 1$. 

Finally, since $m$ is now purely formal, we can define a transformation $\ourlambda$\notation{$\ourlambda$}{The label-shift operator sending $m$ to $m-1$.} (the precise domain and codomain of which are discussed later) that sends the label $m$ to $m-1$. While this shift is not well-defined for fixed integers $m$, it is consistent for a formal variable. To preserve the MOY relations---specifically \cref{eq:bigon2}---the transformations $\ourlambda$ and $\ourmu$ must $q$-commute such that $\ourmu\ourlambda=q\ourlambda\ourmu$. Using this notation, we have:
\[  \tikz[baseline]{\draw[->,very thick] (0,0.5) node[above] {$m$} arc (90:-315:0.5cm);} =\frac{qg\ourmu^{-1} -q^{-1}g^{-1}\ourmu}{\ourmu-\ourmu^{-1}}\tikz[baseline]{\draw[->,very thick] (0,0.5) node[above] {$m-1$} arc (90:-315:0.5cm);} =\frac{qg\ourmu^{-1} -q^{-1}g^{-1}\ourmu}{\ourmu-\ourmu^{-1}}\ourlambda\, \tikz[baseline]{\draw[->,very thick] (0,0.5) node[above] {$m$} arc (90:-315:0.5cm);}\;. \]

Thus, we conclude that the expression $\ourmu-\ourmu^{-1}-\ourlambda(g\ourmu^{-1} -g^{-1}\ourmu)$ annihilates the value of the unknot. This is the deformed $\hat{A}$-polynomial of the unknot, up to several changes of convention. For instance, our results align with \cite[\S 6.1]{aganagicLargeDuality2012}, with minor differences due to their use of symmetric rather than antisymmetric powers; our variables $q, g, \ourmu, \ourlambda$ correspond to their $q^{1/2}, Q^{1/2}, e^{\hat{p}/2}, e^{\hat{x}}$, respectively. After setting $q=1$ and applying the substitution in \cref{eq:substitution}, this matches the augmentation polynomial of the unknot. Note that this derivation does not require an explicit calculation of the quantum invariant; in \cref{rem:difficulties}, we discuss the challenges of defining that invariant in the limits $N\to \infty$ and $q\to 1$. 

Fix a general link $L$. Consider the quantum torus 
\[ \qt{r} = \C[q^{\pm 1}, g^{\pm 1}] \langle \ourlambda_i, \ourmu_i \rangle / ( \ourlambda_j \ourmu_i = q^{-\delta_{ij}} \ourmu_i \ourlambda_j, \ourmu_i \ourmu_j = \ourmu_j \ourmu_i, \ourlambda_i \ourlambda_j = \ourlambda_j \ourlambda_i ) \]
whose number of variables $r$ is equal to the number of components of $L$; let $\nus{r}$ be the subalgebra generated by $q^{\pm 1}$ and $\ourmu_i$. Applying an argument analogous to that above, we obtain a module $\ourabq{L}$ over $\qt{r}$ whose elements are webs where the link is colored by formal parameters $m_i$ shifted by integers, with additional web portions colored by integers. We call this the \textbf{HOMFLYPT difference module} of $L$.

Let $\funcs{r}$ be the space of all maps $\Z^r \to \C(q)[g^{\pm 1}]$. This is a module over the quantum torus $\qt{r}$ by the action 
\[ \ourmu_i f(k_1, \dots, k_r) = q^{k_i} f(k_1, \dots, k_r), \qquad \ourlambda_i f(k_1, \dots, k_r) = f(k_1, \dots, k_i-1, \dots, k_r). \]
There is a natural $\qt{r}$-module homomorphism $\ev \colon \ourabq{L} \to \funcs{r}$ sending a web $w$ to the function $f_w(k_1, \dots, k_r)$ given by the evaluation of $w$ in the HOMFLYPT skein category after substituting $k_i$ for $m_i$. This provides a new perspective on finding recursions satisfied by antisymmetric HOMFLYPT polynomials by framing the problem as an investigation into the structure of the HOMFLYPT difference module $\ourabq{L}$. Let $[L]$ be the element of $\ourabq{L}$ consisting of the link $L$ colored by the parameters $m_i$.

\begin{definition}
    Let $\Iq{L} = \{ x \in \qt{r} \mid x[L] = 0 \}$ be the annihilator of $[L]$.
\end{definition}

Since the homomorphism $\ev$ maps $[L]$ to the colored HOMFLYPT polynomials of $L$ for all antisymmetric colorings (viewed as an element of $\funcs{r}$), all elements of the left ideal $\Iq{L}$ yield recurrence relations satisfied by these polynomials. We conjecture (\cref{conj:ev-injective}) that $\ev$ is injective and that $\Iq{L}$ coincides with the ideal $\Iqp{L}$ of recurrence relations satisfied by the antisymmetric HOMFLYPT polynomials of $L$.

The $q$-holonomicity theorem of \cite{garoufalidisColoredHOMFLYPT2018} shows that the variety defined by $\Iqp{L}$ in the classical limit $q \to 1$ is Lagrangian; accordingly, one would expect the same for $\Iq{L}$, suggesting that $\ourabq{L}$ is ``not too big.''

The advantage of this perspective is that, unlike its image under $\ev$, we can provide an explicit presentation of the module $\ourabq{L}$.

We can write $\ourabq{L}$ as a quotient of the tensor product of the quantum torus and enveloping algebra $U = \qt{r} \otimes U_q(\mathfrak{gl}_s)$ arising from a presentation of $L$ as the closure of a braid $\beta$ with $s$ strands. Specifically, it is a quotient of this algebra by the sum of a left and a right ideal---that is, a tensor product of cyclic left and right modules. The resulting presentation has a particularly simple classical limit as $q \to 1$, consistent with earlier discussions of the augmentation polynomial.

\begin{itheorem}\label{main-theorem-intro}
The base change $\ourab{L}$ of $\ourabq{L}$ to the locus where $q=1$ (keeping $\ourmu_i, \ourlambda_i$ as formal variables) is the degree 0 part of the abelianized knot contact homology $\ngab{L}$ under the parameter change:	
\begin{equation}\label{eq:substitution-intro}
		\ngmu \mapsto \ourmu^{-2}, \qquad U \mapsto g^{-2}, \qquad \nglambda \mapsto -g^{-1}\ourlambda^{-1}.
	\end{equation} 
In particular, the base change $I(L)$ of the ideal $\Iq{L}$ lies in the augmentation ideal $\mathsf{Au}_L$ of $L$, that is, the kernel of the natural map $\C[U^{\pm 1},\allowbreak \ngmu_i^{\pm 1},\allowbreak \lambda_i^{\pm 1}] \to \ngab{L}$.
\end{itheorem}

We conjecture that $\ev$ is injective, so $\Iq{L}$ describes the recursions of the colored HOMFLYPT polynomials and that this base change is precisely the augmentation ideal of $L$ (see \cref{conj:augmentation-ideal}). Together, these conjectures would provide a rigorous proof of the conjecture in \cite{aganagicLargeDuality2012, vafaTopologicalStrings2014} that the augmentation ideal of $L$ is a classical limit of the recursions satisfied by colored HOMFLYPT polynomials.
    
We emphasize that the proof of \cref{main-theorem-intro} is purely algebraic. Compared to the difficulty of calculating knot invariants for specific values of $N$, our approach is relatively elementary and follows the spirit of \cite{garoufalidisColoredHOMFLYPT2018}: one derives recursions for quantum link invariants from MOY calculus and demonstrates, without explicit calculation of invariants, that in the limit $N \to \infty, q \to 1$, these relations recover the equations of the augmentation variety implicitly given in \cite[Th. 1.3]{ekholmFiltrationsKnot2013}.

Formalizing these arguments requires working with a limit of the representation category of $U_q(\mathfrak{gl}_N)$ where, in addition to standard wedge powers $\bigwedge{}^{\! m}\C^{N}$, we include objects where $m$ is treated as a formal variable. This limit appears worthy of further study, although it breaks certain structures of the category, such as its braided and rigid structures.

An index of notation is provided at the end of the paper for reference.

\subsection*{Acknowledgments}
Many thanks to Davide Gaiotto for exceptionally useful discussions; this paper owes a large debt to his generous sharing of his ideas.  We also thank Mina Aganagi\'c and Joel Kamnitzer for helpful comments when this paper was in the draft stage.

The authors were supported by the NSERC through a Discovery Grant and Postdoctoral Fellowship. This research was supported in part by Perimeter Institute for Theoretical Physics. Research at Perimeter Institute is supported in part by the Government of Canada through the Department of Innovation, Science and Economic Development Canada and by the Province of Ontario through the Ministry of Colleges and Universities.

\section{Heavy skein algebras}
\subsection{The HOMFLYPT spider}\label{sec:spider}

Our starting point is the ribbon category of representations of $U_q(\mathfrak{gl}_N)$. This is a braided spherical $\C[q,q^{-1}]$-linear category.  

However, in this paper, rather than considering $N$ as a fixed integer, we vary it, following the approach of Deligne categories. The key idea, going back to Deligne's work on interpolating representation categories \cite{Deligne2007}, is that many categorical constructions that appear to require $N$ to be a positive integer can be defined for formal values of $N$ by working with polynomial relations rather than explicit vector spaces. A version of the Deligne category for $U_q(\mathfrak{gl}_N)$ is discussed by Brundan in \cite[\S 1.4]{brundanRepresentationsOriented2017}, based on the oriented skein category.   We will use a variant of this approach based on the work of Murakami, Ohtsuki, and Yamada \cite{murakamiHOMFLYPolynomial1998}, developed further by Cautis, Kamnitzer, and Morrison \cite{cautisWebsQuantum2014} and Queffelec and Sartori \cite{queffelecMixedQuantum2019}, which gives a combinatorial description of this category in terms of {\it spiders}, certain planar diagrams on which we impose linear relations.  This makes it clear how the relations of the category depend on $N$.  In fact, we can remove the dependence on $N$ entirely by considering the power $g=q^N$ \notation{$g$}{The quantity $q^N$ where we think of $N$ as a variable.} as an independent variable---one can think of this as ``HOMFLYPTizing'' the spider category, since the resulting link invariants become functions of $q$ and $g$ independently, as in the HOMFLYPT polynomial. The resulting category is denoted $\mathsf{Sp}(N)$ in the notation of \cite[Def. 6.5]{queffelecMixedQuantum2019}.  

The variable $N$ only appears in the spider relations in quantum binomial coefficients of the form
\begin{equation}\label{eq:formal-q-binom}
 \qBinomial{N-a}{b}=\frac{(gq^{-a}-g^{-1}q^{a})\cdots (gq^{-a-b+1}-g^{-1}q^{a+b-1})}{(q^{b}-q^{-b})\cdots (q-q^{-1})}.	
\end{equation}
We will continue to use the symbol $\qBinomial{N-a}{b}$, but use it to mean the rational function in $g$ and $q$ given by \eqref{eq:formal-q-binom}.

The category $\mathsf{Sp}(N)$ is defined in terms of ladders:
\begin{definition}\label{def:ladder1}
	A ladder with $k$ uprights is a diagram drawn in a rectangle with  
	\begin{itemize}
		\item 
	$k$ parallel vertical lines that run from the bottom edge to the top edge of the rectangle, oriented upwards or downwards.  
	\item some number of oriented horizontal rungs connecting adjacent uprights,
	\item a labeling of each rung and each section of an upright by an element of $\Z_{\geq 0}$. The labeling must satisfy the condition that the signed sum of the segments at each vertex is zero.
	\end{itemize} 
\end{definition}

\begin{definition}\label{def:spider}
	The \textbf{(HOMFLYPT) spider category} $\Spid{N}$\notation{$\Spid{N}$}{The HOMFLYPT spider category.} is a braided monoidal category generated by objects $ \Z_{\geq 0}^+=\{n^+\mid n\in \Z_{\geq 0}\}$ and their duals $\Z_{\geq 0}^-=\{n^-\mid n\in \Z_{\geq 0}\}$. In the original representation-theoretic interpretation of this category, the objects $n^{\pm}$ corresponded to the $q$-deformation of $\bigwedge{}^{\!n}\C^N$ to a representation of $U_q(\mathfrak{gl}_N)$, and to its dual.

	Thus, a general object in this category is a finite list of these generators, which one should think of as corresponding to a tensor product of the wedge powers mentioned above.

The tensor product of two objects is given by concatenating their lists;  in particular, the empty list gives a unit object in the category. We include an isomorphism of $0^{\pm}$ to the empty list in our category.  
        
	The morphisms from $(a_1^{\epsilon_1},\dots, a_k^{\epsilon_k})$ to $(b_1^{\tau_1},\dots, b_{\ell}^{\tau_{\ell}})$  for $a_i,b_i\in\Z_{\geq 0}$ and $\epsilon_i,\tau_i\in \{+,-\}$ are generated by a fixed isomorphism $0^+\cong \emptyset$ and by ladders in $\R\times [0,1]$ that satisfy:
	\begin{enumerate}
		\item The points where the graph meets $\R\times \{0\}$ are labeled with $a_i$ read from left to right, with the edge oriented upward if $\epsilon_i= +$ and downward if $\epsilon_i = -$.  
		\item The points where the graph meets $\R\times \{1\}$ are labeled with $b_i$ read from left to right, with the edge oriented upward if $\tau_i= +$ and downward if $\tau_i = -$.  
	\end{enumerate}
    The tensor product of morphisms is as usual given by horizontal composition.  We can interpret an arbitrary morphism in this category as a trivalent graph, where applying the isomorphism $0^+\cong \emptyset$ to an upright of a ladder is represented by simply contracting that upright and smoothing out the trivalent vertex it arose from.  Conversely, we can ``ladderize'' any such trivalent graph to obtain a unique morphism, using uprights with the label $0^{\pm}$.  
    
The morphism space between two objects is spanned by formal linear combinations of such graphs over $\C(q)[g^{\pm 1}]$, considered up to isotopy of the corresponding graphs, modulo 
\begin{enumerate}
    \item the relations below,
    \item their mirror images and rotations, and
    \item the relations obtained by reversing the orientation of an upright in a ladder, and switching the label $k$ to $N-k$.  We include the two most important examples of the relations obtained this way in \cref{eq:bigon2,eq:commutation-2}.
\end{enumerate}
\begin{gather}
  \tikz[baseline]{\draw[->] (0,0.5) node[above] {$k$} arc (45:-315:0.5cm);}
  = \qBinomial{N}{k}
  \label{eq:loop}
  \\[1em]
  \begin{tikzpicture}[baseline=20]
    \foreach \n in {0,...,3} {
      \coordinate (z\n) at (0.4*\n, 0.8*\n);
    }
    \draw[mid>] (z0) -- node[right] {$k$}(z1);
    \draw[mid>] (z2) --  node[right] {$k$}(z3);
    \draw[mid>] (z1) to[out=150,in=-190] node[left] {$l$} (z2);
    \draw[mid>] (z1) to[out=-30,in=0]  node[right] {$k-l$}  (z2);
  \end{tikzpicture}
  = \qBinomial{k}{l}
  \tikz[baseline=20]{\draw[mid>] (0,0) -- node[right] {$k$} (1,2);}
  \qquad
  \begin{tikzpicture}[baseline=20]
    \foreach \n in {0,...,3} {
      \coordinate (z\n) at (0.4*\n, 0.8*\n);
    }
    \draw[mid>] (z0) -- node[right] {$k$} (z1);
    \draw[mid>] (z2) -- node[right] {$k$} (z3);
    \draw[mid<] (z1) to[out=150,in=-190] node[left] {$l$} (z2);
    \draw[mid>] (z1) to[out=-30,in=0]   node[right] {$k+l$} (z2);
  \end{tikzpicture}
  = \qBinomial{N-k}{l}
  \tikz[baseline=20]{\draw[mid>] (0,0) -- node[right] {$k$} (1,2);}
  \label{eq:bigon2}
  \displaybreak[1]
  \\[1em]
  \begin{tikzpicture}[baseline=20]
    \foreach \x/\y in {0/0,1/0,2/0,0/1,1/1,0/2} {
      \coordinate (z\x\y) at (\x+\y/2,\y/1.5);
    }
    \coordinate (z03) at (1,2);
    \draw[mid>] (z00) node[below] {$k$}      -- (z01);
    \draw[mid>] (z01)              -- node[left]  {$k+l$} (z02);
    \draw[mid>] (z10) node[below] {$l$}      -- (z01);
    \draw[mid>] (z20) node[below] {$m$}      -- (z02);
    \draw[mid>] (z02) -- node[left] {$k+l+m$} (z03);
  \end{tikzpicture}
  =
  \begin{tikzpicture}[baseline=20]
    \foreach \x/\y in {0/0,1/0,2/0,0/1,1/1,0/2} {
      \coordinate (z\x\y) at (\x+\y/2,\y/1.5);
    }
    \coordinate (z03) at (1,2);
    \draw[mid>] (z00) node[below] {$k$} -- (z02);
    \draw[mid>] (z10) node[below] {$l$} -- (z11);
    \draw[mid>] (z20) node[below] {$m$} -- (z11);
    \draw[mid>] (z11)              -- node[right] {$l+m$} (z02);
    \draw[mid>] (z02) -- node[left] {$k+l+m$} (z03);
  \end{tikzpicture}
  \label{eq:IH}
  \displaybreak[1]
  \\
  \label{eq:commutation}
  \begin{tikzpicture}[baseline=40]
    \laddercoordinates{1}{2}
    \node[left]  at (l00) {$k$};
    \node[right] at (l10) {$l$};
    \ladderEn{0}{0}{$k{-}s$}{$l{+}s$}{$s$}
    \ladderFn{0}{1}{$k{-}s{+}r$}{$l{+}s{-}r$}{$r$}
  \end{tikzpicture}
  = \sum_t \qBinomial{k-l+r-s}{t}
  \begin{tikzpicture}[baseline=40]
    \laddercoordinates{1}{2}
    \node[left]  at (l00) {$k$};
    \node[right] at (l10) {$l$};
    \ladderFn{0}{0}{$k{+}r{-}t$}{$l{-}r{+}t$}{$r{-}t$}
    \ladderEn{0}{1}{$k{-}s{+}r$}{$l{+}s{-}r$}{$s{-}t$}
  \end{tikzpicture}
  \displaybreak[1]
  \\
  \label{eq:commutation-2}
    \begin{tikzpicture}[baseline=40]
    \laddercoordinates{1}{2}
    \node[left]  at (l00) {$k$};
    \node[right] at (l10) {$l$};
    \ladderEd{0}{0}{$k{+}s$}{$l{+}s$}{$s$}{mid<}{mid>}
    \ladderFd{0}{1}{$k{+}s{-}r$}{$l{+}s{-}r$}{$r$}{mid<}{mid>}
  \end{tikzpicture}
  = \sum_t \qBinomial{N-k-l+r-s}{t}
  \begin{tikzpicture}[baseline=40]
    \laddercoordinates{1}{2}
    \node[left]  at (l00) {$k$};
    \node[right] at (l10) {$l$};
    \ladderFd{0}{0}{$k{-}r{+}t$}{$l{-}r{+}t$}{$r{-}t$}{mid<}{mid>}
    \ladderEd{0}{1}{$k{+}s{-}r$}{$l{+}s{-}r$}{$s{-}t$}{mid<}{mid>}
  \end{tikzpicture}
  \displaybreak[1]
  \\
    \label{eq:id1b}
  \tikz[baseline=40]{
    \laddercoordinates{1}{2}
    \ladderEn{0}{0}{$k-s$}{$l+s$}{$s$}
    \ladderEn{0}{1}{$k-s-r$}{$l+s+r$}{$r$}
    \node[left]  at (l00) {$k$};
    \node[right] at (l10) {$l$};
  }
  =
  \qBinomial{r+s}{r}
  \tikz[baseline=20]{
    \laddercoordinates{1}{1}
    \ladderEn{0}{0}{$k-s-r$}{$l+s+r$}{$r+s$}
    \node[left]  at (l00) {$k$};
    \node[right] at (l10) {$l$};
  }
  \displaybreak[1]
  \\
  \begin{tikzpicture}[baseline=32,yscale=.7]
  \laddercoordinates{2}{2}
\node[below] at (l00) {$k_1$};
\node[below] at (l10) {$k_2$};
\node[below] at (l20) {$k_3$};
\ladderEn{0}{0}{$k_1{-}r$}{$k_2{+}r$}{$r$}
\ladderFn{1}{1}{$k_2{+}r{+}s$}{$k_3{-}s$}{$s$}
\ladderIn{0}{1}{1}
\ladderIn{2}{0}{1}
\end{tikzpicture}
 =
\begin{tikzpicture}[baseline=32,yscale=.7]
\laddercoordinates{2}{2}
\node[below] at (l00) {$k_1$};
\node[below] at (l10) {$k_2$};
\node[below] at (l20) {$k_3$};
\ladderEn{0}{1}{$k_1{-}r$}{$k_2{+}r{+}s$}{$r$}
\ladderFn{1}{0}{$k_2{+}s$}{$k_3{-}s$}{$s$}
\ladderIn{0}{0}{1}
\ladderIn{2}{1}{1}
\end{tikzpicture}
\label{eq:IHlad}
\displaybreak[1] \\
\begin{tikzpicture}[baseline=32,yscale=.7]
\laddercoordinates{2}{2}
\node[below] at (l00) {$k_1$};
\node[below] at (l10) {$k_2$};
\node[below] at (l20) {$k_3$};
\ladderFn{0}{0}{$k_1{+}r$}{$k_2{-}r$}{$r$}
\ladderEn{1}{1}{$k_2{-}r{-}s$}{$k_3{+}s$}{$s$}
\ladderIn{0}{1}{1}
\ladderIn{2}{0}{1}
\end{tikzpicture}
 =
\begin{tikzpicture}[baseline=32,yscale=.7,xscale=.9]
\laddercoordinates{2}{2}
\node[below] at (l00) {$k_1$};
\node[below] at (l10) {$k_2$};
\node[below] at (l20) {$k_3$};
\ladderFn{0}{1}{$k_1{+}r$}{$k_2{-}r{-}s$}{$r$}
\ladderEn{1}{0}{$k_2{-}s$}{$k_3{+}s$}{$s$}
\ladderIn{0}{0}{1}
\ladderIn{2}{1}{1}
\end{tikzpicture}
\label{eq:IHlad2}\\
 \begin{tikzpicture}[baseline=45,yscale=.7,xscale=.9]
 \laddercoordinates{2}{3}
\ladderEn{0}{0}{}{}{$1$}
\ladderEn{0}{1}{}{}{$1$}
\ladderEn{1}{2}{}{}{$1$}
\ladderI{0}{2}
\ladderIn{2}{0}{2}
\node[below] at (l00) {$k_1$};
\node[below] at (l10) {$k_2$};
\node[below] at (l20) {$k_3$};
\node[below] at (u03) {$k_1-2$};
\node[below] at (u13) {$k_2+1$};
\node[below] at (u23) {$k_3+1$};
\end{tikzpicture}
- [2]_q
 \begin{tikzpicture}[baseline=45,yscale=.7,xscale=.9]\laddercoordinates{2}{3}
 \ladderEn{0}{0}{}{}{$1$}
\ladderEn{0}{2}{}{}{$1$}
\ladderEn{1}{1}{}{}{$1$}
\ladderI{0}{1}
\ladderI{2}{0}
\ladderI{2}{2}
\node[below] at (l00) {$k_1$};
\node[below] at (l10) {$k_2$};
\node[below] at (l20) {$k_3$};
\node[below] at (u03) {$k_1-2$};
\node[below] at (u13) {$k_2+1$};
\node[below] at (u23) {$k_3+1$};
\end{tikzpicture}
+
 \begin{tikzpicture}[baseline=45,yscale=.7,xscale=.9]\laddercoordinates{2}{3}
\ladderEn{1}{0}{}{}{$1$}
\ladderEn{0}{1}{}{}{$1$}
\ladderEn{0}{2}{}{}{$1$}
\ladderI{0}{0}
\ladderIn{2}{1}{2}
\node[below] at (l00) {$k_1$};
\node[below] at (l10) {$k_2$};
\node[below] at (l20) {$k_3$};
\node[below] at (u03) {$k_1-2$};
\node[below] at (u13) {$k_2+1$};
\node[below] at (u23) {$k_3+1$};
\end{tikzpicture}
= 0\label{eq:serre1}
\displaybreak[1] \\
\begin{tikzpicture}[baseline=45,yscale=.7,xscale=.9]\laddercoordinates{2}{3}
\ladderEn{1}{0}{}{}{$1$}
\ladderEn{1}{1}{}{}{$1$}
\ladderEn{0}{2}{}{}{$1$}
\ladderI{2}{2}
\ladderIn{0}{0}{2}
\node[below] at (l00) {$k_1$};
\node[below] at (l10) {$k_2$};
\node[below] at (l20) {$k_3$};
\node[below] at (u03) {$k_1-1$};
\node[below] at (u13) {$k_2-1$};
\node[below] at (u23) {$k_3+2$};
\end{tikzpicture}
- [2]_q
 \begin{tikzpicture}[baseline=45,yscale=.7,xscale=.9]\laddercoordinates{2}{3}
 \ladderEn{1}{0}{}{}{$1$}
\ladderEn{1}{2}{}{}{$1$}
\ladderEn{0}{1}{}{}{$1$}
\ladderI{2}{1}
\ladderI{0}{0}
\ladderI{0}{2}
\node[below] at (l00) {$k_1$};
\node[below] at (l10) {$k_2$};
\node[below] at (l20) {$k_3$};
\node[below] at (u03) {$k_1-1$};
\node[below] at (u13) {$k_2-1$};
\node[below] at (u23) {$k_3+2$};
\end{tikzpicture}
+
 \begin{tikzpicture}[baseline=45,yscale=.7,xscale=.9]\laddercoordinates{2}{3}
\ladderEn{0}{0}{}{}{$1$}
\ladderEn{1}{1}{}{}{$1$}
\ladderEn{1}{2}{}{}{$1$}
\ladderI{2}{0}
\ladderIn{0}{1}{2}
\node[below] at (l00) {$k_1$};
\node[below] at (l10) {$k_2$};
\node[below] at (l20) {$k_3$};
\node[below] at (u03) {$k_1-1$};
\node[below] at (u13) {$k_2-1$};
\node[below] at (u23) {$k_3+2$};
\end{tikzpicture}
= 0\label{eq:serre2}
\displaybreak[1] 
\end{gather}
We can define over and undercrossings in this category by the formulas below, making it into a braided monoidal category:
\begin{gather}
  \begin{tikzpicture}[baseline=10,xscale=.8]
    \node (C) at (-1,2) {$l$};
    \node (D) at (1,-1) {$l$};
    \node (A) at (1,2)  {$k$};
    \node (B) at (-1,-1){$k$};
    \node (i) at (intersection of A--B and D--C) {};
    \draw[mid>] (B) -- (A);
    \draw[mid>] (D) -- (i);
    \draw[mid>] (i) -- (C);
  \end{tikzpicture}
  = \sum_{\substack{a,b\geq 0 \\ b-a = k-l}} (-q)^{k-b}
  \begin{tikzpicture}[baseline=40]
    \laddercoordinates{1}{2}
    \node[left]  at (l00) {$k$};
    \node[right] at (l10) {$l$};
    \ladderEn{0}{0}{$k{-}b$}{$l{+}b$}{$b$}
    \ladderFn{0}{1}{$l$}{$k$}{$a$}
  \end{tikzpicture}
    = \sum_{\substack{a,b\geq 0 \\ b-a = l-k}} (-q)^{l-b}
  \begin{tikzpicture}[baseline=40]
    \laddercoordinates{1}{2}
    \node[left]  at (l00) {$k$};
    \node[right] at (l10) {$l$};
    \ladderFn{0}{0}{$k{+}b$}{$l{-}b$}{$b$}
    \ladderEn{0}{1}{$l$}{$k$}{$a$}
  \end{tikzpicture} \;,
  \label{eq:braid}
    \\
  \begin{tikzpicture}[baseline=10,xscale=.8]
    \node (C) at (-1,2) {$l$};
    \node (D) at (1,-1) {$l$};
    \node (A) at (1,2)  {$k$};
    \node (B) at (-1,-1){$k$};
    \node (i) at (intersection of A--B and D--C) {};
    \draw[mid>] (D) -- (C);
    \draw[mid>] (B) -- (i);
    \draw[mid>] (i) -- (A);
  \end{tikzpicture}
  = \sum_{\substack{a,b\geq 0 \\ b-a = k-l}} (-q)^{b-k}
  \begin{tikzpicture}[baseline=40]
    \laddercoordinates{1}{2}
    \node[left]  at (l00) {$k$};
    \node[right] at (l10) {$l$};
    \ladderEn{0}{0}{$k{-}b$}{$l{+}b$}{$b$}
    \ladderFn{0}{1}{$l$}{$k$}{$a$}
  \end{tikzpicture}  = \sum_{\substack{a,b\geq 0 \\ b-a = l-k}} (-q)^{b-l}
  \begin{tikzpicture}[baseline=40]
    \laddercoordinates{1}{2}
    \node[left]  at (l00) {$k$};
    \node[right] at (l10) {$l$};
    \ladderFn{0}{0}{$k{+}b$}{$l{-}b$}{$b$}
    \ladderEn{0}{1}{$l$}{$k$}{$a$}
  \end{tikzpicture} \;.
  \label{eq:braid-op}
\end{gather}
\end{definition}

The spider category is closely related to another diagrammatic category:
\begin{lemma}
	There is an inclusion functor $\mathcal{OS}(q^{-1}-q,-g)\to \Spid{N}$ and this functor induces an equivalence of Karoubi envelopes.
\end{lemma}
Here $\mathcal{OS}(z,t)$ denotes the \textbf{oriented skein category}, whose morphisms are tangles in a thickened strip modulo the oriented skein relation; see \cite[\S 3]{brundanRepresentationsOriented2017} for details. This lemma, noted in \cite{brundanRepresentationsOriented2017} and proved in \cite[Prop. 6.7]{queffelecMixedQuantum2019}, shows that our spider category is equivalent (after idempotent completion) to the category underlying the HOMFLYPT skein theory.

 In particular, we have a skein relation generalizing \cite[(S)]{brundanRepresentationsOriented2017}:
\begin{equation}\label{eq:skein}
     \begin{tikzpicture}[baseline=20]
    \node (C) at (0,1.5) {$\ell$};
    \node (D) at (1,0) {$\ell$};
    \node (A) at (1,1.5)  {$1$};
    \node (B) at (0,0){$1$};
    \node (i) at (intersection of A--B and D--C) {};
    \draw[mid>] (B) -- (A);
    \draw[mid>] (D) -- (i);
    \draw[mid>] (i) -- (C);
  \end{tikzpicture} -      \begin{tikzpicture}[baseline=20]
    \node (C) at (0,1.5) {$\ell$};
    \node (D) at (1,0) {$\ell$};
    \node (A) at (1,1.5)  {$1$};
    \node (B) at (0,0){$1$};
    \node (i) at (intersection of A--B and D--C) {};
    \draw[mid>] (D) -- (C);
    \draw[mid>] (B) -- (i);
    \draw[mid>] (i) -- (A);
  \end{tikzpicture} = (q^{-1}-q)\,\begin{tikzpicture}[baseline=20]
       \laddercoordinates{1}{1}
    \ladderFn{0}{0}{$\ell$}{$1$}{$\ell-1$}
    \node[left]  at (l00) {$1$};
    \node[right] at (l10) {$\ell$};
  \end{tikzpicture} \;.
\end{equation}
As mentioned before, these relations give an isomorphism between $0^{\pm}$ and the empty list via the lollipop diagram (a single trivalent vertex, with two of the adjacent edges joined).   Thus, our category is effectively unchanged if we just delete all edges with label $0$.

It follows by \cite[Cor. 6.17]{queffelecMixedQuantum2019} and \cite[Th. 1.3]{brundanRepresentationsOriented2017} that we can recover the usual finite $\mathfrak{gl}_e$ spider category $\Spid{e}_{\leq e}$ for $e\in \Z_{\geq 0}$ by specializing $g=q^e$ and killing the objects $n^\pm$ for all $n>e$.  Note that this specialization already sets the dimension of the objects we have killed to 0.  

Experts will note that we have used a renormalized braiding rather than the usual choice for $\mathfrak{sl}_N$, and thus \cref{eq:braid} does not match \cite[(E.6)]{gaiottoDbranesPlanar2026}; it does match the renormalized crossing introduced in \cite[(C.8)]{gaiottoDbranesPlanar2026}. This yields the same link invariants for links in the zero-framing as the usual choice, but at the price that we can only choose a realization functor to the braided category of $U_q(\mathfrak{gl}_N)$ modules; as noted by Queffelec \cite{queffelecSkeinModules2015}, we need the strand with label $N$ to correspond to a non-trivial object.  This difference is not important for our purposes, and avoiding these powers will simplify the exposition.

An important observation that helps structure this category is \cite[Prop. 5.2.1]{cautisWebsQuantum2014}: there is a functor from the idempotented form $\dot{U}$ of $U_q(\mathfrak{gl}_k)$, considered as a category with each weight as an object. This functor sends the weight $(p_1,\dots, p_k)$ to the tensor product $p_1^+\otimes \cdots \otimes p_k^+$ if $p_i\geq 0$ for all $i$, and to 0 otherwise.  Composing with the realization functor, we recover skew Howe duality (\cite[Th. 4.2.1]{cautisWebsQuantum2014}).  We can generalize this to include the objects $n^{-}$ as in \cite{queffelecMixedQuantum2019} (where they denote $n^+$ by $n$ and $n^-$ by $n^*$); in this case, we still have a functor from a version of the idempotented form of $U_q(\mathfrak{gl}_k)$, but the relations \cref{eq:bigon2,eq:commutation-2} show that we must send $n^-$ to $N-n$ in the weight of the quantum group as in \cite[(7.4)]{queffelecMixedQuantum2019}.

\subsection{Heavy objects}
This subsection introduces heavy uprights and explains their role in the planar limit, giving a brief intuition for readers less familiar with Deligne-style limits.

We wish to expand this category by adding new objects that are ``heavy,'' in contrast with the existing ``light'' objects. Rather than an integer label, a heavy object carries a formal label $m$ in an arbitrary index set $M$\notation{$M$}{The index set that we will use for labels on heavy uprights.}, or more generally, that formal label plus an integer (of either sign).  As befits their name, these objects are represented by thicker lines than light objects.

Before discussing formalities, let us indicate the intuition behind this construction.  Our work is based on \cite[\S 2.2]{gaiottoDbranesPlanar2026}, where these heavy objects are considered in the context of the $N \to  \infty$ limit of Chern--Simons theory.  The key distinction is:
\begin{itemize}
\item \textbf{Light objects} correspond to representations $\bigwedge{}^{\!k}\C^N$ where $k$ is a fixed integer independent of $N$. The quantum dimension of such an object is a polynomial in $g=q^N$.
\item \textbf{Heavy objects} correspond to representations $\bigwedge{}^{\!\alpha N}\C^N$ where $\alpha\in (0,1)$ is fixed, so the label scales with $N$. In the planar limit, these become the ``D-branes'' of the dual string theory.
\end{itemize}
For our purposes, it is better to think of $\alpha N$ as an independent formal variable $m$, but in future work, we hope to give a more detailed analysis from the viewpoint of this scaling interpretation.  

We work over the base ring $\Bbbk=\C(q)[g^{\pm 1},q^{\pm m}]_{m\in M}$\notation{$\Bbbk$}{The base ring $\C(q)[g^{\pm 1},q^{\pm m}]$.}.  When we have a list of elements $m_1,\dots, m_k\in M$, we write $\ourmu_i=q^{m_i}$\notation{$\ourmu,\ourmu_i$}{The formal exponential $q^m,q^{m_i}$.} for simplicity of notation.   In \cite{gaiottoDbranesPlanar2026}, the symbol $\mu_i$ is used for $\ourmu_i$, but we wish to avoid confusion with the notation for knot contact homology, since a change of variables is needed to relate to the usual notation, for example in \cite{ngTopologicalIntroduction2014}.
As before, we use $N$ as just a formal symbol that reminds us to add factors of $g=q^N$.

We consider the ladders as in \cref{def:ladder1}, but now using labels from the set $M$:
\begin{definition}
	A {\bf heavy ladder} with $k$ uprights is a ladder diagram where each rung is labeled by a non-negative integer and each section of an upright is labeled by an element of $\Z_{\geq 0}$, or an element of $M+\Z$.  The labeling must satisfy the condition that the signed sum of the segments at each vertex is zero.
	Note that this means that all the labels on a single upright lie in $m+\Z$ for a fixed $m\in M$ or in $\Z_{\geq 0}$.  
\end{definition}
Let $\Pi=\Z_{\geq 0}\cup (M+\Z)$\notation{$\Pi$}{The combined index sets for light and heavy uprights $\Pi=\Z_{\geq 0}\cup (M+\Z)$.}.  
We call the uprights with labels in $M+\Z$ {\bf heavy} and those with labels in $\Z_{\geq 0}$ {\bf light}.

As we did with the integers, we consider positive and negative copies of $\Pi$ as labels on the  bottom and top of each upright of a ladder: if $r$ is the label on the adjacent portion of the upright, then we label the bottom with $r^+$ if it is oriented up and $r^-$ if it is oriented down.  For convenience, given $r\in \Pi, a\in \Z$, we let $r^{\pm}+a=(r\pm a)^{\pm}$.

\begin{definition}\label{def:SpidM}
As in the definition of $\Spid{N}$, we consider the monoidal category $\SpidM{N}$\notation{$\SpidM{N}$}{The spider category incorporating heavy uprights labeled with elements of $\Pi$.} generated by objects $\Pi^+$ and $\Pi^-$, so that a general object is a list of these generators, with tensor product given by concatenation.

	The morphisms from $(a_1^{\epsilon_1},\dots, a_k^{\epsilon_k})$ to $(b_1^{\tau_1},\dots, b_{\ell}^{\tau_{\ell}})$  for $a_i,b_i\in\Pi$ and $\epsilon_i,\tau_i\in \{+,-\}$ are generated by a fixed isomorphism $0^+\cong \emptyset$ and by heavy ladders in $\R\times [0,1]$ that satisfy:
	\begin{enumerate}
		\item The points where the ladder meets $\R\times \{0\}$ are labeled with $a_i$ read from left to right, with the edge oriented upwards if $\epsilon_i= +$ and downwards if $\epsilon_i = -$.  
		\item The points where the ladder meets $\R\times \{1\}$ are labeled with $b_i$ read from left to right, with the edge oriented upward if $\tau_i= +$ and downward if $\tau_i = -$.  
	\end{enumerate}
We can convert these ladders into graphs as in \cref{def:spider}.  The morphism space from the object $(a_1^{\epsilon_1},\dots, a_k^{\epsilon_k})$ to $(b_1^{\tau_1},\dots, b_{\ell}^{\tau_{\ell}})$ is thus spanned by the graphs arising from heavy ladders as above, modulo the local relations \crefrange{eq:loop}{eq:serre2} with any mix of heavy and light uprights, where we interpret the $q$-binomial coefficient as a formal rational function in terms of $q,g=q^N,\ourmu=q^m$ and $\ourmu'=q^{m'}$ as in \cref{eq:formal-q-binom}:
  \begin{align}
      \qBinomial{m-m'+r-s}{t}&=\frac{\prod_{i=0}^{t-1} (q^{r-s-i}\ourmu(\ourmu')^{-1}-q^{s-r+i}\ourmu^{-1}\ourmu')}{\prod_{i=1}^{t}(q^i-q^{-i})},\\    \qBinomial{N-m-m'+r-s}{t}&=\frac{\prod_{i=0}^{t-1} (gq^{r-s-i}(\ourmu\ourmu')^{-1}-g^{-1}q^{s-r+i}\ourmu\ourmu')}{\prod_{i=1}^{t}(q^i-q^{-i})}.
  \end{align}
\end{definition}
We let $\mathbf{m},\boldsymbol{\epsilon}$ be the $k$-tuples in $M$ (ignoring addition by integers) and $\{\pm\}$ such that $m_i^{\epsilon_i}$ is the label on the $i$th heavy upright.  Since we do not allow heavy labels on rungs, if these sequences are different, then the morphism space between the corresponding objects is trivial.

\begin{definition}
    Let $\Spidm{\mathbf{m},\boldsymbol{\epsilon}}{N}$\notation{$\Spidm{\mathbf{m},\boldsymbol{\epsilon}}{N}$}{The full subcategory of $\SpidM{N}$ consisting of the objects $((m_1+p_1)^{\epsilon_1},\dots(m_n+p_n)^{\epsilon_n})$.} be the full subcategory of $\SpidM{N}$ consisting of the objects $((m_1+p_1)^{\epsilon_1},\dots,(m_n+p_n)^{\epsilon_n})$ with $p_i \in \mathbb{Z}$.  The morphisms in this category can be written as ladders where all uprights are heavy.  

    Let $\Spl{\mathbf{m},\boldsymbol{\epsilon}}{N}$\notation{$\Spl{\mathbf{m},\boldsymbol{\epsilon}}{N}$}{The full subcategory of $\SpidM{N}$ consisting of the objects where the labels on heavy uprights are $((m_1+p_1)^{\epsilon_1},\dots(m_n+p_n)^{\epsilon_n})$, but additional light uprights are allowed.} be the larger category where we allow additional light uprights at the top and bottom of the diagram while leaving the same sequence of heavy uprights.  In particular, in a diagram giving a morphism in $\Spidm{\mathbf{m},\boldsymbol{\epsilon}}{N}$, if $y_0<y_1$ are generic heights (that is, neither height passes through a crossing, minimum or maximum), then the portion of a diagram between the lines $y=y_0$ and $y=y_1$ is a morphism in $\Spl{\mathbf{m},\boldsymbol{\epsilon}}{N}$.
\end{definition}

Another consequence is that we cannot allow heavy objects to braid with each other or to form cups and caps, so we no longer have a braided or rigid category once we add them.  It is an intriguing question in what context we can define such a braiding, which would seem to require an infinite sum.  
\begin{remark}\label{rem:difficulties}
To see some of the difficulties of defining a value for heavy objects, consider the case of the unknot.  Note that the argument in the introduction where we found a recursion for the quantum invariants of the unknot did not require us to assign an actual value to the unknot in the limit $N\to \infty$ and $q\to 1$; obviously, this is not possible as a single complex number, but there is an interesting way to think about it as a divergent series, discussed in \cite[\S 2.2]{gaiottoDbranesPlanar2026}.  For finite values of $N,m$, if we fix $\hbar= \frac{\log g}{N}$ and so $q=e^{\hbar}$, one obtains 
\begin{align}
	\log  \tikz[baseline]{\draw[->,very thick] (0,0.5) node[above] {$m$} arc (90:-315:0.5cm);} &=\sum_{i=0}^{m-1}\log(gq^{-i}-g^{-1}q^{i})-\log(q^{i+1}-q^{-i-1})\\
	&=\sum_{i=0}^{m-1}\log(g^{1-i/N}-g^{-1+i/N})-\log(g^{(i+1)/N}-g^{(-i-1)/N}).
\end{align}
For large $m$ and $N$, this becomes a Riemann sum approximating the integral
\[N\int_0^{m/N} \log(g^{1-x}-g^{-1+x})-\log(g^{x}-g^{-x})dx.\]
We can now treat $m$ and $N$ as continuous quantities, fix the ratio $\alpha=m/N$, and write this as
\[\log  \tikz[baseline]{\draw[->,very thick] (0,0.5) node[above] {$m$} arc (90:-315:0.5cm);}=\frac{\log g}{\hbar}\int_0^{\alpha} \log(g^{1-x}-g^{-1+x})-\log(g^{x}-g^{-x})dx.\]
Thus, considered as a holomorphic function of $\hbar$, the value of the unknot has an essential singularity at $\hbar=0$ proportional to the function $e^{1/\hbar}$, but away from this point, it is holomorphic. 

The study of how to make sense of this sort of function is a central part of resurgent analysis, so this argument suggests a connection with the study of resurgence in the context of link invariants \cite{garoufalidisResurgentStructure2021}.
\end{remark}

\begin{remark}
    One potential approach to making sense of these heavy objects is to consider an ultraproduct category.  That is, for every $\mathbf{k}\in \mathbb{Z}^M$, we consider a copy of the category $\Spid{N}$, and consider the ultraproduct of these categories with respect to a Fubini non-principal ultrafilter on $\mathbb{Z}^M$.  
    We can define light objects in this ultraproduct category by taking the same object $(a)$ in each factor of the product for all $\mathbf{k}$, and we define a heavy object with label $m$ by taking the usual object $(k_m)$ in each factor.  
    The resulting category is linear over the ultraproduct of copies of $\C[q,q^{-1}]$, and still braided and ribbon.   
    The price one pays for this is that the field of scalars is enormous---a closed heavy link has a well-defined value, but it is just the ultraproduct of the values of the HOMFLYPT polynomial of the link with different colors.

    It is an interesting question to ask whether this ultraproduct category has a natural MOY type model; the relations in $\Spidm{\mathbf{m},\boldsymbol{\epsilon}}{N}$ hold of course, but this does not account for skein type relations involving heavy tangles.  In particular, since the tensor product $\bigwedge{}^{\!a}\C^N\otimes \bigwedge{}^{\!b}\C^N$ is multiplicity-free for all $a,b,N$, we have relations that arise from mutation of links.  This shows that two tangles that differ by a mutation give the same morphism in the ultraproduct category by \cite[Th. 5]{mortonDistinguishingMutants1996}.
\end{remark}

On the other hand, we can make sense of the braiding of heavy and light objects by extending the equations \crefrange{eq:braid}{eq:braid-op} when one of the strands in the formula is light.  Let us record the result in the case where the light strand has label $1$:
\begin{align}
  \begin{tikzpicture}[baseline=23,scale=.6]
    \node (C) at (-1,3) {$m$};
    \node (D) at (1,0) {$m$};
    \node (A) at (1,3)  {$1$};
    \node (B) at (-1,0){$1$};
    \node (i) at (intersection of A--B and D--C) {};
    \draw[mid>,very thick] (D) -- (C);
    \draw[mid>] (B) -- (i);
    \draw[mid>] (i) -- (A);
  \end{tikzpicture}
  &= 
  \begin{tikzpicture}[baseline=23,scale=.6]
    \laddercoordinates{1}{2}
    \node[left]  at (l00) {$1$};
    \node[right] at (l10) {$m$};
    \node[left]  at (l02) {$m$};
    \node[right] at (l12) {$1$};
     \node[inner sep=0pt] (i) at (intersection of l00--l11 and l10--l02) {};
          \node[inner sep=0pt] (j) at (intersection of l01--l12 and l10--l02) {};
    \draw[mid>] (l00) -- (i);
    \draw[mid>] (j) -- (l12);
    \draw[mid>, very thick] (l10)-- (l02);
  \end{tikzpicture}
   -q^{-1}
  \begin{tikzpicture}[baseline=23,scale=.6]
    \laddercoordinates{1}{2}
    \node[left]  at (l00) {$1$};
    \node[right] at (l10) {$m$};
    \node[left]  at (l02) {$m$};
    \node[right] at (l12) {$1$};
     \node[inner sep=0pt] (i) at (intersection of l00--l11 and l10--l02) {};
          \node[inner sep=0pt] (j) at (intersection of l01--l12 and l10--l02) {};
              \draw[mid>] (l00) -- (j);
    \draw[mid>] (i) -- (l12);
    \draw[mid>, very thick] (l10)-- (l02);
  \end{tikzpicture}\;,   \label{eq:thick-braid1}\\   \begin{tikzpicture}[baseline=23,scale=.6]
    \node (C) at (1,3) {$m$};
    \node (D) at (-1,0) {$m$};
    \node (A) at (-1,3)  {$1$};
    \node (B) at (1,0){$1$};
    \node (i) at (intersection of A--B and D--C) {};
    \draw[mid>,very thick] (D) -- (C);
    \draw[mid>] (B) -- (i);
    \draw[mid>] (i) -- (A);
  \end{tikzpicture}
  &= 
  \begin{tikzpicture}[baseline=23,yscale=.6,xscale=-.6]
    \laddercoordinates{1}{2}
    \node[right]  at (l00) {$1$};
    \node[left] at (l10) {$m$};
    \node[right]  at (l02) {$m$};
    \node[left] at (l12) {$1$};
     \node[inner sep=0pt] (i) at (intersection of l00--l11 and l10--l02) {};
          \node[inner sep=0pt] (j) at (intersection of l01--l12 and l10--l02) {};
    \draw[mid>] (l00) -- (i);
    \draw[mid>] (j) -- (l12);
    \draw[mid>, very thick] (l10)-- (l02);
  \end{tikzpicture}
   -q
     \begin{tikzpicture}[baseline=23,yscale=.6,xscale=-.6]
    \laddercoordinates{1}{2}
    \node[right]  at (l00) {$1$};
    \node[left] at (l10) {$m$};
    \node[right]  at (l02) {$m$};
    \node[left] at (l12) {$1$};
     \node[inner sep=0pt] (i) at (intersection of l00--l11 and l10--l02) {};
          \node[inner sep=0pt] (j) at (intersection of l01--l12 and l10--l02) {};
              \draw[mid>] (l00) -- (j);
    \draw[mid>] (i) -- (l12);
    \draw[mid>, very thick] (l10)-- (l02);
  \end{tikzpicture} \;,
  \label{eq:thick-braid2}
    \\
  \begin{tikzpicture}[baseline=23,scale=.6]
    \node (C) at (-1,3) {$1$};
    \node (D) at (1,0) {$1$};
    \node (A) at (1,3)  {$m$};
    \node (B) at (-1,0){$m$};
    \node (i) at (intersection of A--B and D--C) {};
    \draw[mid>] (D) -- (C);
    \draw[mid>, very thick] (B) -- (i);
    \draw[mid>, very thick] (i) -- (A);
  \end{tikzpicture}
  &=   \begin{tikzpicture}[baseline=23,yscale=.6,xscale=-.6]
    \laddercoordinates{1}{2}
    \node[right]  at (l00) {$1$};
    \node[left] at (l10) {$m$};
    \node[right]  at (l02) {$m$};
    \node[left] at (l12) {$1$};
     \node[inner sep=0pt] (i) at (intersection of l00--l11 and l10--l02) {};
          \node[inner sep=0pt] (j) at (intersection of l01--l12 and l10--l02) {};
    \draw[mid>] (l00) -- (i);
    \draw[mid>] (j) -- (l12);
    \draw[mid>, very thick] (l10)-- (l02);
  \end{tikzpicture}
   -q^{-1}
  \begin{tikzpicture}[baseline=23,yscale=.6,xscale=-.6]
    \laddercoordinates{1}{2}
    \node[right]  at (l00) {$1$};
    \node[left] at (l10) {$m$};
    \node[right]  at (l02) {$m$};
    \node[left] at (l12) {$1$};
     \node[inner sep=0pt] (i) at (intersection of l00--l11 and l10--l02) {};
          \node[inner sep=0pt] (j) at (intersection of l01--l12 and l10--l02) {};
              \draw[mid>] (l00) -- (j);
    \draw[mid>] (i) -- (l12);
    \draw[mid>, very thick] (l10)-- (l02);
  \end{tikzpicture} \;,
 \label{eq:thick-braid3} \\   \begin{tikzpicture}[baseline=23,scale=.6]
    \node (C) at (1,3) {$1$};
    \node (D) at (-1,0) {$1$};
    \node (A) at (-1,3)  {$m$};
    \node (B) at (1,0){$m$};
    \node (i) at (intersection of A--B and D--C) {};
    \draw[mid>] (D) -- (C);
    \draw[mid>,very thick] (B) -- (i);
    \draw[mid>,very thick] (i) -- (A);
  \end{tikzpicture}
  &= 
  \begin{tikzpicture}[baseline=23,scale=.6]
    \laddercoordinates{1}{2}
    \node[left]  at (l00) {$1$};
    \node[right] at (l10) {$m$};
    \node[left]  at (l02) {$m$};
    \node[right] at (l12) {$1$};
     \node[inner sep=0pt] (i) at (intersection of l00--l11 and l10--l02) {};
          \node[inner sep=0pt] (j) at (intersection of l01--l12 and l10--l02) {};
    \draw[mid>] (l00) -- (i);
    \draw[mid>] (j) -- (l12);
    \draw[mid>, very thick] (l10)-- (l02);
  \end{tikzpicture}
   -q
  \begin{tikzpicture}[baseline=23,scale=.6]
    \laddercoordinates{1}{2}
    \node[left]  at (l00) {$1$};
    \node[right] at (l10) {$m$};
    \node[left]  at (l02) {$m$};
    \node[right] at (l12) {$1$};
     \node[inner sep=0pt] (i) at (intersection of l00--l11 and l10--l02) {};
          \node[inner sep=0pt] (j) at (intersection of l01--l12 and l10--l02) {};
              \draw[mid>] (l00) -- (j);
    \draw[mid>] (i) -- (l12);
    \draw[mid>, very thick] (l10)-- (l02);
  \end{tikzpicture} \;. 
  \label{eq:thick-braid4}
\end{align}
In particular, we have a thick version of the skein relation \eqref{eq:skein}:
\begin{align}\label{eq:thick-skein1}
	\begin{tikzpicture}[baseline=23,scale=.6]
    \node (C) at (1,3) {$m$};
    \node (D) at (-1,0) {$m$};
    \node (A) at (-1,3)  {$1$};
    \node (B) at (1,0){$1$};
    \node (i) at (intersection of A--B and D--C) {};
    \draw[mid>,very thick] (D) -- (C);
    \draw[mid>] (B) -- (i);
    \draw[mid>] (i) -- (A);
  \end{tikzpicture} -   \begin{tikzpicture}[baseline=23,scale=.6]
    \node (C) at (-1,3) {$1$};
    \node (D) at (1,0) {$1$};
    \node (A) at (1,3)  {$m$};
    \node (B) at (-1,0){$m$};
    \node (i) at (intersection of A--B and D--C) {};
    \draw[mid>] (D) -- (C);
    \draw[mid>, very thick] (B) -- (i);
    \draw[mid>, very thick] (i) -- (A);
  \end{tikzpicture}&=(q^{-1}-q)    \begin{tikzpicture}[baseline=23,yscale=.6,xscale=-.6]
    \laddercoordinates{1}{2}
    \node[right]  at (l00) {$1$};
    \node[left] at (l10) {$m$};
    \node[right]  at (l02) {$m$};
    \node[left] at (l12) {$1$};
     \node[inner sep=0pt] (i) at (intersection of l00--l11 and l10--l02) {};
          \node[inner sep=0pt] (j) at (intersection of l01--l12 and l10--l02) {};
              \draw[mid>] (l00) -- (j);
    \draw[mid>] (i) -- (l12);
    \draw[mid>, very thick] (l10)-- (l02);
  \end{tikzpicture} \;,\\\label{eq:thick-skein2}
  \begin{tikzpicture}[baseline=23,scale=.6]
    \node (C) at (1,3) {$1$};
    \node (D) at (-1,0) {$1$};
    \node (A) at (-1,3)  {$m$};
    \node (B) at (1,0){$m$};
    \node (i) at (intersection of A--B and D--C) {};
    \draw[mid>] (D) -- (C);
    \draw[mid>,very thick] (B) -- (i);
    \draw[mid>,very thick] (i) -- (A);
  \end{tikzpicture} -   \begin{tikzpicture}[baseline=23,scale=.6]
    \node (C) at (-1,3) {$m$};
    \node (D) at (1,0) {$m$};
    \node (A) at (1,3)  {$1$};
    \node (B) at (-1,0){$1$};
    \node (i) at (intersection of A--B and D--C) {};
    \draw[mid>,very thick] (D) -- (C);
    \draw[mid>] (B) -- (i);
    \draw[mid>] (i) -- (A);
  \end{tikzpicture}&=  (q^{-1}-q) \begin{tikzpicture}[baseline=23,scale=.6]
    \laddercoordinates{1}{2}
    \node[left]  at (l00) {$1$};
    \node[right] at (l10) {$m$};
    \node[left]  at (l02) {$m$};
    \node[right] at (l12) {$1$};
     \node[inner sep=0pt] (i) at (intersection of l00--l11 and l10--l02) {};
          \node[inner sep=0pt] (j) at (intersection of l01--l12 and l10--l02) {};
              \draw[mid>] (l00) -- (j);
    \draw[mid>] (i) -- (l12);
    \draw[mid>, very thick] (l10)-- (l02);
  \end{tikzpicture} \;. 
\end{align}

This makes $\SpidM{N}$ into a monoidal bimodule category over $\Spid{N}$ with a braiding relating the left and right actions.  Since the functors of horizontal composition on the left or right by a given object are isomorphic, we can package this information into just the left action of $\Spid{N}$, which is a braided monoidal action in the sense of \cite[Def. 3.4]{ben-zviQuantumCharacter2018}, with the automorphism $E$ given by winding the object of $\Spid{N}$ around the thick strand in the positive direction.
As discussed in \cite[Rmk. 3.6]{ben-zviQuantumCharacter2018}, this notion depends on a framing of the annulus, which gives us the framing of the thin strands as they wrap around the thick one.   However, the ribbon structure of $\Spid{N}$ allows us to relate all these different framings by adding in an appropriate number of full twists to the object in $\Spid{N}$ as it wraps around.

Returning to thinking of left and right actions, the  braiding structure of the action exactly captures the fact that there are two isomorphisms between the left and right actions, corresponding to the two different ways of braiding the heavy and light objects.
	
\begin{remark}\label{rem:skein-algebra}
From a ribbon category like $\Spid{N}$, we can produce an invariant of surfaces called the \emph{skein category}.  This is a natural extension of the skein algebras; it seems to have first appeared in written form in unpublished work of Walker. It was shown by Cooke in \cite{cookeExcisionSkein2023} that we can interpret this category as the factorization homology of the ribbon category on the surface.

We should be able to interpret our ``heavy strands'' as taking the factorization homology with marked points as discussed in  \cite[\S 5.1]{ben-zviQuantumCharacter2018}, considering the case with a single heavy upright as a braided module category over $\Spid{N}$, which we then put on each marked point.  

The only case that will be important for us is when the surface is $\R^2$ or $S^2$.  In the former case, we obtain the subcategory of $\SpidM{N}$ where we fix  $k$ heavy uprights with labels in $m_i^{\epsilon_i}+\Z$ for fixed $m_i\in M, \epsilon_i\in \{+,-\}$ and then add in arbitrary light uprights; in the latter case, we obtain a quotient of this category by specializing the central elements of the universal enveloping algebra to the scalars given by the ``charge at $\infty$'' in \cite[\S 3.1]{gaiottoDbranesPlanar2026}.
\end{remark}

It is straightforward to give a basis of the morphism spaces in $\Spidm{\mathbf{m},\boldsymbol{\epsilon}}{N}$, based on the PBW theorem. First, recall that we can give the quantum group $U_q(\mathfrak{gl}_n;\Bbbk)$ a grading by the group $\Z^n$, where the homogeneous component for $(w_1,\dots, w_n)\in \Z^n$ is the set of elements of this weight under the adjoint action of the Cartan---let \[U_{(w_1,\dots,w_n)}=\{u\in U_q(\mathfrak{gl}_n;\Bbbk) \mid K_iuK_i^{-1}=q^{w_i}u \text{ for }i=1,\dots, n\}.\]  

Let $q^{r^+}=q^{r}$ and $q^{r^-}=gq^{-r}$.  For a fixed $n$-tuple $(r_1,\dots, r_n)$ with $r_i\in \Pi^{\pm}$, consider the map $\Bbbk[\mathbf{K}^{\pm 1}]=\Bbbk[K_1^{\pm 1},\dots, K_n^{\pm 1}]\to \Bbbk$ sending $K_i\mapsto q^{r_i}$; let $\Bbbk_{\Br}$ be $\Bbbk$, thought of as a $\Bbbk[\mathbf{K}^{\pm 1}]$-module.

\begin{definition}
    Let $\dot U_{\mathbf{m},\boldsymbol{\epsilon}}:=\dot U_{\mathbf{m},\boldsymbol{\epsilon}}(\mathfrak{gl}_n;\Bbbk)$\notation{$\dot U_{\mathbf{m},\boldsymbol{\epsilon}}$}{The extension of the idempotented quantum group to include the objects $((m_1+p_1)^{\epsilon_1},\dots,(m_n+p_n)^{\epsilon_n})$.} be the category whose objects are $n$-tuples $\Br=(r_1,\dots, r_n)$ with $r_i\in m_i^{\epsilon_i}+\Z$, and where  
    \[ \Hom(\Br,\Br')=\Bbbk_{\Br'}\otimes_{\Bbbk[\mathbf{K}^{\pm 1}]}U_q(\mathfrak{gl}_n;\Bbbk)\otimes_{\Bbbk[\mathbf{K}^{\pm 1}]} \Bbbk_{\Br}\cong 
        U_{\Br'-\Br}\otimes_{\Bbbk[\mathbf{K}^{\pm 1}]} \Bbbk_{\Br}. \]
  Composition is induced by multiplication in $U_q(\mathfrak{gl}_n;\Bbbk)$.
\end{definition}

  By a natural extension of \cite[Prop. 7.2 \& 7.5]{queffelecMixedQuantum2019}, we see that we have:
\begin{lemma}\label{Uq-injective}
	There is an equivalence of categories  \[\dot U_{\mathbf{m},\boldsymbol{\epsilon}}\to \Spidm{\mathbf{m},\boldsymbol{\epsilon}}{N}\] 
    sending $\Br$ to the same $n$-tuple, thought of as an object in $\SpidM{N}$, and sending $E_i$ and $F_i$ to the usual diagrams 
	\[
	  \label{eq:E-F}
    E_i\mapsto \begin{tikzpicture}[baseline=20]
    \laddercoordinates{1}{1}
    \node[left]  at (l00) {};
    \node[right] at (l10) {};
    \TladderFd{0}{0}{}{}{$1$}{}{}
  \end{tikzpicture} \;,
  \qquad  F_i\mapsto 
  \begin{tikzpicture}[baseline=20]
    \laddercoordinates{1}{2}
    \node[left]  at (l00) {};
    \node[right] at (l10) {};
    \TladderEd{0}{0}{}{}{$1$}{}{}
  \end{tikzpicture} \;.
	\]
\end{lemma}
\begin{proof}
  The verification of the quantum group relations is exactly as in \cite[Prop. 7.2]{queffelecMixedQuantum2019}, so we obtain a functor. This functor is full because any ladder can be written in terms of rungs with label $1$; rungs with other integer labels are images of divided powers $E_i^{(r)},F_i^{(r)}$.  It is also essentially surjective by direct construction.
 Thus, the only nontrivial part of the proof is faithfulness.  

 We can always choose integers $a_m$, and we obtain a functor $\SpidM{N}\to \Spid{N}$ to the spider category which sends $(m+k)^{\pm}\to (a_m+k)^\pm$ if $a_m+k\geq 0$ and otherwise sends this object to $0$.  Of course, this functor acts on diagrams by killing all diagrams containing heavy uprights with a label $m+k$ where $a_m+k<0$, and otherwise turning heavy uprights with the label $m+k$ for $m\in M, k\in \Z$ into light uprights with label $a_m+k$. 

Similarly, we have a functor from $\dot U_{\mathbf{m},\boldsymbol{\epsilon}}$ to the mixed Schur algebra $\dot{U}_q(\mathfrak{gl}_{\boldsymbol{\epsilon}})_g$ as defined in \cite[Prop. 7.1]{queffelecMixedQuantum2019}.  We thus have a commutative diagram:
\[\begin{tikzcd}
	{\dot{U}_{\mathbf{m},\boldsymbol{\epsilon}}} & {\SpidM{N}} \\
	{\dot{U}_q(\mathfrak{gl}_{\boldsymbol{\epsilon}})_g} & {\Spid{N}}
	\arrow[from=1-1, to=1-2]
	\arrow[from=1-1, to=2-1]
	\arrow[from=1-2, to=2-2]
	\arrow[from=2-1, to=2-2]
\end{tikzcd} \;.\]
    
    Consider the functor from top left to bottom right.  No fixed choice of $a_m$ gives a faithful functor (as is clear from the definition), but the family over all choices of $a_m$ is faithful.  By \cite[Prop. 7.11]{queffelecMixedQuantum2019}, the bottom arrow is faithful, so we only need to consider the kernel of the left arrow.  Thus, we need to show that for any non-zero element $x\in U_{\Br'-\Br}\otimes_{\Bbbk[\mathbf{K}^{\pm 1}]} \Bbbk_{\Br}$, there exists a choice of $a_m$ such that the image of $x$ in $\dot{U}_q(\mathfrak{gl}_{\boldsymbol{\epsilon}})_g$ is non-zero.

    By \cite[Prop. 7.7]{queffelecMixedQuantum2019}, for any $t\in\Z_{\geq 0}$, we have an action of $\dot{U}_q(\mathfrak{gl}_{\boldsymbol{\epsilon}})_g$ on a direct sum $T$ of tensor products of wedge powers of $\C(q)^t$ and its dual, depending on the sign sequence $\boldsymbol{\epsilon}$, and commuting with the action of $U_q(\mathfrak{sl}_t)$.  
    This action is often called ``skew Howe duality.'' 
    We can understand its kernel by decomposing $T$ into irreducible representations of $\dot{U}_q(\mathfrak{gl}_n)$.  
    We can thus confirm that an element $x$ is non-zero by showing it acts non-trivially on some irreducible summand of $T$.
   The irreducible representations of $\dot{U}_q(\mathfrak{gl}_n)$ that appear have a reasonable combinatorial description: we obtain exactly the irreps corresponding to highest weights $\lambda$ satisfying $0\leq \lambda_1\leq \cdots \leq \lambda_n\leq t$. The identity  $1_\Br$ acts by projection to a weight space for the action of $\dot{U}_q(\mathfrak{gl}_n)$, induced by the specialization $g=q^t,q^{m}=q^{a_m}$.

    Thus, we have reduced to the question of showing that no nonzero element $x\in U_{\Br'-\Br}\otimes_{\Bbbk[\mathbf{K}^{\pm 1}]} \Bbbk_{\Br}$ maps to zero under this specialization for all possible choices of $t,a_m$.  So, for a given $x$, we need to find some $t$ and $a_m$ such that $x$ acts non-trivially on some irreducible representation with highest weight $\lambda$ satisfying $0\leq \lambda_1\leq \cdots \leq \lambda_n\leq t$.
    We can consider the analogue of $U_{\Br}$ in the subalgebras $U_q^{\pm}$ generated by $E_i$'s and $F_i$'s respectively.    That is:
    \[U^{\pm}_{(w_1,\dots,w_n)}=\{u\in U_q^{\pm} \mid K_iuK_i^{-1}=q^{w_i}u\}.\]  First consider $x=x'x''$ for $x'\in U_{\Bw'}^{+}$ and $x''\in U_{\Bw''}^-$ for some $\Bw',\Bw''$.  We can choose a dominant weight $\lambda'$ satisfying $0\leq \lambda_1'\leq \cdots \leq \lambda_n'\leq t'$ for some $t'$ where $x'$ acts non-trivially on the highest weight vector of the corresponding highest-weight irrep $V_{\lambda'}$, and similarly a lowest weight $\lambda''$ satisfying $0\leq \lambda_n''\leq \cdots \leq \lambda_1''\leq t''$ where $x''$ acts non-trivially on the lowest weight vector of the lowest weight irrep $\Lambda_{\lambda''}$.  One can verify that this means that the action of $x=x'x''$ on their tensor product in $V_{\lambda'}\otimes \Lambda_{\lambda''}$ is non-zero.  More generally, if we write $x=\sum x_i'x_i''$ for $x_i'\in U_{\Bw_i'}^{+}$ and $x_i''\in U_{\Bw_i''}^-$, we can choose a single $t'$ and $t''$ such that all non-zero linear combinations of the $x_i'$ act non-trivially on the highest weight vector of some $V_{\lambda'}$, and similarly all non-zero linear combinations of the $x_i''$ act non-trivially on the lowest weight vector of some $\Lambda_{\lambda''}$.  This shows that the action of $x$ on the tensor product of these two irreps is non-zero.
    Choosing $t\geq t'+t''$ ensures all summands of this tensor product appear in $T$, so $x$ acts non-trivially on $T$.  
\end{proof}

We can describe the resulting basis of Hom spaces more geometrically.  Fix an object $(m_1^{\epsilon_1},\dots, m_n^{\epsilon_n})$.  Let $\alpha_{ij}=(\delta_{jk}-\delta_{ik})\in \Z^{n}$ be the usual root vector (so $\alpha_{12}=(-1,1,0,\dots)$, etc.).  For non-negative integers $b_{ij}\in \Z_{\geq 0}$ with $i\neq j\in \{1,\dots, n\}$, collected into an $n\times n$ matrix $B$ with zero diagonal, we set $w(B)=\sum_{i,j}b_{ij}\alpha_{ij}$.  

Consider the ladder diagrams \begin{equation}
\cij =\begin{cases} \begin{tikzpicture}[baseline=20]
	\laddercoordinates{4}{1}
	\node (A) at (l00) {$r_i$};
	\node (B) at (l01) {$r_i-1$};
	\node (C) at (l30) {$r_j$};
	\node (D) at (l31) {$r_j+1$};
	\node (A1) at (l10) {\phantom{$m_i$}};
	\node (B1) at (l11) {\phantom{$m_i$}};
	\node (C1) at (l20) {\phantom{$m_i$}};
	\node (D1) at (l21) {\phantom{$m_i$}};
	\node[inner sep=0pt] (i) at (0,.5 *\ladderY){};
\node[inner sep=0pt] (j) at (3*\ladderX,.5 *\ladderY){};
	\node[inner sep=5pt] (i1) at (intersection of A1--B1 and i--j){};
		\node[inner sep=5pt] (i2) at (intersection of C1--D1 and i--j){};
	\draw[mid>,very thick] (A)--(B);
	\draw[mid>,very thick] (C)--(D); 
		\draw[mid>,very thick] (A1)--(B1);
	\draw[mid>,very thick] (C1)--(D1); 
	\draw[mid>] (i)--(i1);
	\draw[mid>] (i1)--(i2);
		\draw[mid>] (i2)--(j);	
\end{tikzpicture} & i < j\\
-\begin{tikzpicture}[baseline=20]
	\laddercoordinates{4}{1}
	\node (A) at (l00) {$r_j$};
	\node (B) at (l01) {$r_j+1$};
	\node (C) at (l30) {$r_i$};
	\node (D) at (l31) {$r_i-1$};
	\node (A1) at (l10) {\phantom{$m_i$}};
	\node (B1) at (l11) {\phantom{$m_i$}};
	\node (C1) at (l20) {\phantom{$m_i$}};
	\node (D1) at (l21) {\phantom{$m_i$}};
	\node[inner sep=3pt] (i) at (0,.5 *\ladderY){};
\node[inner sep=3pt] (j) at (3*\ladderX,.5 *\ladderY){};
\node[inner sep=0pt] (ii) at (0, .5*\ladderY-.3){};
\node[inner sep=0pt] (jj) at (3*\ladderX,.5 *\ladderY+.3){};
	\node[inner sep=5pt] (i1) at (intersection of A1--B1 and i--j){};
		\node[inner sep=5pt] (i2) at (intersection of C1--D1 and i--j){};
	\draw[mid>,very thick] (A)--(i) ;
	\draw[mid>,very thick] (i)--(B);
	\draw[mid>,very thick] (C)--(D); 
		\draw[mid>,very thick] (A1)--(B1);
	\draw[mid>,very thick] (C1)--(D1); 
	\draw[mid<] (ii) to[out=180,in=-90] (-.3,.5 *\ladderY-.15) to[out=90,in=180] (0,.5 *\ladderY) to (i1);
	\draw[mid>] (jj) to[out=0,in=90] (3*\ladderX + .3,.5 *\ladderY+.15) to[out=-90,in=0]  (j);
	\draw[mid<] (i1)--(i2);
		\draw[mid<] (i2)--(j);	
\end{tikzpicture} & j < i. 
\end{cases}
\end{equation}
For a matrix $B$ as above, consider the morphism from $(r_1,\dots, r_n)$ to $(r_1+w_1(B),\dots, r_n+w_n(B))$ given by 
\[c(B)=\prod_{i=1}^n\prod_{j=1}^n \cij^{b_{ij}}.\]
It follows from the PBW theorem and \cref{Uq-injective} that:
\begin{corollary}\label{cor:spider-basis}
   The space of morphisms \[\Hom_{\Spidm{\mathbf{m},\boldsymbol{\epsilon}}{N}}((r_1,\dots, r_n), (s_1,\dots, s_n))\] has a basis over $\Bbbk$ given by the diagrams $c(B)$ for $B$ as above satisfying $w(B)=(s_1-r_1,\dots, s_n-r_n)$.
\end{corollary}

\subsection{Braids and cups}\label{sec:braid}

There is an important structure arising from braiding heavy strands, even though we do not obtain a braided monoidal structure. 
Braids naturally act on $n$-tuples $\mathbf{m},\boldsymbol{\epsilon}$ by permutation, and in \cite[Lem. 7.3]{queffelecMixedQuantum2019}, Queffelec and Sartori show that there is a braid equivalence $\mathsf{T}_{\beta}\colon \dot U_{\mathbf{m},\boldsymbol{\epsilon}}\to \dot U_{\beta\mathbf{m},\beta\boldsymbol{\epsilon}}$ based on Lusztig's braid group action on the quantum group.  
We can therefore transport this to an equivalence $\mathsf{T}_{\beta}\colon\Spidm{\mathbf{m},\boldsymbol{\epsilon}}{N}\to \Spidm{\beta\mathbf{m},\beta\boldsymbol{\epsilon}}{N}$\notation{$\mathsf{T}_{\beta}$}{The equivalence $\mathsf{T}_{\beta}\colon\Spidm{\mathbf{m},\boldsymbol{\epsilon}}{N}\to \Spidm{\beta\mathbf{m},\beta\boldsymbol{\epsilon}}{N}$ induced by a braid $\beta$.}. Geometrically, this amounts to conjugating a ladder diagram by braiding the heavy uprights and then straightening the uprights.  This creates crossings between rungs and uprights, which we convert into ladders using \crefrange{eq:thick-skein1}{eq:thick-skein2}; this reproduces the formulas of \cite[(7.5)]{queffelecMixedQuantum2019} up to sign.  

Note that if we interpret our category $\Spl{\mathbf{m},\boldsymbol{\epsilon}}{N}$ as factorization homology with marked points as in \cref{rem:skein-algebra}, then this is just the action of the mapping class group of the marked surface on this factorization homology.  

In fact, this action is most naturally viewed as an action of framed braids, where we may add any number of twists to each strand.  This simply multiplies each diagram by a power of $-g\ourmu_i^{-2}$, but it is useful for bookkeeping.

Similarly, while we cannot define a duality between heavy strands like we can between light strands with labels $k^{\pm}$, an important shadow of this structure exists in the form of modules over the category $\Spidm{\mathbf{m},\boldsymbol{\epsilon}}{N}$. 

Let us first describe this in the case of a single cup.  

We consider the category $\dot U_{\mathbf{m},\boldsymbol{\epsilon}}$ for $\mathbf{m}=(m,m)$ and $\boldsymbol{\epsilon}=(+,-)$.  In this case, there is a module $C_m$\notation{$C_m$}{The module over $\Spidm{(m,m),(+,-)}{N}$ corresponding to a single cup.} over $\dot U_{\mathbf{m},\boldsymbol{\epsilon}}$ sending the object $(a,b)$ to a 1-dimensional vector space with basis $v_p$ if $(a,b)=((p+m)^+,(p+m)^-)$ and to 0 otherwise, with the action:
\begin{equation}\label{eq:cup-action}
	E v_p = \frac{q^{p+1}\ourmu-q^{-p-1}\ourmu^{-1}}{q-q^{-1}} v_{p+1},\qquad F v_p = \frac{gq^{-p+1}\ourmu^{-1}-g^{-1}q^{p-1}\ourmu}{q-q^{-1}} v_{p-1}.
\end{equation} 

We visualize this more simply by representing $v_p$ as a single cup joining uprights with labels $(p+m)^+$ and $(p+m)^-$.  Thus the actions of $E$ and $F$ come from stringing a single light strand across the top of the cup either oriented left to right or the opposite: 
\begin{gather}\label{E-cup-action}
	\begin{tikzpicture}[baseline=20,xscale=1.3]
	\laddercoordinates{1}{1}
	\node[inner sep=0pt] (A) at (.5*\ladderX,-.5) {};
	\draw[mid>,very thick] (u00) -- ($(l00)+(0,\ladderY)$) node[above, scale=.75] {$(m+p+1)^+$};
\draw[mid<,very thick] ($(d00)+(\ladderX,0)$) -- ($(l00)+(\ladderX,\ladderY)$) node[above, scale=.75] {$(m+p+1)^-$};
\draw[mid>] ($(d00)+(\ladderX,0)$) -- (u00);
\draw[mid>, very thick] ($(d00)+(\ladderX,0)$) to (l10) to[out=-90,in=0](A) to[out=180,in=-90] (l00) -- (u00);
\end{tikzpicture}=  \frac{q^{p+1}\ourmu-q^{-p-1}\ourmu^{-1}}{q-q^{-1}} \quad		\begin{tikzpicture}[baseline=20,xscale=1.3]
	\laddercoordinates{1}{1}
	\node[inner sep=0pt] (A) at (.5*\ladderX,-.5) {};
	\draw[mid>,very thick] (u00) -- ($(l00)+(0,\ladderY)$) node[above, scale=.75] {$(m+p+1)^+$};
\draw[mid<,very thick] ($(d00)+(\ladderX,0)$) -- ($(l00)+(\ladderX,\ladderY)$) node[above, scale=.75] {$(m+p+1)^-$};
	 \draw[mid>, very thick] ($(d00)+(\ladderX,0)$) to (l10) to[out=-90,in=0](A) to[out=180,in=-90] (l00) -- (u00);
\end{tikzpicture} \;,\\ \label{F-cup-action}
	\begin{tikzpicture}[baseline=20,xscale=1.3]
	\laddercoordinates{1}{1}
	\node[inner sep=0pt] (A) at (.5*\ladderX,-.5) {};
	\draw[mid>,very thick] (u00) -- ($(l00)+(0,\ladderY)$) node[above, scale=.75] {$(m+p-1)^+$};
\draw[mid<,very thick] ($(d00)+(\ladderX,0)$) -- ($(l00)+(\ladderX,\ladderY)$) node[above, scale=.75] {$(m+p-1)^-$};
\draw[mid<] ($(u00)+(\ladderX,0)$) -- (d00);
\draw[mid>, very thick] ($(d00)+(\ladderX,0)$) to (l10) to[out=-90,in=0](A) to[out=180,in=-90] (l00) -- (u00);
\end{tikzpicture}=  \frac{gq^{-p+1}\ourmu^{-1}-g^{-1}q^{p-1}\ourmu}{q-q^{-1}} \quad		\begin{tikzpicture}[baseline=20,xscale=1.3]
	\laddercoordinates{1}{1}
	\node[inner sep=0pt] (A) at (.5*\ladderX,-.5) {};
	\draw[mid>,very thick] (u00) -- ($(l00)+(0,\ladderY)$) node[above, scale=.75] {$(m+p-1)^+$};
\draw[mid<,very thick] ($(d00)+(\ladderX,0)$) -- ($(l00)+(\ladderX,\ladderY)$) node[above, scale=.75] {$(m+p-1)^-$};
\draw[mid>, very thick] ($(d00)+(\ladderX,0)$) to (l10) to[out=-90,in=0](A) to[out=180,in=-90] (l00) -- (u00);
\end{tikzpicture} \;.
\end{gather}
We can transport this to a module over $\Spidm{\mathbf{m},\boldsymbol{\epsilon}}{N}$ by the equivalence of \cref{Uq-injective}.  
It will be important for us to extend this to a module $\widetilde{C}_m$ over $\Spl{\mathbf{m},\boldsymbol{\epsilon}}{N}$.  Geometrically, this means that we allow ourselves to add light uprights with label $0$ on either side of the diagram or in the middle of the cup, and allow rungs to join them from the heavy cup or to form light cups and caps.  We impose the local relations of \cref{def:SpidM} on these, and furthermore, allow light strands to isotope across the minimum of the heavy cup:
\begin{equation}
    \label{eq:cup-isotope}
	\begin{tikzpicture}[baseline=20,xscale=1.3]
	\laddercoordinates{1}{1}
	\node[inner sep=0pt] (A) at (.5*\ladderX,-.5) {};
	\draw[mid>,very thick] (u00) -- ($(l00)+(0,\ladderY)$) node[above, scale=.75] {$(m+p+1)^+$};
\draw[mid<,very thick] ($(d00)+(\ladderX,0)$) -- ($(l00)+(\ladderX,\ladderY)$) node[above, scale=.75] {$(m+p+1)^-$};
\draw[mid>] ($(d00)+(\ladderX,-.5*\ladderY)$) -- ($(u00)+({.5*\ladderX},-.25*\ladderY)$)--($(l00)+({.5*\ladderX},\ladderY)$);
\draw[mid>, very thick] ($(d00)+(\ladderX,0)$) to (l10) to[out=-90,in=0](A) to[out=180,in=-90] (l00) -- (u00);
\end{tikzpicture}= 		\begin{tikzpicture}[baseline=20,xscale=1.3]
	\laddercoordinates{1}{1}
	\node[inner sep=0pt] (A) at (.5*\ladderX,-.5) {};
    \draw[mid>] ($(d00)+(0,-.5*\ladderY)$) -- ($(u00)+({.5*\ladderX},-.25*\ladderY)$)--($(l00)+({.5*\ladderX},\ladderY)$);
	\draw[mid>,very thick] (u00) -- ($(l00)+(0,\ladderY)$) node[above, scale=.75] {$(m+p+1)^+$};
\draw[mid<,very thick] ($(d00)+(\ladderX,0)$) -- ($(l00)+(\ladderX,\ladderY)$) node[above, scale=.75] {$(m+p+1)^-$};
	 \draw[mid>, very thick] ($(d00)+(\ladderX,0)$) to (l10) to[out=-90,in=0](A) to[out=180,in=-90] (l00) -- (u00);
\end{tikzpicture} \;.
\end{equation}
The relations \cref{E-cup-action,F-cup-action} are a consequence of these relations and \cref{eq:bigon2}.

We can also horizontally compose several of these cups to obtain a module $\widetilde{C}_{m_1,\dots, m_k}$ over $\Spl{\mathbf{m},\boldsymbol{\epsilon}}{N}$ for $\mathbf{m}=(m_1,m_1,m_2,m_2,\dots, m_k,m_k)$ for some choice of $m_1,\dots, m_k$ and $\boldsymbol{\epsilon}=(+,-,+,-,\cdots,+,-)$.
That is, we take the external tensor product of the modules $\widetilde{C}_{m_i}$ over the categories $\mathsf{Spl}_{m_i}=\Spl{(m_i,m_i),(+,-)}{N}$ and then consider the induced module
\begin{equation*}
\widetilde{C}_{m_1,\dots, m_k}=\Spl{\mathbf{m},\boldsymbol{\epsilon}}{N}\otimes_{\mathsf{Spl}_{m_1}\otimes\cdots \otimes \mathsf{Spl}_{m_k}}(\widetilde{C}_{m_1}\otimes \cdots \otimes \widetilde{C}_{m_k}).
\end{equation*}
Geometrically, we can represent this as a sequence of $k$ cups, each of which is labeled with $(m_i+p_i)^+$ and $(m_i+p_i)^-$ for different choices of $p_i\in \Z$ at the bottom of the cup, with additional light strands joining the different cups: 
\begin{equation}
		\begin{tikzpicture}[baseline=-5]
	\laddercoordinates{1}{0}
	\node[inner sep=0pt] (A) at (.5*\ladderX,-.5) {};
	 \draw[mid>, very thick] (l10)  
	 to [in=0,out=-90] node[at start, above, scale=.75] {$(m_1+p_1)^-$}(A) to[out=180,in=-90] node[at end, above , scale=.75]{$(m_1+p_1)^+$}(l00);
\end{tikzpicture} 	\quad 	\begin{tikzpicture}[baseline=-5]
	 \draw[mid>, very thick] (l10)  
	 to [in=0,out=-90] node[at start, above, scale=.75] {$(m_2+p_2)^-$}(A) to[out=180,in=-90] node[at end, above , scale=.75]{$(m_2+p_2)^+$}(l00);
\end{tikzpicture}\quad\cdots \quad 	\begin{tikzpicture}[baseline=-5]
	 \draw[mid>, very thick] (l10)  
	 to [in=0,out=-90] node[at start, above, scale=.75] {$(m_k+p_k)^-$}(A) to[out=180,in=-90] node[at end, above , scale=.75]{$(m_k+p_k)^+$}(l00);
\end{tikzpicture} \;.
\end{equation}
We obtain a module $C_{m_1,\dots, m_k}$ over $\Spidm{\mathbf{m},\boldsymbol{\epsilon}}{N}$ by only considering the diagrams with no light strands at the top;  we need to define this module over $\Spl{\mathbf{m},\boldsymbol{\epsilon}}{N}$ first so that we can use the relation \cref{eq:cup-isotope} regardless of what the diagram on top of it looks like.

We can extend this construction by twisting the action on the module $C_{m_1,\dots, m_k}$ by the automorphism $\mathsf{T}_{\beta}$ attached to a framed braid $\beta$.  This allows us to attach a right module ${}_TC$ to any framed tangle $T$ with only maxima; an analogous construction attaches a left module $C_{T'}$ to a tangle $T'$ with only minima.\notation{${}_TC,C_{T'}$}{The $\Spidm{\mathbf{m},\boldsymbol{\epsilon}}{N}$ module corresponding to a tangle $T$ with only maxima (right module ${}_TC$) or a tangle $T'$ with only minima (left module $C_{T'}$).}  

\begin{lemma}\label{lem:only-tangle}
  The module attached to a framed tangle with only minima or only maxima is independent of the choice of braid $\beta$ representing that tangle---that is, it is an invariant of the tangle up to isotopy.  
\end{lemma}
\begin{proof}
	To prove this, we only need to check that:
	\begin{enumerate}
		\item A full twist of the two strands at the top of a single cup sends $C_{m_1,\dots, m_k}$ to itself. 
		\item A crossing of two strands at the top of the $j$th cup over the strands at the top of the $(j+1)$st cup sends $C_{m_1,\dots,m_k}$ to $C_{m_1,\dots, m_{j+1},m_j,\dots,m_k}$.  
 	\end{enumerate}
 	Both of these follow from easy isotopy arguments.  
\end{proof}

\subsection{Planar limit}

We now consider how this construction behaves in the limit $N\to \infty$ and $q\to 1$, with $g=q^N$ fixed.  
We have already done half the work by treating $g$ as an independent parameter, so only the limit $q\to 1$ remains.  
To do this, we work over the base ring $\Bbbk=\C(q)[g^{\pm 1},q^{\pm m}]_{m\in M}$. 
Since this ring has no maximal ideal defined by $q=1$, we instead pass to the subring $\Bbbk'=\C[q]_{(q-1)}[g^{\pm 1},q^{\pm m}]_{m\in M}$\notation{$\Bbbk'$}{The subring $\C[q]_{(q-1)}[g^{\pm 1},q^{\pm m}]_{m\in M}\subset \Bbbk$ where we only use rational functions regular at $q=1$.} and the corresponding subcategory of $\Spidm{\mathbf{m},\boldsymbol{\epsilon}}{N}$.  
Equivalently, we replace $\C(q)$ by the subring of rational functions with no pole at $q=1$; note that this contains the $q$-binomial coefficients $\qBinomial{k}{l}$ for all $k,l\in \Z_{\geq 0}$, and the inverses of these coefficients whenever the corresponding binomial coefficient is nonzero.   
There are obvious difficulties in doing this, caused by the appearance of $q-q^{-1}$ in the denominators in numerous places.  However, we can effectively just clear denominators everywhere and still obtain a well-defined limit.

Thus, in our diagrams, we will use a blue color and a dot-dash pattern on a rung in a ladder to indicate that the rung is multiplied by $q-q^{-1}$. 
\begin{center}
\fbox{\parbox{0.85\textwidth}{\textbf{Convention:} A \textcolor{blue}{blue, dot-dashed rung} in a diagram indicates multiplication by $(q-q^{-1})$. This allows us to work over the subring $\Bbbk'$ where $q=1$ is well-defined.}}
\end{center}
Of course, since this is just a scaling factor, only the number of blue rungs is significant, but since the power of $q-q^{-1}$ we need to multiply by is often proportional to the number of rungs in the ladder, this will prove to be a convenient shorthand.  In particular, the equations \crefrange{eq:thick-skein1}{eq:thick-skein2} become the skein relation for blue rungs:  
\begin{align}\label{eq:thick-blue-skein1}
	 \begin{tikzpicture}[baseline=23,scale=.6]
    \node (C) at (0,3) {$m$};
    \node (D) at (0,0) {$m$};
    \node (A) at (-1,3)  {$1$};
    \node (B) at (1,0){$1$};
    \node (i) at (intersection of A--B and D--C) {};
    \draw[rung>] (B) -- (A);
    \draw[mid>, very thick] (D) -- (i);
    \draw[mid>, very thick] (i) -- (C);
  \end{tikzpicture}-\begin{tikzpicture}[baseline=23,scale=.6]
    \node (C) at (0,3) {$m$};
    \node (D) at (0,0) {$m$};
    \node (A) at (-1,3)  {$1$};
    \node (B) at (1,0){$1$};
    \node (i) at (intersection of A--B and D--C) {};
    \draw[mid>,very thick] (D) -- (C);
    \draw[rung>] (B) -- (i);
    \draw[rung>] (i) -- (A);
  \end{tikzpicture} &=   \begin{tikzpicture}[baseline=23,scale=.6]
    \laddercoordinates{1}{2}
    \node (C) at (0,3) {$m$};
    \node (D) at (0,0) {$m$};
    \node (A) at (-1,3)  {$1$};
    \node (B) at (1,0){$1$};
     \node[inner sep=0pt] (i) at (0,1) {};
          \node[inner sep=0pt] (j) at (0,2){};
              \draw[rung>] (B) -- (j);
    \draw[rung>] (i) -- (A);
    \draw[mid>, very thick] (D)-- (C);
  \end{tikzpicture} \;,\\\label{eq:thick-blue-skein2}
  \begin{tikzpicture}[baseline=23,scale=.6]
    \node (C) at (1,3) {$m$};
    \node (D) at (-1,0) {$m$};
    \node (A) at (0,3)  {$1$};
    \node (B) at (0,0){$1$};
    \node (i) at (intersection of A--B and D--C) {};
    \draw[rung>] (D) -- (C);
    \draw[mid>,very thick] (B) -- (i);
    \draw[mid>,very thick] (i) -- (A);
  \end{tikzpicture} -   \begin{tikzpicture}[baseline=23,scale=.6]
    \node (C) at (0,3) {$m$};
    \node (D) at (0,0) {$m$};
    \node (A) at (1,3)  {$1$};
    \node (B) at (-1,0){$1$};
    \node (i) at (intersection of A--B and D--C) {};
    \draw[mid>,very thick] (D) -- (C);
    \draw[rung>] (B) -- (i);
    \draw[rung>] (i) -- (A);
  \end{tikzpicture}&=  - \begin{tikzpicture}[baseline=23,scale=.6]
    \laddercoordinates{1}{2}
    \node (C) at (0,3) {$m$};
    \node (D) at (0,0) {$m$};
    \node (A) at (1,3)  {$1$};
    \node (B) at (-1,0){$1$};
     \node[inner sep=0pt] (i) at (0,1) {};
          \node[inner sep=0pt] (j) at (0,2){};
              \draw[rung>] (B) -- (j);
    \draw[rung>] (i) -- (A);
    \draw[mid>, very thick] (D)-- (C);
  \end{tikzpicture} \;.
\end{align}
The corresponding $\C[q,q^{-1}]$ lattice in $U_q(\mathfrak{gl}_k)$ is denoted by $U_q(\mathfrak{gl}_k)'$ and its construction is discussed in \cite[\S 7.3]{gavariniGlobalQuantum2012}. Recall that an \textbf{integral form} of a quantum group is a subalgebra defined over the Laurent polynomials $\C[q,q^{-1}]$ rather than the field $\C(q)$; this allows us to specialize $q$ to non-zero complex numbers (such as $q=1$) while maintaining algebraic structure. The limit as $q\to 1$ of this algebra is commutative and can be identified with the ring of functions on the dual group $GL_k^*$; as an algebraic variety, this is isomorphic to $(\C^{\times})^k \times \C^{k(k-1)}$.  The coordinates for the first $k$ factors are the elements $K_i^{\pm 1}$ of the Cartan of $U_q(\mathfrak{gl}_k)$, and the coordinates for the remaining $k(k-1)$ factors are the generators corresponding to the roots of $\mathfrak{gl}_k$; in terms of diagrams, these are given by light strands connecting all pairs of distinct uprights.  
 For example, we can take:
\begin{equation}\label{eq:aij-def}
	 \aij =\begin{cases} \begin{tikzpicture}[baseline=20]
	\laddercoordinates{4}{1}
	\node (A) at (l00) {$m_i$};
	\node (B) at (l01) {$m_i$};
	\node (C) at (l30) {$m_j$};
	\node (D) at (l31) {$m_j$};
	\node (A1) at (l10) {\phantom{$m_i$}};
	\node (B1) at (l11) {\phantom{$m_i$}};
	\node (C1) at (l20) {\phantom{$m_i$}};
	\node (D1) at (l21) {\phantom{$m_i$}};
	\node[inner sep=0pt] (i) at (0,.5 *\ladderY){};
\node[inner sep=0pt] (j) at (3*\ladderX,.5 *\ladderY){};
	\node[inner sep=5pt] (i1) at (intersection of A1--B1 and i--j){};
		\node[inner sep=5pt] (i2) at (intersection of C1--D1 and i--j){};
	\draw[mid>,very thick] (A)--(B);
	\draw[mid>,very thick] (C)--(D); 
		\draw[mid>,very thick] (A1)--(B1);
	\draw[mid>,very thick] (C1)--(D1); 
	\draw[rung>] (i)--(i1);
	\draw[rung>] (i1)--(i2);
		\draw[rung>] (i2)--(j);	
\end{tikzpicture} & i < j\\
-\begin{tikzpicture}[baseline=20]
	\laddercoordinates{4}{1}
	\node (A) at (l00) {$m_j$};
	\node (B) at (l01) {$m_j$};
	\node (C) at (l30) {$m_i$};
	\node (D) at (l31) {$m_i$};
	\node (A1) at (l10) {\phantom{$m_i$}};
	\node (B1) at (l11) {\phantom{$m_i$}};
	\node (C1) at (l20) {\phantom{$m_i$}};
	\node (D1) at (l21) {\phantom{$m_i$}};
	\node[inner sep=3pt] (i) at (0,.5 *\ladderY){};
\node[inner sep=3pt] (j) at (3*\ladderX,.5 *\ladderY){};
\node[inner sep=0pt] (ii) at (0, .5*\ladderY-.3){};
\node[inner sep=0pt] (jj) at (3*\ladderX,.5 *\ladderY+.3){};
	\node[inner sep=5pt] (i1) at (intersection of A1--B1 and i--j){};
		\node[inner sep=5pt] (i2) at (intersection of C1--D1 and i--j){};
	\draw[mid>,very thick] (A)--(i) ;
	\draw[mid>,very thick] (i)--(B);
	\draw[mid>,very thick] (C)--(D); 
		\draw[mid>,very thick] (A1)--(B1);
	\draw[mid>,very thick] (C1)--(D1); 
	\draw[rung<] (ii) to[out=180,in=-90] (-.3,.5 *\ladderY-.15) to[out=90,in=180] (0,.5 *\ladderY) to (i1);
	\draw[rung] (jj) to[out=0,in=90] (3*\ladderX + .3,.5 *\ladderY+.15) to[out=-90,in=0]  (j);
	\draw[rung<] (i1)--(i2);
		\draw[rung<] (i2)--(j);	
\end{tikzpicture} & j < i. 
\end{cases}
\end{equation}
Unlike in $\mathfrak{gl}_k$ itself, there is no canonical choice of these elements, since we need to decide which path between two uprights to take, as well as the choice of framing, which is necessary to consistently apply skein relations. 

\begin{definition}
  Let $\Spidq{\mathbf{m},\boldsymbol{\epsilon}}{N}$\notation{$\Spidq{\mathbf{m},\boldsymbol{\epsilon}}{N}$}{The $\Bbbk'$-lattice in $\Spidm{\mathbf{m},\boldsymbol{\epsilon}}{N}$ generated by the diagrams $\aij$.} be the subcategory of 
$\Spidm{\mathbf{m},\boldsymbol{\epsilon}}{N}$ generated over $\Bbbk'$ by the diagrams $\aij$ above, and let $\Spidp{\mathbf{m},\boldsymbol{\epsilon}}{N}$ be its specialization at $q=1$.\notation{$\Spidp{\mathbf{m},\boldsymbol{\epsilon}}{N}$}{The specialization at $q=1$ of the $\Bbbk'$-lattice $\Spidq{\mathbf{m},\boldsymbol{\epsilon}}{N}$.}  
\end{definition}

One can easily prove that an analogue of \cref{cor:spider-basis} holds in $\Spidq{\mathbf{m},\boldsymbol{\epsilon}}{N}$.

\subsection{Tangles and knots}
The braid action and cup modules also descend to this limit to give $\Spidq{\mathbf{m},\boldsymbol{\epsilon}}{N}$ modules $C'_{m_1,\dots, m_k}$\notation{$C'_{m_1,\dots, m_k}$}{The $\Spidq{\mathbf{m},\boldsymbol{\epsilon}}{N}$-module obtained by specializing $C_{m_1,\dots, m_k}$.} generated by the diagram of the heavy tangle with no light strands.  As before, we use the shorthand $\ourmu=q^{m}$.  
The equation \cref{eq:cup-isotope} holds without change, and the action
\crefrange{E-cup-action}{F-cup-action} becomes
\begin{gather}\label{E-blue-cup}
	\begin{tikzpicture}[baseline=20,xscale=1.3]
	\laddercoordinates{1}{1}
	\node[inner sep=0pt] (A) at (.5*\ladderX,-.5) {};
	\draw[mid>,very thick] (u00) -- ($(l00)+(0,\ladderY)$) node[above, scale=.75] {$(m+p+1)^+$};
\draw[mid<,very thick] ($(d00)+(\ladderX,0)$) -- ($(l00)+(\ladderX,\ladderY)$) node[above, scale=.75] {$(m+p+1)^-$};
\draw[rung>] ($(d00)+(\ladderX,0)$) -- (u00);
\draw[mid>, very thick] ($(d00)+(\ladderX,0)$) to (l10) to[out=-90,in=0](A) to[out=180,in=-90] (l00) -- (u00);
\end{tikzpicture}=  (q^{p+1}\ourmu-q^{-p-1}\ourmu^{-1}) \quad		\begin{tikzpicture}[baseline=20,xscale=1.3]
	\laddercoordinates{1}{1}
	\node[inner sep=0pt] (A) at (.5*\ladderX,-.5) {};
	\draw[mid>,very thick] (u00) -- ($(l00)+(0,\ladderY)$) node[above, scale=.75] {$(m+p+1)^+$};
\draw[mid<,very thick] ($(d00)+(\ladderX,0)$) -- ($(l00)+(\ladderX,\ladderY)$) node[above, scale=.75] {$(m+p+1)^-$};
	 \draw[mid>, very thick] ($(d00)+(\ladderX,0)$) to (l10) to[out=-90,in=0](A) to[out=180,in=-90] (l00) -- (u00);
\end{tikzpicture} \;,\\\label{F-blue-cup}
	\begin{tikzpicture}[baseline=20,xscale=1.3]
	\laddercoordinates{1}{1}
	\node[inner sep=0pt] (A) at (.5*\ladderX,-.5) {};
	\draw[mid>,very thick] (u00) -- ($(l00)+(0,\ladderY)$) node[above, scale=.75] {$(m+p-1)^+$};
\draw[mid<,very thick] ($(d00)+(\ladderX,0)$) -- ($(l00)+(\ladderX,\ladderY)$) node[above, scale=.75] {$(m+p-1)^-$};
\draw[rung<] ($(u00)+(\ladderX,0)$) -- (d00);
\draw[mid>, very thick] ($(d00)+(\ladderX,0)$) to (l10) to[out=-90,in=0](A) to[out=180,in=-90] (l00) -- (u00);
\end{tikzpicture}=  (gq^{-p+1}\ourmu^{-1}-g^{-1}q^{p-1}\ourmu) \quad		\begin{tikzpicture}[baseline=20,xscale=1.3]
	\laddercoordinates{1}{1}
	\node[inner sep=0pt] (A) at (.5*\ladderX,-.5) {};
	\draw[mid>,very thick] (u00) -- ($(l00)+(0,\ladderY)$) node[above, scale=.75] {$(m+p-1)^+$};
\draw[mid<,very thick] ($(d00)+(\ladderX,0)$) -- ($(l00)+(\ladderX,\ladderY)$) node[above, scale=.75] {$(m+p-1)^-$};
\draw[mid>, very thick] ($(d00)+(\ladderX,0)$) to (l10) to[out=-90,in=0](A) to[out=180,in=-90] (l00) -- (u00);
\end{tikzpicture} \;.
\end{gather}

Given a link, there is a familiar notion of writing it as the closure of a braid.  That is, we can take a braid $\beta$ and close it by connecting the strands on the top and bottom in order; we represent this closure on the page by sweeping the strands to the right.

Since all skein-theoretic constructions require a framing, we will always equip a link with the zero-framing, characterized by the condition that each component has linking number 0 with a copy pushed off by the framing.  If we take a braid closure with the blackboard framing, then this linking number is the writhe $\wri(\beta)$\notation{$\wri(\beta)$}{The writhe of the braid $\beta$.} of the braid, that is, the signed count of crossings.  
Thus, we must add $-\wri(\beta)$ twists to the closure to get to the zero-framing.

If $L$ is a knot, then by convention, we do this just above the braid on the leftmost strand, at the spot marked by a star ($*$) in the picture below. 
The star serves two purposes: it marks where the framing correction twists are inserted, and it is also the reference point where the label-shift operator $\ourlambda$, which we will introduce below, acts.
\begin{equation}
	\begin{tikzpicture}[baseline]
	\draw[very thick,lower>] (-1.2,0) -- (-1.2,1) to[out=90,in=180]  node[pos=.2,left]          {$*$}  (2.5,3)to[out=0,in=90]  (6.2,1)-- (6.2,-1) to[out=-90,in=0] (2.5,-3) to[out=180, in=-90] (-1.2,-1)--cycle;
	\draw[very thick,lower>] (1.2,0) -- (1.2,1) to[out=90,in=180] (2.5,1.8) to[out=0,in=90]  (3.8,1)-- (3.8,-1) to[out=-90,in=0] (2.5,-1.8) to[out=180, in=-90] (1.2,-1)--cycle;	
	\draw[very thick,lower>] (-.4,0) -- (-.4,1) to[out=90,in=180] (2.5,2.6)to[out=0,in=90]  (5.4,1)-- (5.4,-1) to[out=-90,in=0] (2.5,-2.6) to[out=180, in=-90] (-.4,-1)--cycle;
	\draw[very thick,lower>] (.4,0) -- (.4,1) to[out=90,in=180] (2.5,2.2)to[out=0,in=90]  (4.6,1)-- (4.6,-1) to[out=-90,in=0] (2.5,-2.2) to[out=180, in=-90] (.4,-1)--cycle;
	\node[inner xsep=40pt,inner ysep=30pt,draw,very thick,fill=white] (box) at (0,0){$\beta$};
	\end{tikzpicture}
\end{equation}

For a link, we must choose such a base point for each component---the usual convention is to choose the leftmost strand of each component at the point just above the braid.

Having constructed a link $L$ as the closure of a braid with heavy strands as above, we consider the formal span of all diagrams obtained by labeling the upward-oriented portions of the heavy link with elements of $M+\Z$, and by adding all possible configurations of light strands. We choose the convention that the elements of $M$ are in bijection with the $r$ components of $L$, and that each upward-oriented heavy segment is labeled by the corresponding element of $M$.

More formally, we slice the braid closure in the middle and obtain a right module ${}_{T_1} C'$ and a left module $C'_{T_2}$ over $\Spidq{\mathbf{m},\boldsymbol{\epsilon}}{N}$, where $T_1$ is the top half (with only maxima) and $T_2$ is the bottom half (with only minima) of the link.

\begin{definition}
	We let $\ourabq L={}_{T_1} C'\otimes_{\Spidq{\mathbf{m},\boldsymbol{\epsilon}}{N}}C'_{T_2}$\notation{$\ourabq L$}{The link invariant $\ourabq L={}_{T_1} C'\otimes_{\Spidq{\mathbf{m},\boldsymbol{\epsilon}}{N}}C'_{T_2}$.}.  
\end{definition}
\begin{theorem}
  The invariant $\ourabq L$ only depends on the isotopy class of $L$.  That is, we obtain the same value for $\ourabq L$ regardless of the choice of splitting of a diagram of $L$ into tangles $T_1,T_2$.
\end{theorem}
\begin{proof}
 Since all of our relations are invariant under isotopy, isotopies that preserve the plane of the decomposition into two tangles do not change $\ourabq L$. Thus, we can reduce to the case where $T_1,T_2$ are obtained by cutting a braid closure of a braid $\beta$ representing $L$. A change of cut position corresponds to an isotopy of this decomposition, so it does not change $\ourabq L$. Furthermore, conjugation in the braid group can be performed compatibly with the cut, so $\ourabq L$ depends only on the conjugacy class, not on the chosen braid representative.
 
 Thus, by the Markov theorem, we only need to show invariance in the case of the addition or removal of a positive or negative stabilization.  We can use our local relations to move the end of any light strand away from the curl made by the stabilization;  once we have done this, the isomorphism is given by the isotopy which adds or removes the curl.  Note that this adds a twist to the heavy strand, which we must isotope back to the star.
\end{proof}

We can view this as a generalization of the computation of the HOMFLYPT polynomial as follows. For the usual HOMFLYPT polynomial with coloring $\bigwedge{}^{\!k_i}\C^N$ on each component, one specializes the formal parameters $m_i$ to the integers $k_i$ and evaluates in the spider category $\Spid{N}$. Our construction keeps $m_i$ formal, which captures the relations between these polynomials for different colorings.

More precisely, we let $\C_g$ be the 1-dimensional representation of $U_q(\mathfrak{gl}_n\oplus \mathfrak{gl}_n)$ considered as a subalgebra of $U_q(\mathfrak{gl}_{2n})$ where $E_i,F_i$ for $i\neq n$ act by $0$ and $K_i$ acts by $1$ if $i\leq n$ and by $g$ if $i>n$.  We then let $\pverma$ be the parabolic Verma module induced from $\C_g$ with $F_n$ acting by $0$ (so this is a lowest weight module), considered as a left $U_q(\mathfrak{gl}_{2n})$-module.  We let $\pverma^{\circ}$ be the right module obtained by twisting the action on $\pverma$ by the Cartan involution $E_i\mapsto F_i$, $F_i\mapsto E_i$, $K_i\mapsto K_i$.  

Given $\mathbf{k}=(k_1,\dots, k_n)\in \Z_{\geq 0}^n$, there is a vector in this module $v_{\mathbf{k}}$ which corresponds to the fully nested cups $T$ given by $E_{n,n+1}^{(k_n)} E_{n-1,n+2}^{(k_{n-1})}E_{n-2,n+3}^{(k_{n-2})}\cdots E_{1,2n}^{(k_1)}  v_{0}$.

We can compute the HOMFLYPT polynomial of a link $L$, written as the closure of a braid $\beta$ on $n$ strands, as follows: the Lusztig braid action of $\beta^{-1}$ induces a self-map $\pverma\to \pverma$ which intertwines the usual action of $U_q(\mathfrak{gl}_{2n})$ with the action twisted by this Lusztig braid action; let $\pverma^{\beta}$ be the parabolic Verma module with this twisted action.  

The tensor product $\pverma^{\circ}\otimes_{U_q(\mathfrak{gl}_{2n})} \pverma^{\beta}$ is 1-dimensional, spanned by $v_0\otimes v_0$; the tensor product $v_{\mathbf{k}}\otimes v_\mathbf{k}$ is precisely the product of the colored HOMFLYPT polynomial for $L$ and $v_0\otimes v_0$.

\subsection{Label-shifting}

The introduction of heavy objects introduces a new structure into this category:  the ability to shift the label on heavy strands.  
\begin{definition}
    For any $m\in M$, the label-shift automorphism $\shift_m$ of $\SpidM{N}$ is an autoequivalence which:
    \begin{enumerate}
        \item Acts on objects by sending $(m+p)^{\pm}$ to $(m+p-1)^{\pm}$, and all other elements of $\Pi^{\pm}$ to themselves.
        \item Acts on diagrams by shifting the label of each heavy upright as above.
        \item Acts on the base ring $\Bbbk$ by
        \begin{equation}\label{eq:shift-torus}
            \shift_m(q)=q,\qquad \shift_m(g)=g,\qquad \shift_m(\ourmu')=q^{-\delta_{m,m'}}\ourmu'.
        \end{equation}
    \end{enumerate}
\end{definition}
Thus, label-shift automorphisms and multiplication by the scalars $\ourmu_i$ acting on morphisms satisfy the relations of the quantum torus $\qt{r}$.  

We can define a smash product category, where we add in a label-shift isomorphism from $\mathbf{a}$ to $\shift_m\mathbf{a}$ for all $m\in M$ and objects $\mathbf{a}$, which commutes past the morphisms of the category $\SpidM{N}$ by the relations \cref{eq:shift-torus}.

More precisely, let $\tSpidM{N}$ be the category with the same underlying objects as $\SpidM{N}$ whose morphism spaces are
\[\Hom_{\tSpidM{N}}(\Br,\Bs) \cong \bigoplus_{(a_1,\dots, a_r)\in \Z^r}\Hom_{\SpidM{N}}(\Br,\prod\shift_i^{-a_i}\cdot \Bs)\cdot\prod\shift_i^{a_i}, \] with composition defined using \cref{eq:shift-torus} to simplify expressions of the form $f_1\cdot\prod\shift_i^{a_i}\cdot f_2 \cdot \prod\shift_i^{b_i}$.

We can extend all the modules appearing in the sections above to this category.  In particular, $\ourabq L$ for a link $L$ inherits a module structure over $\qt{r}$.  
If we consider the base change $\ourab{L}=\ourabq L\otimes_{\C[q,q^{-1}]}\C$\notation{$\ourab{L}$}{The base change of $\ourabq L$ where we set $q=1.$} where we set $q=1$, this becomes a module structure over the Laurent polynomial ring $\Laur{r}=\C[\ourmu^{\pm 1}_i,\ourlambda^{\pm 1}_i,g^{\pm 1}]$\notation{$\Laur{r}$}{The Laurent polynomial ring $\C[\ourmu^{\pm 1}_i,\ourlambda^{\pm 1}_i,g^{\pm 1}]$.}.
Over this algebra, we can interpret \crefrange{E-blue-cup}{F-blue-cup} as a presentation of the module $C_{m}$ over the smash product in terms of the images $\tilde{E},\tilde{F}$ of the elements $(q-q^{-1})E,(q-q^{-1})F$ of $U_q(\mathfrak{gl}_2)$, which act on the vector $v_0$ corresponding to the empty diagram as:  
\[(\tilde{E}-(q\ourmu-q^{-1}\ourmu^{-1})\Lambda^{-1})v_0=0, \qquad (\tilde{F}-(gq\ourmu^{-1}-g^{-1}q^{-1}\ourmu)\Lambda)v_0=0.\]

Recall that in the introduction, we defined an ideal 
\begin{equation}
    \Iq{L}=\{x\in \qt{r}\mid x\cdot [L]=0\}.\notation{$\Iq{L}$}{The ideal in $\qt{r}$ annihilating $[L]$.}
\end{equation}
This is the ``ideal of relations between HOMFLYPT polynomials which can be derived from the MOY relations.''  

In the introduction, we also introduced the space $\funcs{r}$ of all maps $\Z^r\to \C(q)[g^{\pm 1}]$, with its usual $\qt{r}$-module structure.
There is a $\qt{r}$-module homomorphism $\ev\colon \ourabq{L}\to \funcs{r}$ given by $\ev(D)(k_1,\dots, k_r)$ being the evaluation of $D$ in the HOMFLYPT skein module, that is, as an endomorphism of the tensor identity in  $\Spid{N}$, after the substitution $m_i\mapsto k_i$.  Of course, under this specialization, we identify $\ourmu_i\mapsto q^{k_i}$.  
\begin{conjecture}\label{conj:ev-injective}
The map $\ev$ is injective and $\Iq{L}$ is equal to the annihilator $\Iqp{L}$ of $\ev([L])$.
\end{conjecture}
In \cite{garoufalidisColoredHOMFLYPT2018}, Garoufalidis, Lauda, and L\^e show that $\ev([L])$ is a holonomic element of $\funcs{r}$, that is, that the quotient $\qt{r}/\Iqp{L}$ has Gelfand--Kirillov dimension $r$ as a $\qt{r}$-module.

\section{Comparison with knot contact homology}
In \cref{sec:background,sec:comparison}, we consider a fixed knot $K$ and compare the invariant $\ourab K$ with knot contact homology. Then, in \cref{sec:links}, we turn to the difficulties raised by more general links, and in \cref{sec:augmentation} we discuss connections to augmentation varieties.  

\subsection{Background on knot contact homology}\label{sec:background}

Knot contact homology is a powerful knot invariant arising from symplectic geometry. Roughly speaking, it counts holomorphic disks in the cotangent bundle $T^*\R^3$ with boundary on the unit conormal bundle of the knot. For computational purposes, we use the combinatorial formulation of Ng \cite{ngTopologicalIntroduction2014}, which associates to a braid presentation of a knot a differential graded algebra (dga) whose homology is the knot contact homology.

The variable $U$ used there is often denoted $Q$ in more recent literature, but we will use $U$ to more closely match the notation of \cite{ngTopologicalIntroduction2014}.  This defines a dga $\mathcal{A}(K)$ over $\C[\ngmu^{\pm 1},\nglambda^{\pm 1},U^{\pm 1}]$. The invariant of interest to us is the largest commutative quotient algebra of the degree 0 part of $\mathcal{A}(K)$.  To avoid cumbersome notation later, we build in the change of variables that we will need:
\begin{definition}
    Let $\ngab K$\notation{$\ngab K$}{The largest commutative quotient algebra of the degree 0 part of knot contact homology dga $\mathcal{A}(K)$.} be the largest commutative quotient algebra of the degree 0 part of $\mathcal{A}(K)$ with scalars extended from 
  	$\C[\ngmu^{\pm 1},\nglambda^{\pm 1},U^{\pm 1}]$ to $\Laur{1}=\C[\ourmu^{\pm 1},\ourlambda^{\pm 1},g^{\pm 1}]$ via the injective map
    	\begin{equation}\label{eq:substitution}
		\ngmu\mapsto\ourmu^{-2}, \qquad U\mapsto g^{-2},\qquad \nglambda\mapsto-g^{-1}\ourlambda^{-1}.
  \end{equation}  \notation{$\ngmu,U,\nglambda$}{The usual variables for knot contact homology, related to ours by $\ngmu=\ourmu^{-2}, U= g^{-2}, \nglambda=-g^{-1}\ourlambda^{-1}$.} This is the coordinate ring of the {\bf full augmentation variety} of the knot $K$ in the sense of Diogo and Ekholm \cite{diogoAugmentationsAnnuli2025};  it could thus also be called the full augmentation algebra.  The full augmentation variety maps to the usual augmentation variety by forgetting the variables $a_{ij}$ introduced below, which correspond to augmentations of Reeb chords between different strands of the braid.  Thus, the augmentation algebra is the subalgebra of $\ngab K$ generated by $\ourmu^{\pm 1},\ourlambda^{\pm 1},g^{\pm 1}$.
\end{definition}

The algebra $\mathcal{A}(K)$ is concentrated in degrees $\geq 0$, with the convention that the degree of the differential is $-1$ (in contrast to much of the literature on dg-algebras).  Thus, $\ngab K$ is generated by Ng's degree 0 generators $a_{ij}, \ngmu^{\pm 1},\nglambda^{\pm 1},U^{\pm 1}$, subject to the relations that the resulting algebra is commutative and that the differential of any degree 1 element is 0.  Let $\freeA$\notation{$\freeA$}{The free commutative algebra over $\C[\ngmu^{\pm 1},\nglambda^{\pm 1},U^{\pm 1}]$ generated by $a_{ij}$ with $i\neq j$.  } be the free commutative algebra over $\C[\ngmu^{\pm 1},\nglambda^{\pm 1},U^{\pm 1}]$ generated by $a_{ij}$ with $i\neq j$.  

The algebra $\freeA$ inherits a braid group action from \cite[Def. 3.3]{ngTopologicalIntroduction2014}, where the positive braid generator $\sigma_k$ acts by:
\begin{equation}
\phi_{\sigma_k}\notation{$\phi_{\beta}$}{The action of a braid $\beta$ on the algebra $\freeA$.}: \begin{cases}a_{i j} \mapsto a_{i j}, & i, j \neq k, k+1 \\ a_{k+1, i} \mapsto a_{k i}, & i \neq k, k+1 \\ a_{i, k+1} \mapsto a_{i k}, & i \neq k, k+1 \\ a_{k, k+1} \mapsto-a_{k+1, k} & \\ a_{k+1, k} \mapsto-a_{k, k+1} & \\ a_{k i} \mapsto a_{k+1, i}-a_{k+1, k} a_{k i}, & i \neq k, k+1 \\ a_{i k} \mapsto a_{i, k+1}-a_{i k} a_{k, k+1}, & i \neq k, k+1 .\end{cases}
\end{equation}

We now introduce matrices that organize the generators $a_{ij}$ in a way that makes the connection to our skein-theoretic constructions transparent. Define $n\times n$ matrices $\mathbf{A},\mathbf{\hat{A}}$ in $\freeA$ by:
\[\mathbf{A}_{ij}=\begin{cases} a_{ij}, & i< j\\ -\ngmu a_{ij}, & i>j\\
1-\ngmu, & i=j,
\end{cases}\qquad \qquad\mathbf{\hat{A}}_{ij}=\begin{cases} Ua_{ij}, & i< j\\ -\ngmu a_{ij}, & i>j\\
U-\ngmu, & i=j.
\end{cases}\]
\notation{$\mathbf{A},\mathbf{\hat{A}},\boldsymbol{\Lambda}$}{Matrices in $\freeA$ that appear in the relations of $\ngab{K}$.}These matrices encode the augmentation data: the diagonal entries record the meridian variable $\ngmu$, while off-diagonal entries $a_{ij}$ correspond to ``paths'' between strands $i$ and $j$ that contribute to the augmentation variety. 

The braid group action above can be captured by maps
\[\beta \mapsto \boldsymbol{\Phi}^L_{\beta},\qquad \beta \mapsto (\boldsymbol{\Phi}^R_{\beta})^{-1}.\]
These matrices are defined as in \cite[\S 2.3]{ngFramedKnot2008} (note that there is a typo in the corresponding formula in \cite{ngTopologicalIntroduction2014}):\notation{$\boldsymbol{\Phi}^L_{\beta},\boldsymbol{\Phi}^R_{\beta}$}{Matrices associated to the braid $\beta$ which describe the action on the matrix $\mathbf{A}$.}
\begin{equation}
\phi_\beta\left(a_{i *}\right)=\sum_{j=1}^n\left(\boldsymbol{\Phi}_\beta^L\right)_{i j} a_{j *} \quad \text { and } \quad \phi_\beta\left(a_{* j}\right)=\sum_{i=1}^n a_{* i}\left(\boldsymbol{\Phi}_\beta^R\right)_{i j},
\end{equation}
where $*$ denotes any index larger than all strands involved in the braid $\beta$; if necessary, we add an extra strand to the right of all existing strands to ensure such an index exists.
This is uniquely determined by starting with the matrices corresponding to the generators of the braid group:
\[
\boldsymbol{\Phi}^L_{\sigma_i}=\left[\begin{array}{c|cc|c}
    I_{i-1} & 0 & 0&0\\\hline 
    0 & -a_{i+1,i} & 1& 0\\
    0 & 1 & 0 & 0\\\hline 
    0 & 0 & 0 & I_{n-i-1}
\end{array}\right],\qquad \boldsymbol{\Phi}^R_{\sigma_i}=\left[\begin{array}{c|cc|c}
    I_{i-1} & 0 & 0&0\\\hline 
    0 & -a_{i,i+1} & 1& 0\\
    0 & 1 & 0 & 0\\\hline 
    0 & 0 & 0 & I_{n-i-1}
\end{array}\right],
\]
and combining them according to the rule that \[\boldsymbol{\Phi}_{\beta_1\beta_2}^L= \phi_{\beta_1}(\boldsymbol{\Phi}_{\beta_2}^L)\boldsymbol{\Phi}_{\beta_1}^L \quad \text { and } \quad \boldsymbol{\Phi}_{\beta_1\beta_2}^R=\boldsymbol{\Phi}_{\beta_1}^R \phi_{\beta_1}(\boldsymbol{\Phi}_{\beta_2}^R).  \]
These matrices satisfy the property that $\phi_{\beta}(\mathbf{A})=\boldsymbol{\Phi}^L_{\beta}\mathbf{A}\boldsymbol{\Phi}^R_{\beta}$.

\begin{lemma}\label{lem:Phi-braid}
	Under the map $\freeA\to \Spidp{\mathbf{m},\boldsymbol{\epsilon}}{N}$ sending $a_{ij}\mapsto \aij$, this braid group action is intertwined with the braid group action of \cref{sec:braid}.
\end{lemma}
Before the proof of this result, let us record a set of relations that show how we can change framing, which have already appeared in \cite[(E.21-23)]{gaiottoDbranesPlanar2026}: 
\begin{align}\label{eq:framingchange1}
\begin{tikzpicture}[baseline=-3]
\node (A) at (-1,-0.5) {};
\node (B) at (0,0) {};
\coordinate (B1) at (-0.3,0.3) {};
\coordinate (B2) at (0.3,0.3) {};
\coordinate (B3) at (-0.3,-0.3) {};
\coordinate (B4) at (0.3,-0.3) {};
\node (C) at (1,-0.5) {};
\draw[rung] (A) to[out=0,in=225,looseness=1]  (B3);
\draw[rung] (B3) --  (B);
\draw[rung] (B) --  (B2);
\draw[rung] (B1) -- (B4);
\draw[rung] (B4) to[out=-45,in=180,looseness=1]  (C);
\draw[rung] (B2) to[out=45,in=135,looseness=3] (B1);
\end{tikzpicture}\quad =& \quad -g\begin{tikzpicture}[baseline=-3]
\node (A) at (1,0) {};
\node (B) at (-1,0) {};
\draw[rung] (B) --  (A);
\end{tikzpicture} \;,
\end{align}
\begin{align}\label{eq:framingchange2}
\begin{tikzpicture}[baseline=-3]
\node (C) at (0,1) {};
\node (D) at (0,-1) {};
\coordinate (i) at (0,-0.2) {};
\node (j) at (0,0.2) {};
\coordinate (j2) at (0,0.2) {};
\node (A) at (1,0.2) {};
\draw[rung>] (j2) to[out=180,in=180,looseness=4] (i);
\draw[rung>] (A) -- (j2);
\draw[mid>,very thick] (D) -- (i);
\draw[very thick] (i) -- (j);
\draw[mid>,very thick] (j) -- (C);
\end{tikzpicture}
\quad = \quad - g^{-1}\ourmu \quad
\begin{tikzpicture}[baseline=-3]
\node (C) at (0,1) {};
\node (D) at (0,-1) {};
\coordinate (A) at (0,0) {};
\node (B) at (1,0) {};
\draw[rung<] (A) -- (B);
\draw[mid>,very thick] (D) -- (A);
\draw[mid>,very thick] (A) -- (C);
\end{tikzpicture} \;,
\quad \quad \quad \quad \quad \quad \quad
\begin{tikzpicture}[baseline=-3]
\node (C) at (0,1) {};
\node (D) at (0,-1) {};
\coordinate (i) at (0,-0.2) {};
\coordinate (i2) at (0,0) {};
\node (j) at (0,0.2) {};
\coordinate (j2) at (0,0.2) {};
\node (A) at (1,0.2) {};
\draw[rung>] (j) to[out=180,in=180,looseness=4] (i);
\draw[rung>] (A) -- (j);
\draw[mid>,very thick] (D) -- (i);
\draw[very thick] (i) -- (j2);
\draw[mid>,very thick] (j2) -- (C);
\end{tikzpicture}
\quad = \quad - g \ourmu^{-1} \quad
\begin{tikzpicture}[baseline=-3]
\node (C) at (0,1) {};
\node (D) at (0,-1) {};
\coordinate (A) at (0,0) {};
\node (B) at (1,0) {};
\draw[rung<] (A) -- (B);
\draw[mid>,very thick] (D) -- (A);
\draw[mid>,very thick] (A) -- (C);
\end{tikzpicture} \;.
\end{align}

\begin{proof}
	This is a straightforward computation.  The case $i, j \neq k, k+1$ is simple isotopy ignoring the crossing strands.  The case of $a_{k+1, i}$ and $a_{i, k+1}$ for $i \neq k, k+1$ is an isotopy along the $(k+1)$th strand as shown below:
	\begin{equation}
	\begin{tikzpicture}[baseline=23,scale=.6]
    \node (C) at (1,3) {$m_k$};
    \node (D) at (-1,0) {$m_k$};
    \node (A) at (-1,3)  {$m_{k+1}$};
    \node (B) at (1,0){$m_{k+1}$};
    \node (i) at (intersection of A--B and D--C) {};
    \draw[rung>] (1.5,1) -- (.33,1); 
    \draw[very thick,->] (D) -- (C);
    \draw[very thick] (B) -- (i);
    \draw[very thick,->] (i) -- (A);
  \end{tikzpicture}
  = 
		  \begin{tikzpicture}[baseline=23,scale=.6]
    \node (C) at (1,3) {$m_k$};
    \node (D) at (-1,0) {$m_k$};
    \node (A) at (-1,3)  {$m_{k+1}$};
    \node (B) at (1,0){$m_{k+1}$};
    \node (i) at (intersection of A--B and D--C) {};

    \draw[rung>] (1.5,2) -- (-.33,2); 
    \draw[line width=4pt,white] (D)--(C);
    \draw[very thick,->] (D) -- (C);
    \draw[very thick] (B) -- (i);
    \draw[very thick,->] (i) -- (A);
  \end{tikzpicture} \;,  \hspace{1in} 		  \begin{tikzpicture}[baseline=23,scale=.6]
    \node (C) at (1,3) {$m_k$};
    \node (D) at (-1,0) {$m_k$};
    \node (A) at (-1,3)  {$m_{k+1}$};
    \node (B) at (1,0){$m_{k+1}$};
    \node (i) at (intersection of A--B and D--C) {};
    \draw[rung<] (-1.5,1) -- (.33,1); 
    \draw[line width=4pt,white] (D)--(C);
    \draw[very thick,->] (D) -- (C);
    \draw[very thick] (B) -- (i);
    \draw[very thick,->] (i) -- (A);
  \end{tikzpicture}
  = 
		  \begin{tikzpicture}[baseline=23,scale=.6]
    \node (C) at (1,3) {$m_k$};
    \node (D) at (-1,0) {$m_k$};
    \node (A) at (-1,3)  {$m_{k+1}$};
    \node (B) at (1,0){$m_{k+1}$};
    \node (i) at (intersection of A--B and D--C) {};

    \draw[rung<] (-1.5,2) -- (-.33,2); 
    \draw[line width=4pt,white] (D)--(C);
    \draw[very thick,->] (D) -- (C);
    \draw[very thick] (B) -- (i);
    \draw[very thick,->] (i) -- (A);
  \end{tikzpicture} \;.   
	\end{equation}
	The case $a_{k, k+1}$ follows from 
the isotopy
	\begin{equation}
\begin{tikzpicture}[baseline=23,scale=.6]
    \node (C) at (1,3) {$m_k$};
    \node (D) at (-1,0) {$m_k$};
    \node (A) at (-1,3)  {$m_{k+1}$};
    \node (B) at (1,0){$m_{k+1}$};
    \node (i) at (intersection of A--B and D--C) {};
    \draw[rung<] ({1-.6/1.5},.6) -- ({-1+.6/1.5},.6); 
    \draw[very thick,->] (D) -- (C);
    \draw[very thick] (B) -- (i);
    \draw[very thick,->] (i) -- (A);
  \end{tikzpicture}
=\begin{tikzpicture}[baseline=23,scale=.6]
    \node (C) at (1,3) {$m_k$};
    \node (D) at (-1,0) {$m_k$};
    \node (A) at (-1,3)  {$m_{k+1}$};
    \node (B) at (1,0){$m_{k+1}$};
    \node (i) at (intersection of A--B and D--C) {};
    \draw[very thick,->] (i) -- (A);
\draw[line width=2pt,white]    (.85,2.2) to[out=-160,in=20] (-.6,1.9) ;
    \draw[rung>] ({1-.5/1.5},{3-.7}) to[out=0,in=20] (.85,2.2) to[out=-160,in=20] (-.6,1.9)  to[out=-160,in=180] ({-1+1.2/1.5},{3-1.2}); 
    
    \draw[line width=3pt,white] (D)--(C);
    \draw[very thick,->] (D) -- (C);
    \draw[very thick] (B) -- (i);
  \end{tikzpicture} \;.
	\end{equation}
	The case of $a_{k, i}$ and $a_{i, k}$ for $i \neq k, k+1$ follows from the skein relation \crefrange{eq:thick-blue-skein1}{eq:thick-blue-skein2}.  For example, if $i<k$, then for $a_{i,k}$ we calculate:
		\begin{equation}
\begin{tikzpicture}[baseline=23,scale=.6]
    \node (C) at (1,3) {$m_k$};
    \node (D) at (-1,0) {$m_k$};
    \node (A) at (-1,3)  {$m_{k+1}$};
    \node (B) at (1,0){$m_{k+1}$};
    \node (i) at (intersection of A--B and D--C) {};
    \draw[rung>] (-1.5,1) -- (-.33,1); 
    \draw[line width=4pt,white] (D)--(C);
    \draw[very thick,->] (D) -- (C);
    \draw[very thick] (B) -- (i);
    \draw[very thick,->] (i) -- (A);
  \end{tikzpicture}
  = 
		  \begin{tikzpicture}[baseline=23,scale=.6]
    \node (C) at (1,3) {$m_k$};
    \node (D) at (-1,0) {$m_k$};
    \node (A) at (-1,3)  {$m_{k+1}$};
    \node (B) at (1,0){$m_{k+1}$};
    \node (i) at (intersection of A--B and D--C) {};

    \draw[very thick,->] (i) -- (A);
    \draw[line width=4pt,white] (-1.5,2) -- (.33,2);
    \draw[rung>] (-1.5,2) -- (.33,2); 
    \draw[very thick,->] (D) -- (C);
    \draw[very thick] (B) -- (i);
\end{tikzpicture}   = 
		  \begin{tikzpicture}[baseline=23,scale=.6]
    \node (C) at (1,3) {$m_k$};
    \node (D) at (-1,0) {$m_k$};
    \node (A) at (-1,3)  {$m_{k+1}$};
    \node (B) at (1,0){$m_{k+1}$};
    \node (i) at (intersection of A--B and D--C) {};
    \draw[rung>] (-1.5,2) -- (.33,2);
       \draw[line width=4pt,white] (i) -- (A); 
    \draw[very thick,->] (D) -- (C);
    \draw[very thick] (B) -- (i);
    \draw[very thick,->] (i) -- (A);
  \end{tikzpicture}   -
  		  \begin{tikzpicture}[baseline=23,scale=.6]
    \node (C) at (1,3) {$m_k$};
    \node (D) at (-1,0) {$m_k$};
    \node (A) at (-1,3)  {$m_{k+1}$};
    \node (B) at (1,0){$m_{k+1}$};
    \node (i) at (intersection of A--B and D--C) {};
    \draw[rung>] (-1.5,2) -- (-.33,2);
    \draw[rung<] ({1-.6/1.5},{3-.6}) -- ({-1+.6/1.5},{3-.6});
    \draw[very thick,->] (D) -- (C);
    \draw[very thick] (B) -- (i);
    \draw[very thick,->] (i) -- (A);
  \end{tikzpicture} \;.  
\end{equation}
	The other cases are similar.
\end{proof}
\begin{remark}\label{rem:two-halves}
    It is useful to think about this computation as follows: we can interpret $\gamma_{i*}$ as a light strand going from the $i$-th strand to the far right under all other strands, and $\gamma_{*i}$ as this diagram with the light strand going in the opposite direction.  We can then view $\aij$ (up to some changes of framing) as connecting these two light strands at the far right of the diagram.  Calculations as in the proof of \cref{lem:Phi-braid} show that the rules for sliding $\gamma_{i*}$ and $\gamma_{*i}$ over the braid diagram are described by left multiplication by $\boldsymbol{\Phi}^L_{\beta}$ and right multiplication by $\boldsymbol{\Phi}^R_{\beta}$ respectively.  This is just the special case of \cref{lem:Phi-braid} where one end of the light strand is to the right of any strands moved by the braid.   In \cite{ngFramedKnot2008,ngTopologicalIntroduction2014}, this is dealt with by adding another heavy upright which only serves a bookkeeping purpose, to give the light strand a fixed place to connect at the far right of the diagram. Thus, in our framework, we can omit this extra strand and just anchor the light strand in space.

    Thus, we can do the computation in \cref{lem:Phi-braid} on the two halves of the braid separately, with left multiplication by $\boldsymbol{\Phi}^L_{\beta}$ describing how the first half of the light strand interacts with the braid, and right multiplication by $\boldsymbol{\Phi}^R_{\beta}$ describing how the second half does.  Together these show that we recover $\phi_{\beta}(\mathbf{A})=\boldsymbol{\Phi}^L_{\beta}\mathbf{A}\boldsymbol{\Phi}^R_{\beta}$.  
\end{remark}

Consider the diagonal matrix $\boldsymbol{\Lambda}=\diag(\nglambda\ngmu^{\wri(\beta)}U^{(n-\wri(\beta)-1)/2},1,\dots, 1)$.
\begin{lemma}[\mbox{\cite[Th. 1.3]{ekholmFiltrationsKnot2013}}]
With these definitions, $\ngab K$ is the quotient of $\freeA$ by the relations:
\begin{align}\label{KCH1}
  \mathbf{A}&=\boldsymbol{\Lambda}\cdot \phi_{\beta}(\mathbf{A}) \cdot \boldsymbol{\Lambda}^{-1},\\  \label{KCH2} \mathbf{\hat{A}}&=\boldsymbol{\Lambda}\cdot \boldsymbol{\Phi}^L_{\beta}\cdot \mathbf{A},\\ \label{KCH3}
  \mathbf{A}&=\mathbf{\hat{A}}\cdot \boldsymbol{\Phi}^R_{\beta}\cdot\boldsymbol{\Lambda}^{-1}.  
\end{align}	
\end{lemma}

Note that \eqref{KCH1} follows from \eqref{KCH2} and \eqref{KCH3}.

\subsection{Comparison}\label{sec:comparison}
\begin{theorem}\label{main-theorem}
Let $K$ be a knot. We have an isomorphism of $\Laur{1}$-modules
\[\ngab K\cong \ourab K.\]
\end{theorem}

The proof strategy is as follows: we construct a map $Z\colon \freeA \to \ourab{K}$ sending each generator $a_{ij}$ to a specific diagram in the braid closure. We then show that this map descends to $\ngab{K}$ by verifying that the relations that define $\ngab{K}$ (which come from the differentials in the knot contact homology dga) are satisfied in $\ourab{K}$. Surjectivity follows from showing that all diagrams in $\ourab{K}$ can be written in terms of the images of the $a_{ij}$. The key technical tool is ``sliding'' light strands through the braid closure and tracking how this affects the diagrams.

\begin{remark}
By \cite[Th. 1.2]{cornwellKCHRepresentations2017},  every component where $\ngmu\neq 1$ of the spectrum of the specialization of $\ngab K$ at $U=1$ is a component of the moduli space of KCH representations of $\pi_1(S^3\setminus K)$ of a given dimension.  It seems to be widely expected, but not proven, that this moduli space is always 1-dimensional.  This, together with \cref{main-theorem}, would imply that $\ourabq{K}$  is holonomic, analogous to the result of \cite{garoufalidisColoredHOMFLYPT2018}.  

The spectrum of $\ngab K$ is called the ``full augmentation variety'' in \cite{diogoAugmentationsAnnuli2025}; a proof that this affine variety is 2-dimensional for all knots (as would be expected from holonomicity) has been announced by Dimitroglou Rizell and Legout, but the full proof has not yet appeared \cite{dimitroglourizellRelativeCalabi2024}.  
\end{remark}
    
This isomorphism is induced by sending $a_{ij}$ to the diagram $\aij$ inserted into the braid closure above the braid and the star that marks the twists required to fix the framing.  We also fix the labeling on thick strands so that it is $m$ at the star.  Thus, if we shift a light strand past the star, we must act by the automorphism $\ourlambda$ to keep labels consistent. This is easiest to remember via the local relation in \cite[(E.13-14)]{gaiottoDbranesPlanar2026}:
\begin{align}\label{eq:lambda-local}
\begin{tikzpicture}[baseline=-3]
\node (C) at (0,1) {};
\node (E) at (.2,0) {*};
\node (D) at (0,-1) {};
\node (A) at (-1,-0.2) {};
\coordinate (i) at (0,-0.2) {};
\draw[rung>] (A) -- (i);
\draw[lower>,very thick] (D) -- (C);
\end{tikzpicture}\quad=&\quad \ourlambda\begin{tikzpicture}[baseline=-3]
\node (C) at (0,1) {};
\node (E) at (.2,0) {*};
\node (D) at (0,-1) {};
\node (A) at (-1,0.2) {};
\coordinate (j) at (0,0.2) {};
\draw[rung>] (A) -- (j);
\draw[mid>,very thick] (D)-- (C);
\end{tikzpicture} \;,
&
\begin{tikzpicture}[baseline=-3]
\node (C) at (0,1) {};
\node (E) at (.2,0) {*};
\node (D) at (0,-1) {};
\node (A) at (-1,-0.2) {};
\coordinate (i) at (0,-0.2) {};
\draw[rung>] (i) -- (A);
\draw[mid>,very thick] (D) -- (C);
\end{tikzpicture}\quad=&\quad \ourlambda^{-1}\begin{tikzpicture}[baseline=-3]
\node (C) at (0,1) {};
\node (E) at (.2,0) {*};\node (D) at (0,-1) {};
\node (A) at (-1,0.2) {};
\coordinate (j) at (0,0.2) {};
\draw[rung>] (j) -- (A);
\draw[mid>,very thick] (D) -- (C);
\end{tikzpicture} \;.
\end{align}

In the diagram below, we show the image of $a_{14}$ for a 4-strand braid:
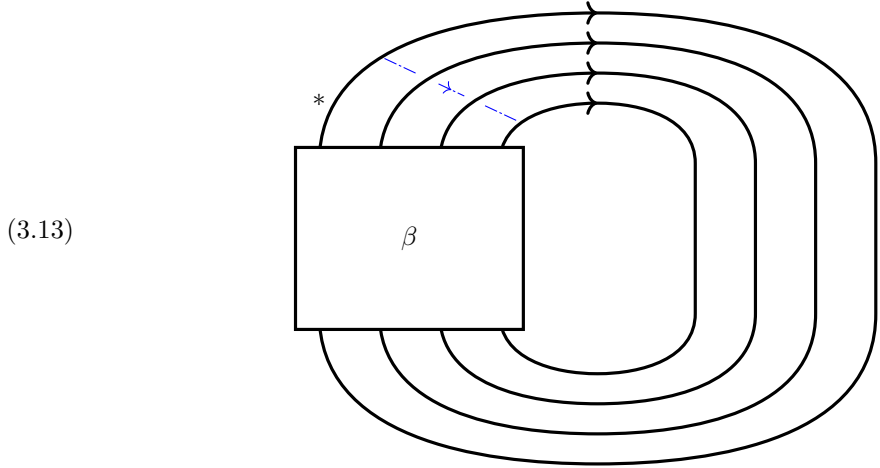
\begin{equation}
	\begin{tikzpicture}[baseline]
	\draw[very thick,lower>] (-1.2,0) -- (-1.2,1) to[out=90,in=180]  node[pos=.4,inner sep=0] (A) {} node[pos=.2, left]{$*$}  (2.5,3)to[out=0,in=90]  (6.2,1)-- (6.2,-1) to[out=-90,in=0] (2.5,-3) to[out=180, in=-90] (-1.2,-1)--cycle;
	\draw[very thick,lower>] (-.4,0) -- (-.4,1) to[out=90,in=180] node[pos=.4] (B) {}(2.5,2.6)to[out=0,in=90]  (5.4,1)-- (5.4,-1) to[out=-90,in=0] (2.5,-2.6) to[out=180, in=-90] (-.4,-1)--cycle;
	\draw[very thick,lower>] (.4,0) -- (.4,1) to[out=90,in=180] node[pos=.4] (C)  {}(2.5,2.2)to[out=0,in=90]  (4.6,1)-- (4.6,-1) to[out=-90,in=0] (2.5,-2.2) to[out=180, in=-90] (.4,-1)--cycle;
		\draw[very thick,lower>] (1.2,0) -- (1.2,1) to[out=90,in=180] node[pos=.4,inner sep=0] (D) {} (2.5,1.8) to[out=0,in=90]  (3.8,1)-- (3.8,-1) to[out=-90,in=0] (2.5,-1.8) to[out=180, in=-90] (1.2,-1)--cycle;
	\node[inner xsep=40pt,inner ysep=30pt,draw,very thick,fill=white] (box) at (0,0){$\beta$};
	\draw[rung>] (A)--(B)--(C)--(D);
	\end{tikzpicture} \;.
\end{equation}
\begin{remark}
	This is the first point where the fact that $K$ is a knot is used: if $K$ is a link and we draw a light strand connecting two different components of $K$, then there is no consistent way of labeling the heavy strands, and thus this does not give a well-defined element of $\ourab K$.  We discuss in \cref{sec:links} how to modify this construction to work for links. 
\end{remark}

Note that in our notation, \[\boldsymbol{\Lambda}=\diag(-\ourlambda^{-1}\ourmu^{-2\wri(\beta)}g^{\wri(\beta)-n},1,\dots, 1)=\diag(\ourlambda^{-1}\ourmu^{-2\wri(\beta)}(-g)^{\wri(\beta)-n},1,\dots, 1);\] since $K$ is a knot, the underlying permutation of $\beta$ is always an $n$-cycle and so $\wri(\beta)-n$ is always odd.

Thus, for each $i$, there is a unique integer $d(i)\in\{0,\dots, n-1\}$ such that $1=\beta^{d(i)}(i)$.  Note that 
\begin{equation}
    d(\beta(i))=\begin{cases}
    d(i)-1,& i\neq 1\\
    n-1,& i=1.
\end{cases}
\end{equation}
Let $k(i,j)=d(j)-d(i)$.\notation{$d(i),k(i,j), D,D_{\beta}$}{Statistics and diagonal matrices that measure how many iterations of $\beta$ send $i$ to $1$.}  Let us rewrite the definitions above in terms that will be more convenient for us.  Let $D=\diag((-g)^{d(1)},\dots, (-g)^{d(n)})$.  Consider the automorphism $\Psi$ of $\freeA$ that sends $\mathbf{A}\to D\mathbf{A}D^{-1}$, that is $a_{ij}\mapsto (-g)^{-k(i,j)}a_{ij}$.  Let $D_{\beta}=\diag((-g)^{d(\beta(1))},\dots, (-g)^{d(\beta(n))})$.  Note that 
\[-g\boldsymbol{\Lambda}D_{\beta}D^{-1}=\boldsymbol{\Lambda}':=\diag(\ourlambda^{-1}\ourmu^{-2\wri(\beta)}(-g)^{\wri(\beta)},1,\dots, 1).\]

\begin{lemma} The automorphism $\Psi$ acts on the matrices $\boldsymbol{\Phi}^L_{\beta}, \boldsymbol{\Phi}^R_{\beta}$ by the formulas
    \[\Psi(\boldsymbol{\Phi}^L_{\beta})=D_{\beta}\boldsymbol{\Phi}^L_{\beta}D^{-1},\qquad \Psi(\boldsymbol{\Phi}^R_{\beta})=D\boldsymbol{\Phi}^R_{\beta}D^{-1}_{\beta}.\]
\end{lemma}
\begin{proof}
One can easily check by induction that the $ij$-entry of $\boldsymbol{\Phi}^L_{\beta}$ is a sum of monomials $a_{\beta(i),j_1}\allowbreak a_{j_1,j_2}\allowbreak \cdots \allowbreak a_{j_{\ell},j}$ for different sequences $j_1,\dots, j_{\ell}$. The automorphism $\Psi$ acts on these by multiplication by $(-g)^{-k(\beta(i),j)}$.  The result follows for $\boldsymbol{\Phi}^L_{\beta}$. The argument for $\boldsymbol{\Phi}^R_{\beta}$ is identical with indices reversed.
\end{proof}
Thus, when we apply $\Psi$ to the equations \crefrange{KCH1}{KCH3}, we obtain the modified relations:
\begin{align}\label{ourCH1}
  \mathbf{A}&=\boldsymbol{\Lambda}'\cdot \phi_{\beta}(\mathbf{A}) \cdot (\boldsymbol{\Lambda}')^{-1},\\  \label{ourCH2} -g\mathbf{\hat{A}}&=\boldsymbol{\Lambda}'\cdot \boldsymbol{\Phi}^L_{\beta}\cdot \mathbf{A},\\ \label{ourCH3}
  \mathbf{A}&=-g\mathbf{\hat{A}}\cdot \boldsymbol{\Phi}^R_{\beta}\cdot(\boldsymbol{\Lambda}')^{-1}.  
\end{align}

We can rewrite the matrices $\mathbf{A}$ and $\mathbf{\hat{A}}$ identifying the formal symbols $a_{ij}$ with diagrams $\aij$ defined in \cref{eq:aij-def}. It will be more convenient for our proof to consider $-g\mathbf{\hat{A}}$, so we write out the entries in our conventions:
\[\mathbf{A}_{ij}=\begin{cases} a_{ij}, & i< j\\ -\ourmu^{-2} a_{ij}, & i>j\\
1-\ourmu^{-2}, & i=j,
\end{cases}\qquad \qquad-g\mathbf{\hat{A}}_{ij}=\begin{cases} -g^{-1}a_{ij}, & i< j\\ g\ourmu^{-2} a_{ij}, & i>j\\
g\ourmu^{-2}-g^{-1}, & i=j.
\end{cases}\]

If we map $a_{ij}\to \aij$, then we will map the entries of $\mathbf{A}_{ij}$ to  diagrams which we can simplify using the framing change relations \crefrange{eq:framingchange1}{eq:framingchange2},
\begin{align}
	\begin{tikzpicture}[baseline=20]
	\laddercoordinates{4}{1}
	\node (A) at (l00) {$m_i$};
	\node (B) at (l01) {$m_i$};
	\node (C) at (l30) {$m_j$};
	\node (D) at (l31) {$m_j$};
	\node (A1) at (l10) {\phantom{$m_i$}};
	\node (B1) at (l11) {\phantom{$m_i$}};
	\node (C1) at (l20) {\phantom{$m_i$}};
	\node (D1) at (l21) {\phantom{$m_i$}};
	\node[inner sep=0pt] (i) at (0,.5 *\ladderY){};
\node[inner sep=0pt] (j) at (3*\ladderX,.5 *\ladderY){};
	\node[inner sep=5pt] (i1) at (intersection of A1--B1 and i--j){};
		\node[inner sep=5pt] (i2) at (intersection of C1--D1 and i--j){};
	\draw[mid>,very thick] (A)--(B);
	\draw[mid>,very thick] (C)--(D); 
		\draw[mid>,very thick] (A1)--(B1);
	\draw[mid>,very thick] (C1)--(D1); 
	\draw[rung>] (i)--(i1);
	\draw[rung>] (i1)--(i2);
		\draw[rung>] (i2)--(j);	
\end{tikzpicture} &= \ourmu^{-1}\begin{tikzpicture}[baseline=20]
	\laddercoordinates{4}{1}
	\node (A) at (l00) {$m_i$};
	\node (B) at (l01) {$m_i$};
	\node (C) at (l30) {$m_j$};
	\node (D) at (l31) {$m_j$};
	\node (A1) at (l10) {\phantom{$m_i$}};
	\node (B1) at (l11) {\phantom{$m_i$}};
	\node (C1) at (l20) {\phantom{$m_i$}};
	\node (D1) at (l21) {\phantom{$m_i$}};
	\node[inner sep=0pt] (i) at (0,.5 *\ladderY){};
\node[inner sep=3pt] (j) at (3*\ladderX,.5 *\ladderY){};
\node[inner sep=0pt] (ii) at (0, .5*\ladderY-.3){};
\node[inner sep=0pt] (jj) at (3*\ladderX,.5 *\ladderY+.3){};
	\node[inner sep=5pt] (i1) at (intersection of A1--B1 and i--j){};
		\node[inner sep=5pt] (i2) at (intersection of C1--D1 and i--j){};
	\draw[mid>,very thick] (A)--(B);
	\draw[mid>,very thick] (C)--(D); 
		\draw[mid>,very thick] (A1)--(B1);
	\draw[mid>,very thick] (C1)--(D1); 
	\draw[rung>]  (i) to (i1);
	\draw[rung] (jj) to[out=0,in=90] (3*\ladderX + .3,.5 *\ladderY+.15) to[out=-90,in=0]  (j);
	\draw[rung>] (i1)--(i2);
		\draw[rung>] (i2)--(j);	
\end{tikzpicture}& i < j,\label{eq:A-matrix1}\\
\ourmu^{-2}\begin{tikzpicture}[baseline=20]
	\laddercoordinates{4}{1}
	\node (A) at (l00) {$m_j$};
	\node (B) at (l01) {$m_j$};
	\node (C) at (l30) {$m_i$};
	\node (D) at (l31) {$m_i$};
	\node (A1) at (l10) {\phantom{$m_i$}};
	\node (B1) at (l11) {\phantom{$m_i$}};
	\node (C1) at (l20) {\phantom{$m_i$}};
	\node (D1) at (l21) {\phantom{$m_i$}};
	\node[inner sep=3pt] (i) at (0,.5 *\ladderY){};
\node[inner sep=3pt] (j) at (3*\ladderX,.5 *\ladderY){};
\node[inner sep=0pt] (ii) at (0, .5*\ladderY-.3){};
\node[inner sep=0pt] (jj) at (3*\ladderX,.5 *\ladderY+.3){};
	\node[inner sep=5pt] (i1) at (intersection of A1--B1 and i--j){};
		\node[inner sep=5pt] (i2) at (intersection of C1--D1 and i--j){};
	\draw[mid>,very thick] (A)--(i) ;
	\draw[mid>,very thick] (i)--(B);
	\draw[mid>,very thick] (C)--(D); 
		\draw[mid>,very thick] (A1)--(B1);
	\draw[mid>,very thick] (C1)--(D1); 
	\draw[rung<] (ii) to[out=180,in=-90] (-.3,.5 *\ladderY-.15) to[out=90,in=180] (0,.5 *\ladderY) to (i1);
	\draw[rung] (jj) to[out=0,in=90] (3*\ladderX + .3,.5 *\ladderY+.15) to[out=-90,in=0]  (j);
	\draw[rung<] (i1)--(i2);
		\draw[rung<] (i2)--(j);	
\end{tikzpicture} &=\ourmu^{-1} \begin{tikzpicture}[baseline=20]
	\laddercoordinates{4}{1}
	\node (A) at (l00) {$m_j$};
	\node (B) at (l01) {$m_j$};
	\node (C) at (l30) {$m_i$};
	\node (D) at (l31) {$m_i$};
	\node (A1) at (l10) {\phantom{$m_i$}};
	\node (B1) at (l11) {\phantom{$m_i$}};
	\node (C1) at (l20) {\phantom{$m_i$}};
	\node (D1) at (l21) {\phantom{$m_i$}};
	\node[inner sep=3pt] (i) at (0,.5 *\ladderY){};
\node[inner sep=3pt] (j) at (3*\ladderX,.5 *\ladderY){};
\node[inner sep=0pt] (ii) at (0, .5*\ladderY-.3){};
\node[inner sep=0pt] (jj) at (3*\ladderX,.5 *\ladderY-.3){};
	\node[inner sep=5pt] (i1) at (intersection of A1--B1 and i--j){};
		\node[inner sep=5pt] (i2) at (intersection of C1--D1 and i--j){};
	\draw[mid>,very thick] (A)--(B);
	\draw[mid>,very thick] (C)--(D); 
		\draw[mid>,very thick] (A1)--(B1);
	\draw[mid>,very thick] (C1)--(D1); 
	\draw[rung<] (0,.5 *\ladderY) to (i1);
	\draw[rung] (jj) to[out=0,in=-90] (3*\ladderX + .3,.5 *\ladderY-.15) to[out=90,in=0]  (j);
	\draw[rung<] (i1)--(i2);
		\draw[rung<] (i2)--(j);	
\end{tikzpicture} & j < i, \label{eq:A-matrix2}
\end{align}
and the entries of $-g\mathbf{\hat A}_{ij}$ to 
\begin{align}
	-g^{-1}\begin{tikzpicture}[baseline=20]
	\laddercoordinates{4}{1}
	\node (A) at (l00) {$m_i$};
	\node (B) at (l01) {$m_i$};
	\node (C) at (l30) {$m_j$};
	\node (D) at (l31) {$m_j$};
	\node (A1) at (l10) {\phantom{$m_i$}};
	\node (B1) at (l11) {\phantom{$m_i$}};
	\node (C1) at (l20) {\phantom{$m_i$}};
	\node (D1) at (l21) {\phantom{$m_i$}};
	\node[inner sep=0pt] (i) at (0,.5 *\ladderY){};
\node[inner sep=0pt] (j) at (3*\ladderX,.5 *\ladderY){};
	\node[inner sep=5pt] (i1) at (intersection of A1--B1 and i--j){};
		\node[inner sep=5pt] (i2) at (intersection of C1--D1 and i--j){};
	\draw[mid>,very thick] (A)--(B);
	\draw[mid>,very thick] (C)--(D); 
		\draw[mid>,very thick] (A1)--(B1);
	\draw[mid>,very thick] (C1)--(D1); 
	\draw[rung>] (i)--(i1);
	\draw[rung>] (i1)--(i2);
		\draw[rung>] (i2)--(j);	
\end{tikzpicture} &= \ourmu^{-1}\begin{tikzpicture}[baseline=20]
	\laddercoordinates{4}{1}
	\node (A) at (l00) {$m_i$};
	\node (B) at (l01) {$m_i$};
	\node (C) at (l30) {$m_j$};
	\node (D) at (l31) {$m_j$};
	\node (A1) at (l10) {\phantom{$m_i$}};
	\node (B1) at (l11) {\phantom{$m_i$}};
	\node (C1) at (l20) {\phantom{$m_i$}};
	\node (D1) at (l21) {\phantom{$m_i$}};
	\node[inner sep=0pt] (i) at (0,.5 *\ladderY){};
\node[inner sep=3pt] (j) at (3*\ladderX,.5 *\ladderY){};
\node[inner sep=0pt] (ii) at (0, .5*\ladderY-.3){};
\node[inner sep=0pt] (jj) at (3*\ladderX,.5 *\ladderY-.3){};
	\node[inner sep=5pt] (i1) at (intersection of A1--B1 and i--j){};
		\node[inner sep=5pt] (i2) at (intersection of C1--D1 and i--j){};
	\draw[mid>,very thick] (A)--(B);
	\draw[mid>,very thick] (C)--(D); 
		\draw[mid>,very thick] (A1)--(B1);
	\draw[mid>,very thick] (C1)--(D1); 
	\draw[rung>]  (i) to (i1);
	\draw[rung] (jj) to[out=0,in=-90] (3*\ladderX + .3,.5 *\ladderY-.15) to[out=90,in=0]  (j);
	\draw[rung>] (i1)--(i2);
		\draw[rung>] (i2)--(j);	
\end{tikzpicture}& i < j,\\
-g\ourmu^{-2}\begin{tikzpicture}[baseline=20]
	\laddercoordinates{4}{1}
	\node (A) at (l00) {$m_j$};
	\node (B) at (l01) {$m_j$};
	\node (C) at (l30) {$m_i$};
	\node (D) at (l31) {$m_i$};
	\node (A1) at (l10) {\phantom{$m_i$}};
	\node (B1) at (l11) {\phantom{$m_i$}};
	\node (C1) at (l20) {\phantom{$m_i$}};
	\node (D1) at (l21) {\phantom{$m_i$}};
	\node[inner sep=3pt] (i) at (0,.5 *\ladderY){};
\node[inner sep=3pt] (j) at (3*\ladderX,.5 *\ladderY){};
\node[inner sep=0pt] (ii) at (0, .5*\ladderY-.3){};
\node[inner sep=0pt] (jj) at (3*\ladderX,.5 *\ladderY+.3){};
	\node[inner sep=5pt] (i1) at (intersection of A1--B1 and i--j){};
		\node[inner sep=5pt] (i2) at (intersection of C1--D1 and i--j){};
	\draw[mid>,very thick] (A)--(i) ;
	\draw[mid>,very thick] (i)--(B);
	\draw[mid>,very thick] (C)--(D); 
		\draw[mid>,very thick] (A1)--(B1);
	\draw[mid>,very thick] (C1)--(D1); 
	\draw[rung<] (ii) to[out=180,in=-90] (-.3,.5 *\ladderY-.15) to[out=90,in=180] (0,.5 *\ladderY) to (i1);
	\draw[rung] (jj) to[out=0,in=90] (3*\ladderX + .3,.5 *\ladderY+.15) to[out=-90,in=0]  (j);
	\draw[rung<] (i1)--(i2);
		\draw[rung<] (i2)--(j);	
\end{tikzpicture} &=\ourmu^{-1} \begin{tikzpicture}[baseline=20]
	\laddercoordinates{4}{1}
	\node (A) at (l00) {$m_j$};
	\node (B) at (l01) {$m_j$};
	\node (C) at (l30) {$m_i$};
	\node (D) at (l31) {$m_i$};
	\node (A1) at (l10) {\phantom{$m_i$}};
	\node (B1) at (l11) {\phantom{$m_i$}};
	\node (C1) at (l20) {\phantom{$m_i$}};
	\node (D1) at (l21) {\phantom{$m_i$}};
	\node[inner sep=3pt] (i) at (0,.5 *\ladderY){};
\node[inner sep=3pt] (j) at (3*\ladderX,.5 *\ladderY){};
\node[inner sep=0pt] (ii) at (0, .5*\ladderY-.3){};
\node[inner sep=0pt] (jj) at (3*\ladderX,.5 *\ladderY+.3){};
	\node[inner sep=5pt] (i1) at (intersection of A1--B1 and i--j){};
		\node[inner sep=5pt] (i2) at (intersection of C1--D1 and i--j){};
	\draw[mid>,very thick] (A)--(B);
	\draw[mid>,very thick] (C)--(D); 
		\draw[mid>,very thick] (A1)--(B1);
	\draw[mid>,very thick] (C1)--(D1); 
	\draw[rung<] (0,.5 *\ladderY) to (i1);
	\draw[rung] (jj) to[out=0,in=90] (3*\ladderX + .3,.5 *\ladderY+.15) to[out=-90,in=0]  (j);
	\draw[rung<] (i1)--(i2);
		\draw[rung<] (i2)--(j);	
\end{tikzpicture} & j < i. 
\end{align}
Note that the RHS of these equations also matches the diagonal entries of the matrices by the $q\to 1$ limit of \cref{eq:bigon2}:
\begin{equation}
	  \begin{tikzpicture}[baseline=30]
    \foreach \n in {0,...,3} {
      \coordinate (z\n) at (0.4*\n, 0.8*\n);
    }
    \draw[mid>,very thick] (z0) -- node[right] {$m$} (z1);
    \draw[mid>,very thick] (z2) -- node[right] {$m$} (z3);
    \draw[mid>,very thick] (z1) to[out=150,in=-190] node[left] {$m-1$} (z2);
    \draw[rung>] (z1) to[out=-30,in=0]    (z2);
  \end{tikzpicture}
  = (\ourmu-\ourmu^{-1})
  \tikz[baseline=30]{\draw[mid>,very thick] (0,0) -- node[right] {$m$} (1,2);} \;,
  \qquad
  \begin{tikzpicture}[baseline=30]
    \foreach \n in {0,...,3} {
      \coordinate (z\n) at (0.4*\n, 0.8*\n);
    }
    \draw[mid>,very thick] (z0) -- node[right] {$m$} (z1);
    \draw[mid>,very thick] (z2) -- node[right] {$m$} (z3);
    \draw[mid>,very thick] (z1) to[out=150,in=-190] node[left] {$m+1$} (z2);
    \draw[rung<] (z1) to[out=-30,in=0]   (z2);
  \end{tikzpicture}
  = (g\ourmu^{-1}-g^{-1}\ourmu)
  \tikz[baseline=30]{\draw[mid>,very thick] (0,0) -- node[right] {$m$} (1,2);} \;.
\end{equation}
Written in this form, we readily see that: 
\begin{lemma}
	We have a surjective map \[Z\colon \ngab K\to  \ourab K.\]  
\end{lemma}
\begin{proof}
	
Whichever order $i$ and $j$ come in, we divide the light strand at its rightmost point into two curves: its {\it first half}, connecting the $i$th heavy strand, and its {\it second half}, connecting the $j$th heavy strand.

To prove that \cref{ourCH3} holds, we start with \crefrange{eq:A-matrix1}{eq:A-matrix2} and slide the second half of the light strand clockwise around the closed braid to obtain a diagram like the one below:
\begin{equation}
\ourmu\mathbf{A}_{ij}=\,	\begin{tikzpicture}[baseline]
	\newcommand{\braidwidth}{1.3}
	\node (D) at (0*\braidwidth,1.6){};
	\draw[very thick,lower>] (-2*\braidwidth,0) -- (-2*\braidwidth,1) to[out=90,in=180]  node[pos=.4,inner sep =0] (A) {} node[pos=.2,left]{$*$}  (0,3)to[out=0,in=90]  (2*\braidwidth,1)-- (2*\braidwidth,-1) to[out=-90,in=0] (0,-3) to[out=180, in=-90] (-2*\braidwidth,-1)--cycle;
	\draw[very thick,lower>] (-1.5*\braidwidth,0) -- (-1.5*\braidwidth,1) to[out=90,in=180] node[pos=.4] (B) {}(0,2.6)to[out=0,in=90]  (1.5*\braidwidth,1)-- (1.5*\braidwidth,-1) to[out=-90,in=0] (0,-2.6) to[out=180, in=-90] (-1.5*\braidwidth,-1)--cycle;
	\draw[very thick,lower>] (-\braidwidth,0) -- (-\braidwidth,1) to[out=90,in=180] node[pos=.4] (C) {}node[pos=.6,inner sep=0] (E) {}(0,2.2)to[out=0,in=90]  (\braidwidth,1)-- (\braidwidth,-1) to[out=-90,in=0] (0,-2.2) to[out=180, in=-90] (-\braidwidth,-1)--cycle;
		\draw[rung>] (A) to node[above left,at start]{$i$}  (B)--(C);
	\draw[rung>] (C) to [in=-125, out=-45] (D) to[out=55,in=-45] node [above left,at end,yshift=-1pt,xshift=1pt]{$j$} (E);
		\draw[white,line width=5pt]		(-.5*\braidwidth,1) to[out=90,in=180] (0,1.8) ;
		\draw[very thick,lower>] (-.5*\braidwidth,0) -- (-.5*\braidwidth,1) to[out=90,in=180] (0,1.8) to[out=0,in=90]  (.5*\braidwidth,1)-- (.5*\braidwidth,-1) to[out=-90,in=0] (0,-1.8) to[out=180, in=-90] (-.5*\braidwidth,-1)--cycle;
	\node[inner xsep=35pt,inner ysep=28pt,draw,very thick,fill=white] (box) at (-1.25*\braidwidth,0){$\beta$};
	\end{tikzpicture}\quad =	\quad\begin{tikzpicture}[baseline]
	\newcommand{\braidwidth}{1.5}
	\node (D) at (0*\braidwidth,.7){};
	\draw[very thick,lower>] (-2*\braidwidth,0) -- (-2*\braidwidth,1) to[out=90,in=180]  node[pos=.4,inner sep =0] (A) {} node[pos=.2,left]{$*$}   (0,3) to[out=0,in=90]  (2*\braidwidth,1)-- (2*\braidwidth,-1) to[out=-90,in=0] (0,-3) to[out=180, in=-90] (-2*\braidwidth,-1)--cycle;
	\draw[very thick,lower>] (-1.5*\braidwidth,0) -- (-1.5*\braidwidth,1) to[out=90,in=180] node[pos=.4] (B) {}(0,2.6)to[out=0,in=90]  (1.5*\braidwidth,1)-- (1.5*\braidwidth,-1) to[out=-90,in=0] (0,-2.6) to[out=180, in=-90] (-1.5*\braidwidth,-1)--cycle;
	\draw[very thick,lower>] (-\braidwidth,0) -- (-\braidwidth,1) to[out=90,in=180] node[pos=.4] (C) {}(0,2.2)to[out=0,in=90]  (\braidwidth,1)-- (\braidwidth,-1) to[out=-90,in=0] (0,-2.2) to[out=180, in=-90] node[pos=.6,inner sep=0] (E) {}(-\braidwidth,-1)--cycle;
		\draw[rung>] (A) to node[above left,at start]{$i$}(B)--(C);
	\draw[rung>] (C) to [in=90, out=-30](D) -- (0*\braidwidth,-.7) to[out=-90,in=-45] node [below left,at end,yshift=1pt,xshift=1pt]{$j$} (E);
		\draw[white,line width=5pt]		(-.5*\braidwidth,1) to[out=90,in=180] (0,1.8) ;
		\draw[white,line width=5pt]	(0,-1.8) to[out=180, in=-90] (-.5*\braidwidth,-1);
		\draw[very thick,lower>] (-.5*\braidwidth,0) -- (-.5*\braidwidth,1) to[out=90,in=180] (0,1.8) to[out=0,in=90]  (.5*\braidwidth,1)-- (.5*\braidwidth,-1) to[out=-90,in=0] (0,-1.8) to[out=180, in=-90] (-.5*\braidwidth,-1)--cycle;
	\node[inner xsep=35pt,inner ysep=28pt,draw,very thick,fill=white] (box) at (-1.25*\braidwidth,0){$\beta$};
	\end{tikzpicture} \;.
	\label{eq:slide-around}
\end{equation}

Now, we slide the second half of the light strand upward through the braid to obtain a sum of diagrams $\gamma_{ik}$ for different values of $k$. As discussed in \cref{rem:two-halves}, since we are only sliding half of the light strand through, this sum is exactly described by the matrix $\boldsymbol{\Phi}^R_{\beta}$, as the sum 
\begin{equation}
\ourmu\mathbf{A}_{ij}=\sum_{k=1}^n (\boldsymbol{\Phi}^R_{\beta})_{kj}\quad 
\begin{tikzpicture}[baseline]
	\newcommand{\braidwidth}{1.5}
	\coordinate (D) at (-.14*\braidwidth,1.5);
	\draw[very thick,lower>] (-2*\braidwidth,0) -- (-2*\braidwidth,1) to[out=90,in=180]  node[pos=.4,inner sep =0] (A) {} node[pos=.2,left]{$*$}   (0,3)to[out=0,in=90]  (2*\braidwidth,1)-- (2*\braidwidth,-1) to[out=-90,in=0] (0,-3) to[out=180, in=-90] (-2*\braidwidth,-1)--cycle;
	\draw[very thick,lower>] (-1.5*\braidwidth,0) -- (-1.5*\braidwidth,1) to[out=90,in=180] node[pos=.5] (B) {}node[pos=.2,inner sep=0] (E) {} (0,2.6) to[out=0,in=90]  (1.5*\braidwidth,1)-- (1.5*\braidwidth,-1) to[out=-90,in=0] (0,-2.6) to[out=180, in=-90] (-1.5*\braidwidth,-1)--cycle;
		\draw[rung>](D) to[out=-120,in=0] node [above left,at end]{$k$}  (E);
	\draw[white,line width=5pt]		(-1*\braidwidth,0)--(-1*\braidwidth,1) to[out=90,in=180] (0,2.1) ;
	\draw[very thick,lower>] (-\braidwidth,0) -- (-\braidwidth,1) to[out=90,in=180] node[pos=.6] (C) {}  (0,2.2)to[out=0,in=90]  (\braidwidth,1)-- (\braidwidth,-1) to[out=-90,in=0] (0,-2.2) to[out=180, in=-90] (-\braidwidth,-1)--cycle;
					\draw[rung>] (A) to node[above left,at start]{$i$}(B)--(C);
	\draw[rung>] (C) to [in=60, out=-30](D);
	\draw[white,line width=5pt]	(0,-1.8) to[out=180, in=-90] (-.5*\braidwidth,-1);
	\draw[white,line width=5pt]		(-.5*\braidwidth,0)--(-.5*\braidwidth,1) to[out=90,in=180] (0,1.8) ;
		\draw[very thick,lower>] (-.5*\braidwidth,0) -- (-.5*\braidwidth,1) to[out=90,in=180] (0,1.8) to[out=0,in=90]  (.5*\braidwidth,1)-- (.5*\braidwidth,-1) to[out=-90,in=0] (0,-1.8) to[out=180, in=-90] (-.5*\braidwidth,-1)--cycle;
	\node[inner xsep=35pt,inner ysep=28pt,draw,very thick,fill=white] (box) at (-1.25*\braidwidth,0){$\beta$};
	\end{tikzpicture} \;.
    \label{eq:after-beta}
\end{equation}
This result is very close to $-g\ourmu\mathbf{\hat{A}}\cdot \boldsymbol{\Phi}^R_{\beta}$, but if $k=1$, we have to drag past the star and $-\wri(\beta)$ twists in the heavy strand. Together this is the same as multiplying by $(\boldsymbol{\Lambda}')^{-1}$ by \cref{eq:lambda-local} and the relation:
\begin{equation}
	\begin{tikzpicture}[baseline]
	\draw[rung>] (-.5,0) to [out=-90,in=180] (0,-.45) to[out=0,in=-90] (.5,-.25) to[out=90,in=0] (0,0);
	\draw[white,line width=5pt] (0,-1) --(0,-.1);

		\draw[very thick,mid>] (0,-1) --(0,0) -- (0,1);
		\draw[white,line width=5pt] (1,0) to [out=180,in=0] (0,.35) to[out=180,in=90] (-.5,0);
		\draw[rung>] (1,0) to [out=180,in=0] (0,.35) to[out=180,in=90] (-.5,0);
	\end{tikzpicture}\,\,= -g^{-1}\ourmu^2 \,\,\,\begin{tikzpicture}[baseline]
		\draw[very thick,mid>] (0,-1) --(0,0) -- (0,1);
		\draw[rung>] (1,0) to (0,0);
	\end{tikzpicture} \;.
\end{equation}
This proves \cref{ourCH3}. 

A similar argument applies to \eqref{ourCH2}, sliding the first half of the light strand rather than the second half.  Since we are sliding the first half through the braid, we get the action of $\boldsymbol{\Phi}^L_{\beta}$ on the left, but otherwise the argument is the same.  The equation \eqref{ourCH1} follows from substituting \eqref{ourCH2} into \eqref{ourCH3}.
\end{proof}

\begin{lemma}\label{lem:Z-isomorphism}
	The map $Z$ is an isomorphism.
\end{lemma}
\begin{proof}
	To prove this, we need only construct an inverse map $Y\colon \ourab K\to \ngab K.$
  By the universal property of tensor products, to construct $Y$ it is enough to construct a $\Spidq{\mathbf{m},\boldsymbol{\epsilon}}{N}$-bilinear map $  \tilde{Y}\colon {}_{T_1} C' \times C'_{T_2}\to \ngab K$ where ${}_{T_1}C'$ is the right module for the top tangle $T_1$ with $n$ nested caps (only maxima), and $C'_{T_2}$ is the left module for the bottom tangle $T_2$ with $n$ nested cups (only minima), with the left ends of $T_2$ braided by $\beta$.  The tangles are shown below with a representative light strand.
	\begin{equation}
	 T_1=\begin{tikzpicture}[baseline=30]
	\newcommand{\braidwidth}{1.3}
			\draw[rung>] (-2*\braidwidth,.5) to node[above left,at start]{$i$}  (
			-.5*\braidwidth,.5);
		\draw[white,line width=5pt]		(-1.5*\braidwidth,1) -- (-1.5*\braidwidth,0) ;
				\draw[white,line width=5pt]		(-1*\braidwidth,1) -- (-1*\braidwidth,0) ;
	\draw[very thick,lower>] (-2*\braidwidth,0) -- (-2*\braidwidth,1) to[out=90,in=180]  node[pos=.4,inner sep =0] (A) {} node[pos=.2,left]{$*$}  (0,3)to[out=0,in=90]  (2*\braidwidth,1)-- (2*\braidwidth,0) ;
	\draw[very thick,lower>] (-1.5*\braidwidth,0) -- (-1.5*\braidwidth,1) to[out=90,in=180] node[pos=.4] (B) {}(0,2.6)to[out=0,in=90]  (1.5*\braidwidth,1)-- (1.5*\braidwidth,0);
	\draw[very thick,lower>] (-\braidwidth,0) -- (-\braidwidth,1) to[out=90,in=180] node[pos=.4] (C) {}node[pos=.6,inner sep=0] (E) {}(0,2.2)to[out=0,in=90]  (\braidwidth,1)-- (\braidwidth,0);
		\draw[very thick,lower>] (-.5*\braidwidth,0) -- (-.5*\braidwidth,1) to[out=90,in=180] (0,1.8) to[out=0,in=90]  (.5*\braidwidth,1)-- (.5*\braidwidth,0);
	\end{tikzpicture}\qquad\qquad  T_2=	\begin{tikzpicture}[baseline=-30]
	\newcommand{\braidwidth}{1.3}
\draw[rung>] (-2*\braidwidth,.5) to node[above left,at start]{$i$}  (
			-.5*\braidwidth,.5);
	\node (D) at (0*\braidwidth,1.6){};
	\draw[very thick,lower>]  (2*\braidwidth,1)-- (2*\braidwidth,-1) to[out=-90,in=0] (0,-3) to[out=180, in=-90] (-2*\braidwidth,-1)-- (-2*\braidwidth,1);
	\draw[very thick,lower>]  (1.5*\braidwidth,1)-- (1.5*\braidwidth,-1) to[out=-90,in=0] (0,-2.6) to[out=180, in=-90] (-1.5*\braidwidth,-1)--(-1.5*\braidwidth,1);
	\draw[very thick,lower>]  (\braidwidth,1)-- (\braidwidth,-1) to[out=-90,in=0] (0,-2.2) to[out=180, in=-90] (-\braidwidth,-1)--(-\braidwidth,1) ;
	\draw[very thick,lower>]  (.5*\braidwidth,1)-- (.5*\braidwidth,-1) to[out=-90,in=0] (0,-1.8) to[out=180, in=-90] (-.5*\braidwidth,-1)--(-.5*\braidwidth,1) ;
	\node[inner xsep=35pt,inner ysep=18pt,draw,very thick,fill=white] (box) at (-1.25*\braidwidth,-.75){$\beta$};
	\end{tikzpicture}
	\end{equation}
	 Since ${}_{T_1}C'$ is a limit of an induction, we can give it a basis by 
	 \begin{enumerate}
	 \item Multiplying on the leftmost $n$ strands by a monomial in $a_{ij}$ for $i\neq j$, as shown in the picture above.
	 	\item Fixing the labels on the left ends of the strands to be $m+p_1, \dots, m+p_n$; the labels on the right ends are determined by the pattern of the light strands.
	 \end{enumerate}
	 Similarly, we can give a basis to $C'_{T_2}$ by similar multiplication and fixing labels.  Let us emphasize that this is done above the braid.  For $a\in \freeA$ and $\mathbf{p}$, let ${}_{\mathbf{p},a}v$ be the corresponding vector in ${}_{T_1}C'$ and $v_{\mathbf{p},a}\in C'_{T_2}$.
	 
	 Thus, we can define the map $\tilde{Y}$ by taking the unique $\Laur{1}$-linear map such that  $\tilde{Y}({}_{\mathbf{p},a}v,{v}_{\mathbf{p'},a'})=0$ if the labels on the bottom of ${}_{\mathbf{p},a}v$ don't match those on top of $ {v}_{\mathbf{p'},a'}$, and $\tilde{Y}({}_{\mathbf{p},a}v,{v}_{\mathbf{p'},a'})$ is the image of $\ourlambda^{-p_1} aa'$ in $\ngab{K}$ otherwise.   
	 
	 Now we only need to confirm $\Spidq{\mathbf{m},\boldsymbol{\epsilon}}{N}$-bilinearity. That is, we have to prove the compatibility \[\tilde{Y}({}_{\mathbf{p},a}v\cdot \gamma_{ij},{v}_{\mathbf{p'},a'})=\tilde{Y}({}_{\mathbf{p},a}v,\gamma_{ij}\cdot {v}_{\mathbf{p'},a'}),\] with multiplication by $\gamma_{ij}$ for $i,j\in \{1,\dots, 2n\}$.
	 \begin{enumerate}
	 	\item If $i,j\leq n$, then this is manifest from the definition.  
   	\item If $i\leq n <j$, then this follows from \cref{ourCH3}.  The diagram $\gamma_{ij}$ is an intermediate step in the proof of this equality; the first half uses equalities in ${}_{T_1} C'$ and the second half equalities in $C'_{T_2}$.
   	\item If $j\leq n <i$, then this follows by the same argument as for \cref{ourCH2}.
   	\item If $n<i,j$, then this follows by the same argument from \cref{ourCH1}.\qedhere
	 \end{enumerate}
   Therefore, $\tilde{Y}$ descends to a well-defined map $Y\colon \ourab K\to \ngab K$. By construction, $Y$ and $Z$ are inverse to each other on the generators, so they are inverse isomorphisms.
\end{proof}
This completes the proof of \cref{main-theorem}.

\subsection{Links}
\label{sec:links}
If, instead of a knot, we consider a link $L=L_1\cup \cdots \cup L_r$ with multiple components, then the results of the paper continue with minor changes.  Let us list here the changes that need to be made to carry this out in the case of links to compare with the knot contact homology defined for links in \cite[Appendix]{ngTopologicalIntroduction2014}. \begin{enumerate}[wide]
	\item The definition of $\ourab{L}$ remains unchanged, but we can now independently choose the label $m_i$ on the different components, which we number in order from left to right in terms of where we encounter their leftmost strand.  Thus, the vector space $\ourab{L}$ is now a module over the ring $\Laur{r}$.  
	\item We now add the correct number of twists to produce the zero-framing of each component on the leftmost strand lying in that component, just above the braid as before.  
  \item We now define the integer $d(i)$ to be the smallest non-negative integer such that $\beta^{d(i)}(i)$ is the leftmost strand of its component, and we define the automorphism $\Psi$ by
  \[\Psi(a_{ij})=(-g)^{d(i)-d(j)}\frac{\ourmu_i}{\ourmu_j}a_{ij}.\]
	\item With this modification, the relations \crefrange{ourCH1}{ourCH3} hold for the diagonal matrix $\boldsymbol{\Lambda}'$ whose entries are $1$ except for those corresponding to the leftmost strand of each component, which are $\ourlambda_i^{-1}\ourmu_i^{-2\wri(L_i)}(-g)^{\wri(L_i)}$, where $\wri(L_i)$ is the writhe of the diagram of $L_i$ obtained by forgetting all other components. 
	\item The map $Y\colon \ourab{L}\to \ngab{L}$ is well-defined by the same arguments as in \cref{lem:Z-isomorphism}.
\end{enumerate}
Unfortunately, the map $Y$ is not manifestly surjective. For each component $L_i$, we let $k_i$ be the number of light strands beginning on the $i$th component minus the number of strands ending on it. The vector space $\ngab{L}$ is graded by $\Z^r$, where the grading of a diagram is $(k_1,\dots, k_r)$. Since $\sum k_i=0$, this grading is trivial if $r=1$, but for $r>1$, it is not clear to us whether it is trivial. Thus, applying the same arguments again shows that:
\begin{theorem}\label{th:link-iso}
	The map $Y\colon \ourab{L} \to \ngab{L}$ is an isomorphism to the degree $0$ subspace of $\ngab{L}$.
\end{theorem}

\begin{remark}
  It has been noted previously that the spectrum of $\ngab{L}$ can have dimension that is too large for a link; for example, \cite[\S 4.4]{gaoAugmentationsSheaves2021} describes a 4-dimensional component for the 3-component unlink.  It would be interesting to consider whether this degree 0 subspace points the way to a better definition for links that has the expected dimension.
\end{remark}

\subsection{Comparison with the augmentation ideal}\label{sec:augmentation}

\cref{th:link-iso} shows that the annihilator of the class $[L]$ in $\ourab{L}$, viewed as an ideal in $\Laur{r}$, is the augmentation ideal $\mathsf{Au}_L$.  However, it is not clear that this equals the ideal $I(L)$ obtained by base-changing $\Iq{L}$; we cannot rule out the possibility that $\mathsf{Au}_L$ is larger.
It is helpful to study this via the exact sequence of $\C[q^{\pm 1}]$-modules
\[0\to \qt{r}/\Iq{L}\to \ourabq{L}\to \ourabq{L}/\qt{r}\cdot [L]\to 0.\]
The universal coefficient theorem then gives a six-term exact sequence:
\begin{multline*}
    0\to\Tor_{q-1}(\qt{r}/\Iq{L})\to\Tor_{q-1}(\ourabq{L})\to \Tor_{q-1}(\ourabq{L}/\qt{r}\cdot [L])\overset{\partial}\to
    \\   
    \Laur{r}/I{(L)}\to \ourab{L}\to \ourab{L}/(\Laur{r}\cdot [L])\to 0.
\end{multline*}
where $\Tor_{q-1}(M)=\{m\in M\mid(q-1)m=0\}$ for any $\C[q^{\pm 1}]$-module $M$.
The quotient $\mathsf{Au}_L/I(L)$ is the image of the connecting homomorphism $\partial$.  We expect that this connecting homomorphism $\partial$ is 0 for all links $L$.  Combined with \cref{conj:ev-injective}, this would imply that:
\begin{conjecture}\label{conj:augmentation-ideal}
We have equalities $\Iq{L}=\Iqp{L}$ and $I(L) = \mathsf{Au}_L$ for all links $L$. That is, the augmentation ideal $\mathsf{Au}_L$ is the classical limit of the ideal defining the recursion relations on antisymmetric HOMFLYPT polynomials.

In particular, the vanishing set of the augmentation ideal $\mathsf{Au}_L$ is Lagrangian in $(\C^{\times})^{2r}$.
\end{conjecture}
In both cases, our ability to prove this conjecture is limited by our poor understanding of the structure of the module $\ourabq{L}$.  If we knew, for example, that this module was free as a $\C[q^{\pm 1},\ourmu_i^{\pm 1}]$-module of known rank, then our task would be considerably simplified.

\bigskip
\IndexOfNotation
\printbibliography

\end{document}